\documentclass[11pt,twoside]{article}
\usepackage{etex}
\usepackage[dvipsnames]{xcolor} 

\usepackage{color}
\usepackage{booktabs}
\usepackage{amsthm}

\usepackage{amsfonts,amssymb,amsmath,amsxtra,url,float} 
\allowdisplaybreaks[4] 
\usepackage[colorlinks,
  linkcolor=magenta, %
  anchorcolor=Periwinkle,
  citecolor=violet,
  urlcolor=blue
  ]{hyperref} 
\usepackage{enumitem}
\usepackage{geometry} 
\usepackage{rotating} 
\usepackage{lscape} 
\usepackage{yhmath} 
\usepackage{multirow}
\usepackage{graphicx} 
\usepackage{subfigure} 
\usepackage{tikz}
\usepackage{pgfplots}
\usepackage{tikz-3dplot}
\usetikzlibrary{patterns}
\usetikzlibrary{3d,calc}
\usetikzlibrary{decorations.pathreplacing,decorations.markings}
 \tikzset{
  on each segment/.style={
    decorate,
    decoration={
      show path construction,
      moveto code={},
      lineto code={
        \path [#1]
        (\tikzinputsegmentfirst) -- (\tikzinputsegmentlast);
      },
      curveto code={
        \path [#1] (\tikzinputsegmentfirst)
        .. controls
        (\tikzinputsegmentsupporta) and (\tikzinputsegmentsupportb)
        ..
        (\tikzinputsegmentlast);
      },
      closepath code={
        \path [#1]
        (\tikzinputsegmentfirst) -- (\tikzinputsegmentlast);
      },
    },
  },
  mid arrow/.style={postaction={decorate,decoration={
        markings,
        mark=at position 0.6 with {\arrow[#1]{stealth}} 
      }}},
}
\usetikzlibrary{arrows}
\usetikzlibrary{trees}

\usetikzlibrary{matrix}
\usetikzlibrary{patterns}
\usetikzlibrary{shadings} 
\usepackage{fancyhdr} 
\def\headertitle{Semi-orthogonal and derived decompositions for gentle algebras}
\def\fstpage{1} 
\def\page{$\begin{matrix} {\color{white}0} \\ \thepage \end{matrix}$} 

\usepackage[all]{xy} 
\usepackage{dsfont} 
\usepackage{cite}
\usepackage{mathrsfs} 
\numberwithin{figure}{section}
\usepackage{marginnote} 
\usepackage{graphicx} 
\usepackage{multicol} 

\usepackage{enumitem}
\setenumerate[1]{itemsep=0pt,partopsep=0pt,parsep=\parskip,topsep=3pt}
\setitemize[1]{itemsep=0pt,partopsep=0pt,parsep=\parskip,topsep=3pt}
\setdescription{itemsep=0pt,partopsep=0pt,parsep=\parskip,topsep=3pt}
\setlist[itemize]{leftmargin=35pt}
\setlist[enumerate]{leftmargin=35pt}
\usepackage{changepage} 
\newcommand{\checks}[1]{{\color{black}{#1}}} 

\newtheorem{theorem}{Theorem}[section]
\newtheorem{lemma}[theorem]{Lemma}
\newtheorem{corollary}[theorem]{Corollary}
\newtheorem{main theorem}[theorem]{Main Theorem}
\newtheorem{proposition}[theorem]{Proposition}
\newtheorem{definition}[theorem]{Definition}
\newtheorem{construction}[theorem]{Construction}
\newtheorem{remark}[theorem]{Remark}
\newtheorem{example}[theorem]{Example}

\newtheorem{question}[theorem]{Question}

\usetikzlibrary{arrows}

\numberwithin{equation}{section}

\def\orcid{
\begin{tikzpicture}[baseline=-1mm]
\filldraw[Green!35] (0,0) circle (5pt);
\filldraw[white] (0,0) node{\tiny\textbf{iD}};
\end{tikzpicture}
}
\def\orcid{
\begin{tikzpicture}[baseline=-1mm]
\filldraw[Green!35] (0,0) circle (5pt);
\filldraw[white] (0,0) node{\tiny\textbf{iD}};
\end{tikzpicture}
}
\newcommand{\ORCID}[1]{ORCID: \href{https://orcid.org/#1}{#1}}
\newcommand{\ORCIDNOTATION}[1]{\href{https://orcid.org/#1}{\orcid}}
\def\EnglishTitle{Semi-orthogonal and derived decompositions for gentle algebras}
\def\EnglishFundings{
Jiangsheng Hu is supported by
the National Natural Science Foundation of China (Grant No. 12571035);
Yu-Zhe Liu is supported by
the National Natural Science Foundation of China (Grant Nos. 12401042 and 12561008),
the Science and Technology Foundation of the Guizhou S\&T Department (Grant Nos. VZD[2026]001, ZD[2025]085 and ZK[2024]YiBan066),
and Scientific Research Foundation of Guizhou University (Grant No. [2023]16);
Tiwei Zhao is supported by
the National Natural Science Foundation of China (Grant No. 12471036) and
Hubei Provincial Natural Science Foundation of China (Grant No. 2026AFA094)
}
\def\FirstAuthorORICD{
0000-0001-7510-9754}
\def\SecondAuthorORICD{
0009-0005-1110-386X}
\def\ThirdAuthorORICD{
0000-0001-8121-4243}

\def\PaperAuthorsENname{
Jiangsheng Hu
$^{\ref{Author1}, \ORCIDNOTATION{\FirstAuthorORICD}\ref{orcid1}}$,
Yu-Zhe Liu
$^{\ref{Author2}, \ORCIDNOTATION{\SecondAuthorORICD}\ref{orcid2},~\ref{CorrespondingAuthor}}$,
Tiwei Zhao
$^{\ref{Author3}, \ORCIDNOTATION{\ThirdAuthorORICD}\ref{orcid3}}$
}
\def\FirstEnOrgani{School of Mathematics, Hangzhou Normal University, Hangzhou 311121, Zhejiang, P. R. China}
\def\SecondEnOrgani{School of Mathematics and Statistics, Guizhou University, Guiyang 550025, Guizhou, P. R. China}
\def\ThirdEnOrgani{School of Artificial Intelligence, Jianghan University, Wuhan 430056, Hubei, P. R. China}

\def\FirstEmail{\url{hujs@hznu.edu.cn} (J. Hu)}
\def\SecondEmail{\url{yzliu3@163.com}/\url{liuyz@gzu.edu.cn} (Y.-Z. Liu)}
\def\ThirdEmail{\url{tiweizhao@jhun.edu.cn} (T. Zhao)}

\def\NN{\mathbb{N}} 
\def\ZZ{\mathbb{Z}} 

\newcommand{\Ima}{\operatorname{Im}}
\newcommand{\Pic}{Figure\ }
\newcommand{\modcat}{\mathsf{mod}}

\newcommand{\add}{\mathsf{add}}
\newcommand{\proj}{\mathsf{proj}}

\newcommand{\Dcat}{\mathsf{D}}
\newcommand{\ind}{\mathsf{ind}}

\def\kk{\Bbbk} 
\def\Q{\mathcal{Q}} 
\def\I{\mathcal{I}}

\def\compos{\ \lower-0.2ex\hbox{\tikz\draw (0pt, 0pt) circle (.1em);} \ }

\newcommand{\per}{\mathsf{per}} 
\newcommand{\Hom}{\mathrm{Hom}} %
\newcommand{\End}{\mathrm{End}} %
\newcommand{\Ext}{\mathrm{Ext}} %

\renewcommand{\H}{\mathrm{H}} %

\newcommand{\To}[1]{\mathop{-\!\!\!-\!\!\!\longrightarrow}\limits^{#1}}

\newcommand{\defines}[1]{{\it\color{violet}#1}}

\title{\bf \EnglishTitle$^{\color{red}\dag}$
\footnotetext[2]{ \tiny \EnglishFundings}
}

\vspace{5mm}

\author{\PaperAuthorsENname}
\date{ }

\begin{document}



\thispagestyle{empty}

\maketitle

\begin{enumerate}[label=\textbf{\color{red}$\ddag$}]
  \item \footnotesize
    \begin{center}
      Corresponding author
    \end{center} \label{CorrespondingAuthor}
\end{enumerate}


\begin{enumerate}[leftmargin=6.6cm] \footnotesize
  \item[\orcid]
      \ORCID{\FirstAuthorORICD}
      \label{orcid1} 
  \item[\orcid]
      \ORCID{\SecondAuthorORICD}
      \label{orcid2} 
  \item[\orcid]
      \ORCID{\ThirdAuthorORICD}
      \label{orcid3} 
\end{enumerate}

\vspace{2mm}
\begin{enumerate}[label=\textbf{\color{red}\arabic*}] \footnotesize
  \item
    \begin{center}
      \FirstEnOrgani

      E-mail: \FirstEmail
    \end{center}
    \label{Author1}

  \item
    \begin{center}
      \SecondEnOrgani

      E-mail: \SecondEmail
    \end{center}
    \label{Author2}

  \item
    \begin{center}
      \ThirdEnOrgani

       E-mail: \ThirdEmail
    \end{center}
    \label{Author3}
%
%
%
\end{enumerate}




\vspace{1mm}


\begin{adjustwidth}{1cm}{1cm}
  \noindent \footnotesize
  \textbf{Abstract}:
We study semi-orthogonal decompositions of perfect derived categories of gentle algebras via marked ribbon surfaces. We characterize such decompositions in terms of suitable disjoint union decompositions of full formal arc systems, and relate this description to good cuts of the corresponding surfaces. For gentle algebras, rotations of curves induce fully faithful functors from extension-closed subcategories of module categories to the components of the associated semi-orthogonal decompositions. Under additional Abelian and extension-comparison conditions, these constructions give derived decompositions of the module categories.
\vspace{1mm}

     \label{2020MSC}

\vspace{1mm}

     \label{Keywords}
\end{adjustwidth}

\newpage
\tableofcontents


\newcommand{\spacing}[1]{%
  \renewcommand{\baselinestretch}{#1}%
  \normalsize%
}

\section{Introduction}

\def\la{\langle} 
\def\ra{\rangle} 
\def\lala{\langle\!\langle}
\def\rara{\rangle\!\rangle}
\def\=<{\leqslant}
\def\>={\geqslant}
\def\s{\mathfrak{s}}
\def\t{\mathfrak{t}}
\def\pdim{\mathrm{proj.dim}}
\def\idim{\mathrm{inj.dim}}
\def\gldim{\mathrm{gl.dim}}
\def\fdim{\mathrm{fin.dim}}
\def\Left{\mathrm{L}}
\def\Right{\mathrm{R}}

\def\calA{\mathcal{A}}
\def\calB{\mathcal{B}}
\def\calC{\mathcal{C}}
\def\calD{\mathcal{D}}
\def\calE{\mathcal{E}}
\def\calF{\mathcal{F}}
\def\calG{\mathcal{G}}
\def\calR{\mathcal{R}}
\def\calS{\mathcal{S}}
\def\calT{\mathcal{T}}
\def\calU{\mathcal{U}}
\def\calV{\mathcal{V}}
\def\calW{\mathcal{W}}
\def\calX{\mathcal{X}}
\def\calY{\mathcal{Y}}
\def\calZ{\mathcal{Z}}
\newcommand{\shift}[2]{{#2}[#1]}

\newcommand{\SURF}{\mathbf{S}} 
\newcommand{\Surf}{\mathcal{S}} 
\newcommand{\bSurf}{\partial\mathcal{S}} 
\newcommand{\M}{\mathcal{M}} 
\newcommand{\MM}{\mathfrak{M}} 
\newcommand{\X}{\mathfrak{X}}
  \newcommand{\gbullet}{{\color{blue}\bullet}} 
\newcommand{\rbullet}{{\color{red}\circ}} 
\newcommand{\E}{\mathcal{E}} 
\newcommand{\D}{\Delta} 
  \newcommand{\Dblue}{\Delta_{\color{blue}\bullet}} 
\newcommand{\Dred}{\Delta_{\color{red}\circ}} 
\newcommand{\dualDgreen}{\Delta_{{\color{ForestGreen}\bullet}}^{\star}} 
\newcommand{\tDgreen}{\widetilde{\Delta}_{{\color{ForestGreen}\bullet}}} 
\newcommand{\dualDred}{\Delta_{{\color{red}\circ}}^{\star}} 
\newcommand{\tDred}{\widetilde{\Delta}_{{\color{red}\circ}}} 
\newcommand{\OEP}{\mathrm{OEP}} 
\newcommand{\CEP}{\mathrm{CEP}} 
\newcommand{\PP}{\mathcal{P}} 
\newcommand{\PGD}{\mathrm{PGD}} 
\newcommand{\GD}{\mathrm{GD}} 
\newcommand{\SURFblue}{\mathbf{S}_{\color{blue}\bullet}} 
\newcommand{\SURFextra}{\mathbf{S}^{\mathcal{E}}_{\color{blue}\bullet}} 
\newcommand{\innerSurf}{\mathcal{S}\backslash\partial\mathcal{S}} 
\newcommand{\innerSurfA}{\mathcal{S}_A\backslash\partial\mathcal{S}_A} 
\newcommand{\SURFred}[1]{\mathbf{S}^{#1}_{\color{red}\circ}} 
\newcommand{\htp}{\mathrm{htp}} 
\newcommand{\tc}{\tilde{c}} 
\newcommand{\tvarsig}{\tilde{\varsigma}} 
\newcommand{\Y}{\mathcal{Y}} 
\newcommand{\F}{\mathcal{F}} 
\newcommand{\Int}{\mathrm{Int}} 
\newcommand{\ii}{\mathfrak{i}} 
\newcommand{\tilt}{\mathrm{tilt}} 
\newcommand{\silt}{\mathrm{silt}} 
\newcommand{\PC}{\mathrm{PC}} 
\newcommand{\CC}{\mathrm{CC}} 
\newcommand{\AC}{\mathrm{AC}} 
\newcommand{\Egreen}{\mathfrak{E}_{\gbullet}}
\newcommand{\Ered}{\mathfrak{E}_{\rbullet}}

\def\bsm{\begin{smallmatrix}}
\def\esm{\end{smallmatrix}}
\def\m{\mathfrak{m}}
\def\rota{{\color{blue}\pmb{\circlearrowright}}}
\def\antirota{{\color{red}\pmb{\circlearrowleft}}}
\def\emb{\mathbf{e}}

Gentle algebras \checks{(over algebraically closed field)} were introduced by Assem and Skowro\'nski in the study of algebras derived equivalent to hereditary algebras of type $\widetilde{\mathbb A}$ \cite{AS1987}.
They form a particularly tractable class of finite-dimensional algebras,
while at the same time occurring in a variety of contexts,
including tilting and silting theory \cite[etc]{FGLZ2023,CS2023b,LiuZhou2025},
module categories and representation types \cite[etc]{BCS2021,Pla2019},
classification theory of derived equivalences \cite[etc]{ALP2016,KY2018,OPS2018,Kal2015,LP2020,APS2023},
and homological properties \cite[etc]{LZZpre2023,LGH2024}.
One reason for their prominent role is that both their module categories
and their derived categories admit concrete combinatorial and geometric descriptions.

The surface approach to gentle algebras originates in the topological Fukaya categories of graded marked surfaces developed by Haiden, Katzarkov and Kontsevich \cite{HKK2017}.
Opper, Plamondon and Schroll subsequently constructed a geometric model for the perfect derived category of a gentle algebra \cite{OPS2018}.
In the method given in \cite{OPS2018}, indecomposable objects are represented by graded admissible curves on an associated marked ribbon surface,
while intersections of curves describe morphisms and their smoothings describe mapping cones.
Baur and Coelho Sim\~oes constructed a geometric model for the category of finitely generated modules in terms of permissible curves \cite{BCS2021}.
These models have proved effective in translating homological and representation-theoretic questions into the topology and
combinatorics of marked surfaces.

Semi-orthogonal decompositions are a fundamental tool for studying a triangulated category through smaller triangulated subcategories.
For a triangulated category $\calT$, such a decomposition expresses $\calT$ as
\[ \calT=\langle\calC,\calD\rangle,\]
where $\calC$ and $\calD$ are full triangulated subcategories satisfying $\Hom_{\calT}(\calC,\calD)=0$, and every object of $\calT$ is obtained from an object of $\calC$ and an object of $\calD$ by a distinguished triangle.
Kop\v{r}iva and \v{S}\v{t}ov\'i\v{c}ek studied such decompositions for gentle algebras through the geometric model of \cite{OPS2018}.
They established a one-to-one correspondence between semi-orthogonal decompositions of the perfect, or bounded, derived category of a gentle algebra and good cuts of its marked ribbon surface \cite{KS2022}.

There is also an Abelian counterpart of this problem.
Chen and Xi introduced derived decompositions of an Abelian category in terms of
full Abelian subcategories whose bounded derived categories embed fully faithfully
and form a semi-orthogonal decomposition of the ambient derived category \cite{CX2021}.
Thus, it is natural to ask the following question.
\begin{question}
Given a gentle algebra $A$, can a semi-orthogonal decomposition of the derived category of $A$ be realized by suitable subcategories of the module category of $A$?
\end{question}
The two surface models mentioned above describe the objects of $\modcat(A)$ and $\per(A)$ by different types of curves,
so answering this question requires a geometric transformation between permissible curves and admissible curves.

Such a transformation is provided by the rotations of permissible curves.
Here, each permissible curve $c$ is a special curve on the geometric model of a gentle algebra $A$,
it can be used to describe the indecomposable module $\MM(c)$.
The rotation of a permissible curve $c$ produces a graded admissible curve $\tc^{\rota}$
whose corresponding complex $\X(\tc^{\rota})$ corresponds to the projective resolution of $\MM(c)$.
Such a rotation gives a good description of the canonical embedding from the module category of a gentle algebra to its derived category,
and was initially employed to show the absence of strictly shod algebras on hereditary gentle algebras \cite{ZhangLiu2024}.
Chang also used this rotation method to study the heart of the derived category of gentle algebras \cite{Chang2025}.
The aim of the present paper is to combine this rotation procedure with the geometry of semi-orthogonal decompositions.
In this way, we first describe semi-orthogonal decompositions by decompositions of full formal arc systems, 
and then relate their components to extension-closed subcategories of the module category.
Throughout the paper we concentrate on permissible and admissible curves with endpoints, 
hence on string modules and string complexes, and band modules and band complexes are not considered.
\checks{Indeed, the rotation constructions considered in this paper concern permissible and admissible curves with endpoints, and hence string modules and string complexes. Nevertheless, the statements concerning semi-orthogonal decompositions apply to the whole category $\per(A)$, including band complexes, since the latter are automatically covered by triangulated generation from a full formal arc system.}

Let $A$ be a gentle algebra and let $\SURF$ be its marked ribbon surface.
We use $\modcat(A)$ to represent the finite dimensional module category of $A$,
and use $\Dcat^b(A)$ to represent the bounded derived category of $A$.
Our first main result gives a criterion for the existence of a semi-orthogonal decomposition directly in terms of a disjoint union of $\gbullet$-full formal arc system $\Delta$,
this disjoint union need satisfies two special conditions \ref{OD1} and \ref{OD2}
(see Section \ref{sec:semi-orth of gentle} for details).

\begin{theorem}[Theorem \ref{thm:main 260724}] \label{thm:main 1 260806}
Assume $A$ is a gentle algebra.
Then the derived category $\Dcat^b(A)$ admits a semi-orthogonal decomposition if and only if
$\SURF$ admits a $\gbullet$-full formal arc system $\Delta=D_1\sqcup D_2$
such that the two families satisfy \ref{OD1} and \ref{OD2}.
\end{theorem}

Furthermore, we have the following result which shows that two subcategories obtained by semi-orthogonal decomposition of the perfect derived category of a gentle algebra are respectively perfect derived categories of two gentle algebras (not necessarily connected).

\begin{corollary}[{Corollary \ref{coro:SOD-components-gentle}}]
Let $A$ be a gentle algebra and suppose that $\per(A)$ has a non-trivial semi-orthogonal decomposition $\langle \calX,\calY\rangle$.
\begin{enumerate}[label={\rm(\arabic*)}]
  \item Then there exist gentle algebras $B_{\calX}$ and $B_{\calY}$, which are not necessarily connected,
and triangle equivalences $\calX\simeq \per(B_{\calX})$ and $\calY\simeq \per(B_{\calY})$.
  \item If, in addition, $A$ has a finite global dimension, then $B_{\calX}$ and $B_{\calY}$ have finite global dimension.
Consequently, $\calX\simeq \Dcat^b(B_{\calX})$ and $\calY\simeq \Dcat^b(B_{\calY})$.
\end{enumerate}
\end{corollary}

We next turn to the connection with the module category.
Given a family $\Gamma$ of permissible curves and assuming that $A$ has finite global dimension,
we consider a functor $F_{\Gamma}$ sending each $\MM(c)$ with $c\in\Gamma$ to $\X(\tc^{\rota})$.
Short exact sequences in $\modcat(A)$ become distinguished triangles in $\per(A)$,
which allows this construction to be extended from the generators to their extension closure.
We prove that $F_{\Gamma}$ is fully faithful, see Lemma \ref{lemm:FGamma} and Proposition \ref{prop:FGamma}.
Applying this construction to the inverse rotations of the two parts of a full formal arc system yields our second main result.

\begin{theorem}[{\rm Theorem \ref{thm:main 260725}}] \label{thm:main 2 260806}
Let $A$ be a homologically smooth gentle algebra {\rm(}i.e., $A$ has a finite global dimension{\rm)}.
Suppose that $\Delta=D_1\sqcup D_2$ satisfies the conditions \ref{OD1} and \ref{OD2}.
If every curve of $\Delta$ has an inverse rotation, then we have a semi-orthogonal decomposition
\[
\langle
  \langle
    F_{\Gamma}
    (\calC_1)
  \rangle_{\per(A)},
  \langle
    F_{\Gamma}
    (\calC_2)
  \rangle_{\per(A)}
\rangle,
\]
where $\Gamma=\Delta^{\antirota}:=\{c^{\antirota}: c\in\Delta\}$
{\rm(}$c^{\antirota}$ is the inverse rotation of $c${\rm)};
and for each $i\in\{1,2\}$, $\calC_i:=\add(\MM(D_i))$,
$D_i^{\antirota}:=\{c^{\antirota}: c\in D_i\}$,
and $\langle F_{\Gamma} (\calC_i) \rangle_{\per(A)}$
is the smallest full triangulated subcategory of the perfect derived category $\per(A)$
containing $F_{\Gamma} (\calC_i)$.
\end{theorem}

Furthermore, we have a corollary for the above theorem.

\begin{corollary}[{\rm Corollary \ref{coro:main 260725}}]
Keep the notations from Theorem \ref{thm:main 2 260806}.
If the $\calC_i$ are Abelian subcategories of $\modcat(A)$ and the canonical map
$\Ext^n_{\calC_i}(X,Y)\to \Ext^n_A(X,Y)$ is an isomorphism for all $X,Y\in\calC_i$ and all $n\>= 0$,
then the canonical functors $\Dcat^b(\calC_i)\to\Dcat^b(A)$ are fully faithful.
Under these hypotheses, the preceding semi-orthogonal decomposition induces
the derived decomposition $[\calC_1,\calC_2]$ of $\modcat(A)$.
\end{corollary}

The paper is organized as follows.
In Section \ref{sec:prelim}, we recall gentle algebras, their marked ribbon surfaces,
the geometric models for module and perfect derived categories,
semi-orthogonal decompositions, good cuts, and derived decompositions of Abelian categories.
In Section \ref{sect:rota}, we study rotations and inverse rotations of curves
and characterize when an admissible curve has an inverse rotation.
We then establish the full formal arc system description of semi-orthogonal decompositions, construct the fully faithful functors
$F_{\Gamma}$, see Section \ref{sec:semi-orth of gentle}, and apply them to the derived-decomposition problem in Section \ref{sect:derivedcompos}.

\section{Preliminary} \label{sec:prelim}

\subsection{Gentle algebras and their marked ribbon surfaces}

Recall that a \defines{gentle pair} $(\Q,\I)$ is a bound quiver satisfying the following conditions:
\begin{enumerate}[label=\textrm{(G\arabic*)}]
  \item Any vertex of $\Q$ is the source and target of at most two arrows.

  \item For each arrow $a:x\to y$, there is at most one arrow $b$
  whose source is $y$ such that $ab\in\I$,
  and there is at most one arrow $b$ whose source is $y$
  such that $ab\notin\I$.

  \item For each arrow $a:x\to y$, there is at most one arrow $b$
  whose target is $x$ such that $ba\in\I$,
  and there is at most one arrow $b$ whose target is $x$
  such that $ba\notin\I$.

  \item $\I$ is admissible and it is generated by paths of length $2$.
\end{enumerate}

\begin{definition}[\!{\cite{AS1987}}]
We call that $A=\kk\Q/\I$ is a \defines{gentle algebra} if its bound quiver is gentle.
\end{definition}

\begin{example} \label{examp:gentle260802} \rm
Let $(\Q, \I)$ is the pair given by the quiver
\[\Q = ~ \xymatrix{
  1 \ar[r]^{a_1} & 2 \ar[r]^{a_2} & 3 & 4 \ar[l]_{a_3} \ar[r]^{a_4} & 5
}\]
and the ideal $\I=\langle a_1a_2\rangle$.
Then it is a gentle pair and the algebra $A=\kk\Q/\I$ is a gentle algebra.
\end{example}

Let $\Surf$ be a surface with a nonempty boundary $\bSurf$.
A \defines{marked surface} is a triple $(\Surf, \M, \Y)$, where:
\begin{enumerate}[label=\textrm{(\arabic*)}]
  \item $\M$ and $\Y$ are finite subsets of $\bSurf$;
  \item the elements of $\Y$ fall into two classes:
    those lying on some boundary components of $\Surf$ that contain no element of $M$ and contain exactly one element of $\Y$;
    the remaining elements of $\Y$ alternate with the elements of $M$ on each of the other boundary components.
\end{enumerate}
Elements in $\M$ are called \defines{open marked points} or \defines{$\gbullet$-marked points},
and elements in $\Y$ are called \defines{closed marked points} or \defines{$\rbullet$-marked points}.
A $\gbullet$-curve (resp. a $\rbullet$-curve) is a curve in $\Surf$ whose endpoints are $\gbullet$-marked points (resp. $\rbullet$-marked points) \checks{if its endpoints exist}.
We always assume that arbitrary two curves are representatives in their homotopy classes such that their intersections are minimal.

A \defines{full formal $\gbullet$-arc system} (= $\gbullet$-FFAS, for short) of $(\Surf, \M, \Y)$, say $\Dblue$,
is a set of some $\gbullet$-curves such that:
\begin{enumerate}[label=\textrm{(\arabic*)}]
  \item any two $\gbullet$-curves of $\Dblue$ have no intersection in $\innerSurf$;
  \item every \defines{elementary $\gbullet$-polygon}, the polygon obtained by $\Dblue$ cutting $\Surf$,
    has a unique edge on $\bSurf$
    (note that the elementary $\gbullet$-polygon shown in \Pic \ref{fig:inftypolygon} is called an $\infty$-elementary $\gbullet$-polygon in our paper, and the boundary component $b$ of $\Surf$ without $\gbullet$-marked point is seen as a boundary of this polygon).
\end{enumerate}

\begin{figure}[H]
  \centering
\definecolor{ffqqqq}{rgb}{1,0,0}
\definecolor{bluearc}{rgb}{0,0,1}
\begin{tikzpicture}[scale=0.75] \small
\filldraw[black!20] (0,0) circle (0.3cm);
\draw[line width=1.2pt] (0,0) circle (0.3cm);
\draw[bluearc][line width=1.2pt] (1.73,1) -- (0, 2) -- (-1.73, 1);
\draw[bluearc][line width=1.2pt, dotted] (-1.73, 1) -- (-1.73,-1);
\draw[bluearc][line width=1.2pt] (-1.73,-1) -- (0,-2) -- (1.73,-1) -- (1.73,1);
\fill [bluearc] ( 1.73, 1) circle (2.8pt);
\fill [bluearc] ( 0.00, 2) circle (2.8pt);
\fill [bluearc] (-1.73, 1) circle (2.8pt);
\fill [bluearc] (-1.73,-1) circle (2.8pt);
\fill [bluearc] ( 0.00,-2) circle (2.8pt);
\fill [bluearc] ( 1.73,-1) circle (2.8pt);
\fill [bluearc] ( 1.73, 1) circle (2.8pt);
\draw [red][line width=0.55pt] (0,0.3) -- (-0.89, 1.45);
\draw [red][line width=0.55pt] (0,0.3) -- ( 0.89, 1.45);
\draw [red][line width=0.55pt] (0,0.3) to[out= 40,in=140] ( 1.73, 0.00);
\draw [red][line width=0.55pt] (0,0.3) to[out= 20,in= 70] ( 0.89,-1.45);
\draw [red][line width=0.55pt] (0,0.3) to[out=160,in=110] (-0.89,-1.45);
\draw [red][line width=0.55pt] (0,0.3) to[out=140,in= 40] (-1.73, 0.00) [dotted];
\fill [red] (0,0.3) circle (2.8pt); \fill [white] (0,0.3) circle (2pt);
\draw (0,0) node{$b$};
\end{tikzpicture}
  \caption{$\infty$-elementary $\gbullet$-polygon}
  \label{fig:inftypolygon}
\end{figure}
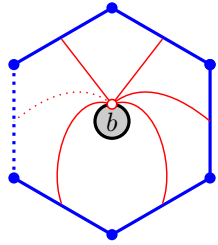

All elements in $\Dblue$ are called \defines{$\Dblue$-arcs}.
For any elementary $\gbullet$-polygon $\PP$, we denote by $\Egreen(\PP)$ the set of all $\Dblue$-arcs.
Similarly, we can define \defines{full formal $\rbullet$-arc system} (= $\rbullet$-FFAS, for short),
\defines{$\Dred$-arcs}, \defines{elementary $\rbullet$-polygon $\PP$}, and $\Ered(\PP)$.
Note that all $\rbullet$-marked points lying in digon are called \defines{extra marked points}
and we denote by $\E$ the set of all extra marked points. Obviously, $\E$ is a subset of $\Y$.

\begin{definition}\rm
A \defines{marked ribbon surface} $\SURF:=(\Surf, \M, \Y, \Dblue, \Dred)$
is a marked surface $(\Surf, \M, \Y)$ with $\gbullet$-FFAS $\Dblue$ and $\rbullet$-FFAS $\Dred$
such that $\Dred$ is the \defines{dual dissection} of $\Dblue$,
that is, for any $\rbullet$-curve $a_{\rbullet}$, there is exactly a unique $\gbullet$-curve $a_{\gbullet}$ intersecting with $a_{\rbullet}$ (in this case, $a_{\gbullet}$ and $a_{\rbullet}$ have only one intersection).
We say $a_{\gbullet}$ (resp., $a_{\rbullet}$) is the \defines{dual arc} of $a_{\rbullet}$ (resp., $a_{\gbullet}$)
and write it as $a_{\rbullet}^{\bot}$ (resp., $a_{\gbullet}^{\bot}$).
\end{definition}

Any marked ribbon surface $\SURF$ defines a graded algebra by the following construction.

\begin{construction} \label{construction} \rm
The algebra $A(\SURF)$ of $\SURF$ is the finite dimensional algebra $\kk\Q/\I$ given by the following steps:
\begin{itemize}
  \item[Step 1]
    there is a bijection $\pmb{V}: \Dblue \to \Q_0$, we write $\Q_0 = \Dblue$ in this paper without causing confusion
    (or equivalently, there is a bijection $\pmb{W}: \Dred \to \Q_0$, and we write $\Q_0 = \Dblue$ without causing confusion);
  \item[Step 2]
    any elementary $\gbullet$-polygon $\PP$ given by $\Dblue$ provides some arrows $\alpha: \pmb{V}(a^1_{\gbullet}) \to \pmb{V}(a^2_{\gbullet})$,
    where $a^1_{\gbullet}, a^2_{\gbullet}\in \Dblue$ are two edges of $\PP$ with common endpoints $p\in \M$
    and $a^2_{\gbullet}$ is left to $a^1_{\gbullet}$ at $p$
    (or equivalently, any elementary $\rbullet$-polygon $\PP'$ given by $\Dred$ provides some arrows $\alpha: \pmb{W}(a_{\rbullet}^1) \to \pmb{W}(a^2_{\rbullet})$,
    where $a_{\rbullet}^1=(a^1_{\gbullet})^{\bot}, a_{\rbullet}^2 = (a^2_{\gbullet})^{\bot}\in \Dred$
    are two edges of $\PP'$ with common endpoints $q\in \Y$,
    and $a^2_{\rbullet}$ is right to $a^1_{\rbullet}$ at $q$), see \Pic \ref{fig:surf-arrow};
\begin{figure}[H]
  \centering
\begin{tikzpicture}
\fill[black!25] (-1,0)--(1,0)--(1,-0.2)--(-1,-0.2)--(-1,0);
\draw (-1,0)--(1,0);
\fill[blue] (0,0) circle(2pt) node[below, black]{$p$};
\draw[blue][line width=1pt] ( 0, 0) -- ( 1.5, 1.5) node[above right, black]{$a^1_{\gbullet}$};
\draw[blue][line width=1pt] ( 0, 0) -- (-1.5, 1.5) node[above  left, black]{$a^2_{\gbullet}$};
\draw[ red][line width=1pt] ( 0, 2) -- (-1.5, 0.5) node[below  left, black]{$(a^2_{\gbullet})^{\bot}=a^2_{\rbullet}$};
\draw[ red][line width=1pt] ( 0, 2) -- ( 1.5, 0.5) node[below right, black]{$a^1_{\rbullet}=(a^1_{\gbullet})^{\bot}$};
\draw[black][line width=1pt][<-] (-0.71,0.71) arc(135:45:1);
\fill[ red ] (0,2) circle (0.1cm);
\fill[white] (0,2) circle (1.8pt);
\draw[ red ] (0,2) node[above]{$q$};
\draw[black] (0,1) node[above]{$\alpha$};
\end{tikzpicture}
  \caption{Two $\Dblue$-arcs $a^2_{\gbullet}$ and $a^1_{\gbullet}$}
\vspace{-3mm}
\begin{center}
  ($a^2_{\gbullet}$ is left to $a^1_{\gbullet}$ at the $\gbullet$-marked point $p$,
  and $a^2_{\rbullet}$ is left to $a^1_{\rbullet}$ at the $\rbullet$-marked point $q$)
\end{center}
  \label{fig:surf-arrow}
\end{figure}
  \item[Step 3]
    the ideal $\I$ is generated by $\alpha\beta$, where $\pmb{V}^{-1}(\s(\alpha)), \pmb{V}^{-1}(\t(\alpha))=\pmb{V}^{-1}(\s(\beta)),\pmb{V}^{-1}(\t(\beta))$
    are edges of the same elementary $\gbullet$-polygon which position are of the form shown in \Pic \ref{fig:surf-relation}.
\begin{figure}[H]
  \centering
\begin{tikzpicture}
\fill[blue]
  ( 2.00, 0.00) circle(2pt)
  ( 1.00, 1.73) circle(2pt)
  (-1.00, 1.73) circle(2pt)
  (-2.00, 0.00) circle(2pt);
\draw[blue][line width=1pt]
  ( 2.00, 0.00) -- ( 1.00, 1.73) -- (-1.00, 1.73) -- (-2.00, 0.00);
\draw
  ( 1.50, 0.86) node[above right]{$\pmb{V}^{-1}(\s(\alpha))$}
  (-1.50, 0.86) node[above left ]{$\pmb{V}^{-1}(\t(\beta))$}
  ( 0.00, 1.73) node[above]{$\pmb{V}^{-1}(\s(\beta))=\pmb{V}^{-1}(\t(\alpha))$};
\draw[line width=1pt][->]
  ( 1.50, 0.86) -- ( 0.10, 1.67);
\draw[line width=1pt][->]
  (-0.10, 1.67) -- (-1.50, 0.86);
\draw ( 0.5, 1.15) node{$\alpha$};
\draw (-0.5, 1.15) node{$\beta$};
\fill[black!25] ( 0.0,-0.50) circle(0.5cm);
\draw[black!99] ( 0.0,-0.50) circle(0.5cm) [line width=1pt][dashed];
\draw[ red ] [line width=1pt] ( 0.0, 0.0) -- ( 0.0, 2.5);
\draw[ red ] [line width=1pt] ( 0.0, 0.0) -- ( 2.0, 0.7);
\draw[ red ] [line width=1pt] ( 0.0, 0.0) -- (-2.0, 0.7);
\fill[ red ] (0,0) circle (0.1cm);
\fill[white] (0,0) circle (1.8pt);
\end{tikzpicture}
  \caption{The path $\alpha\beta \in \I$ given by three $\Dblue$-arcs ($\Dred$-arcs)}
  \label{fig:surf-relation}
\end{figure}
\end{itemize}
\end{construction}

\begin{remark} \label{rmk:FFAS} \rm
Note that every marked surface induces a triangulated category, whose indecomposable objects are $\gbullet$-curves equipped with an ``index'' (this index is called a grading), and whose morphisms are determined jointly by the angles at the intersections of these curves and the ``indices'' on them, see \cite{HKK2017}.
Furthermore, authors showed that each marked surface is equivalent to the derived category of a gentle algebra with finite global dimension.
In the works of Opper, Plamondon, and Schroll in \cite{OPS2018},
authors show that the above result holds for gentle algebra with infinite global dimension (see \cite[Section 1.3]{OPS2018}).
Thus, we have known that every gentle algebra corresponds to a marked surface,
and each marked surface has a $\gbullet$-FFAS and a $\rbullet$-FFAS.
The same result also appeared in the works of Baur and Coelho Sim\~{o}es in \cite{BCS2021}.
\end{remark}

\begin{example} \label{examp:gentle260802-surf} \rm
In this example, we provide the marked ribbon surface of the gentle algebra $A$ given in Example \ref{examp:gentle260802},
see \Pic \ref{fig:gentle260802-surf}.
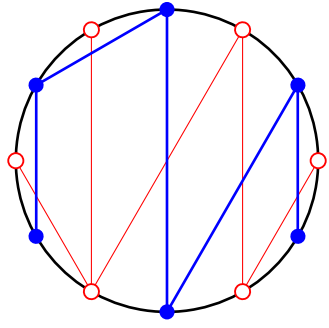
\begin{figure}[H]
  \centering
\begin{tikzpicture}
\draw[red] (-1.00,-1.73) -- (-2.00, 0.00);
\draw[red] (-1.00,-1.73) -- (-1.00, 1.73);
\draw[red] (-1.00,-1.73) -- ( 1.00, 1.73);
\draw[red] ( 1.00, 1.73) -- ( 1.00,-1.73);
\draw[red] ( 1.00,-1.73) -- ( 2.00, 0.00);
\draw[blue][line width=1pt]
  (-1.73,-1.00) -- (-1.73, 1.00) --
  ( 0.00, 2.00) -- ( 0.00,-2.00) --
  ( 1.73, 1.00) -- ( 1.73,-1.00);
\draw[line width=1pt] (0,0) circle(2cm);
\foreach \x in {0,60,120,180,240,300}
\fill[blue][rotate= \x] (0,2) circle(1mm);
\foreach \x in {0,60,120,180,240,300}
\fill[white][rotate= \x] (2,0) circle(1mm);
\foreach \x in {0,60,120,180,240,300}
\draw[red][line width=0.7pt][rotate= \x] (2,0) circle(1mm);
\end{tikzpicture}
  \caption{The marked ribbon surface of the gentle algebra given in Example \ref{examp:gentle260802}}
  \label{fig:gentle260802-surf}
\end{figure}
\end{example}

An \defines{$\Dblue$-arc segment} in $\SURF$ is a homotopy class of segments in some elementary $\gbullet$-polygon which have four cases shown in \Pic \ref{fig:arc segment I},
and an \defines{$\Dred$-arc segment} is a homotopy class of segments in some elementary $\rbullet$-polygon which have two cases shown in \Pic \ref{fig:arc segment II}.

\begin{figure}[H]
\centering
\definecolor{ffqqqq}{rgb}{1,0,0}
\definecolor{bluearc}{rgb}{0,0,1}
\begin{tikzpicture}
\draw[black] (-0.5,2)--( 0.5,2) [line width=1pt];
\draw[bluearc] ( 0, 2)--(-1, 0) [line width=1pt];
\draw[bluearc] ( 0, 2)--( 1, 0) [line width=1pt];
\fill[bluearc] ( 0, 2) circle (0.1cm);
\fill[bluearc] (-1, 0) circle (0.1cm);
\draw[orange][line width=1pt] (-1, 0) -- ( 0.5, 1);
\draw (0,-0.5) node{Case $\gbullet$-A};
\end{tikzpicture}
\ \ \
\begin{tikzpicture}
\draw[black] (-0.5,2)--( 0.5,2) [line width=1pt];
\draw[bluearc] ( 0, 2)--(-1, 0) [line width=1pt];
\draw[bluearc] ( 0, 2)--( 1, 0) [line width=1pt];
\fill[bluearc] ( 0, 2) circle (0.1cm);
\draw[orange][line width=1pt] (-0.5, 1) to[out=-45, in=-135] ( 0.5, 1);
\draw (0,-0.5) node{Case $\gbullet$-B};
\end{tikzpicture}
\ \ \
\begin{tikzpicture}
\draw[black] (0,2) to[out=180,in=90] (-2,0) [line width=1pt];
\draw[bluearc] ( 0, 2) to[out=-90,in=0] (-2, 0) [line width=1pt];
\fill[bluearc] ( 0, 2) circle (0.1cm);
\fill[bluearc] (-2, 0) circle (0.1cm);
\fill[ red ] (-1.41, 1.41) circle (0.10cm) [line width=1pt];
\fill[white] (-1.41, 1.41) circle (0.07cm) [line width=1pt];
\draw[orange] (-1.41, 1.41) -- (-0.57, 0.57) [line width=1pt];
\draw (-1,-0.5) node{Case $\gbullet$-C};
\end{tikzpicture}
\ \ \
\begin{tikzpicture}
\draw[black] (-1,2) -- (1,2) [line width=1pt];
\fill[black!25] (0,1) circle(0.2cm) [line width=1pt];
\draw[black] (0,1) circle(0.2cm) [line width=1pt];
\fill[bluearc] (0,2) circle (0.1cm);
\draw[bluearc][line width=1pt] (0,2) to[out=-135,in=90] (-1,1)
  arc(180:360:1) to[out=90,in=-45] (0,2);
\draw[orange][line width=1pt]
  ( 0.00, 2.00) to[out=-45,in=90] (0.5, 1) arc(0:-270:0.5) -- (1,1.5);
\fill[ red ] (0,1.2) circle (0.10cm) [line width=1pt];
\fill[white] (0,1.2) circle (0.07cm) [line width=1pt];
\draw (0,-0.5) node{Case $\gbullet$-D};
\end{tikzpicture}
\caption{$\Dblue$-arc segments}
\label{fig:arc segment I}
\end{figure}

\begin{figure}[H]
\centering
\definecolor{ffqqqq}{rgb}{1,0,0}
\definecolor{bluearc}{rgb}{0,0,1}
\begin{tikzpicture} [scale=1.01]
\draw[black] (-0.5*1.5, -0.86*1.5) -- ( 0.5*1.5, -0.86*1.5) [line width=1pt];
\draw[red] ( 0.5*1.5,-0.86*1.5) -- ( 1*1.5, 0) [line width=1pt];
\draw[red] ( 1*1.5, 0) -- ( 0.5*1.5, 0.86*1.5) [line width=1pt][dotted];
\draw[red] (-0.5*1.5, 0.86*1.5) -- ( 0.5*1.5, 0.86*1.5) [line width=1pt];
\draw[red] (-1*1.5, 0) -- (-0.5*1.5, 0.86*1.5) [line width=1pt][dotted];
\draw[red] (-0.5*1.5,-0.86*1.5) -- (-1*1.5, 0) [line width=1pt];
\fill[red] ( 0.5*1.5,-0.86*1.5) circle (0.1cm);
\fill[white] ( 0.5*1.5, -0.86*1.5) circle (1.8pt);
\fill[red] ( 1*1.5, 0) circle (0.1cm);
\fill[white] ( 1*1.5, 0) circle (1.8pt);
\fill[red] ( 0.5*1.5, 0.86*1.5) circle (0.1cm);
\fill[white] ( 0.5*1.5, 0.86*1.5) circle (1.8pt);
\fill[red] (-0.5*1.5, -0.86*1.5) circle (0.1cm);
\fill[white] (-0.5*1.5, -0.86*1.5) circle (1.8pt);
\fill[red] (-1*1.5, 0) circle (0.1cm);
\fill[white] (-1*1.5, 0) circle (1.8pt);
\fill[red] (-0.5*1.5, 0.86*1.5) circle (0.1cm);
\fill[white] (-0.5*1.5, 0.86*1.5) circle (1.8pt);
\fill[bluearc] (0,-0.86*1.5) circle (0.1cm) [line width=1pt];
\draw[violet] (0,-0.86*1.5) -- (0, 0.86*1.5) [line width=1pt];
\draw (0,-1.8) node{Case $\rbullet$-A};
\end{tikzpicture}
\ \ \
\begin{tikzpicture} [scale=1.01]
\draw[red] ( 0.5*1.5,-0.86*1.5) -- ( 1*1.5, 0) [line width=1pt][dotted];
\draw[red] ( 1*1.5, 0) -- ( 0.5*1.5, 0.86*1.5) [line width=1pt];
\draw[red] (-0.5*1.5, 0.86*1.5) -- ( 0.5*1.5, 0.86*1.5) [line width=1pt][dotted];
\draw[red] (-1*1.5, 0) -- (-0.5*1.5, 0.86*1.5) [line width=1pt];
\draw[red] (-0.5*1.5,-0.86*1.5) -- (-1*1.5, 0) [line width=1pt][dotted];
\fill[red] ( 0.5*1.5,-0.86*1.5) circle (0.1cm);
\fill[white] ( 0.5*1.5, -0.86*1.5) circle (1.8pt);
\fill[red] ( 1*1.5, 0) circle (0.1cm);
\fill[white] ( 1*1.5, 0) circle (1.8pt);
\fill[red] ( 0.5*1.5, 0.86*1.5) circle (0.1cm);
\fill[white] ( 0.5*1.5, 0.86*1.5) circle (1.8pt);
\fill[red] (-0.5*1.5, -0.86*1.5) circle (0.1cm);
\fill[white] (-0.5*1.5, -0.86*1.5) circle (1.8pt);
\fill[red] (-1*1.5, 0) circle (0.1cm);
\fill[white] (-1*1.5, 0) circle (1.8pt);
\fill[red] (-0.5*1.5, 0.86*1.5) circle (0.1cm);
\fill[white] (-0.5*1.5, 0.86*1.5) circle (1.8pt);
\draw[violet] (-1.15, 0.86*0.75) -- ( 1.15, 0.86*0.75) [line width=1pt];
\draw (0,-1.8) node{Case $\rbullet$-B};
\end{tikzpicture}
\caption{$\Dred$-arc segments}
\label{fig:arc segment II}
\end{figure}

\begin{definition} \label{def:perm and adm curve}
\rm Let $\SURF$ be a marked ribbon surface and $c$ be a curve in $\SURF$.
\begin{enumerate}[label=\textrm{(\arabic*)}]
\item[(1)]
  We call that a curve $c: [0,1] \to \Surf$ is a \defines{permissible curve} (see \Pic \ref{fig:perm curve}) if
  it is a sequence of $\Dblue$-arc segments $\{c_{(i,i+1)}\}_{0\=< i\=< m(c)}$ ($m(c)\in\NN$) such that
  the following conditions hold:
  \label{def:perm curv}
  \begin{itemize}
  \item[(1.1)] $c(t_i)=c_{(i,i+1)}(t_i)$ and $c(t_{m(c)+1})=c_{(m(c), m(c)+1)}(t_{m(c)+1})$, where $0=t_0 <t_1 <\cdots <t_{m(c)} < t_{m(c)+1} = 1$ (note that $c(0)=c(1)\in\Surf\backslash\bSurf$ is allowed to occur);
  \item[(1.2)] two adjacent $\Dblue$-arc segments $c_{(i,i+1)}$ and $c_{(i+1,i+2)}$ are different,
      that is, $c_{(i,i+1)}\ne c_{(i+1,i+2)}$ and $c_{(i,i+1)}\ne c_{(i+1,i+2)}^{-1}$.
  \end{itemize}
  Moreover, we say that any curve $c: [0,1] \to \Surf$ with $c(0), c(1) \in \M\cap\E$
  crossing no $\gbullet$-arc is a \defines{trivial permissible curve}.

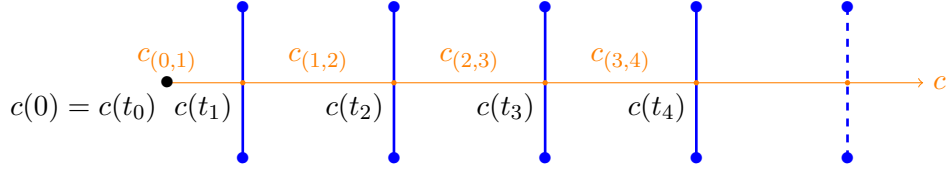
\begin{figure}[H]
  \centering
\begin{tikzpicture}
\foreach \x in {-4,-2,0,2}
\draw[shift={(\x,0)}][blue][line width=1pt] (0,-1) node{$\gbullet$} -- (0,1) node{$\gbullet$};
\draw[shift={( 4,0)}][blue][line width=1pt] (0,-1) node{$\gbullet$} -- (0,1) node{$\gbullet$} [dashed];
\draw[orange][->] (-5,0)--( 5,0); \draw[orange] (5,0) node[right]{$c$}; \draw (-5,0) node{$\bullet$};
\foreach \x in {-4,-2,0,2,4}
\fill[orange] (\x,0) circle(1pt);
\draw[orange] (-5,0) node[above]{$c_{(0,1)}$};
\draw[orange] (-3,0) node[above]{$c_{(1,2)}$};
\draw[orange] (-1,0) node[above]{$c_{(2,3)}$};
\draw[orange] ( 1,0) node[above]{$c_{(3,4)}$};
\draw (-5,0) node[below left]{$c(0)=c(t_0)$};
\draw (-4,0) node[below left]{$c(t_1)$};
\draw (-2,0) node[below left]{$c(t_2)$};
\draw ( 0,0) node[below left]{$c(t_3)$};
\draw ( 2,0) node[below left]{$c(t_4)$};
\end{tikzpicture}
  \caption{Permissible curve}
  \label{fig:perm curve}
\end{figure}

\item[(2a)]
  We call that a $\gbullet$-curve $c: [0,1] \to \Surf$ is a \defines{finite admissible curve} (see \Pic \ref{fig:adm curve}) if
  it is a sequence of $\Dred$-arc segments $\{c_{[i,i+1]}\}_{0\=< i\=< n(c)}$ ($n(c)\in\NN$)
  with a grading $\tc$ such that the following conditions hold:
  \label{def:adm curv}
  \begin{itemize}
  \item[(2a.1)] $c(t_i)=c_{[i,i+1]}(t_i)$ and $c(t_{n(c)+1})=c_{[n(c), n(c)+1]}(t_{n(c)+1})$,
      where $0=t_0 <t_1 <\cdots <t_{n(c)} < t_{n(c)+1} = 1$
      (note that $c(0)=c(1)\in\Surf\backslash\bSurf$ is allowed to occur);
  \item[(2a.2)] two adjacent $\Dred$-arc segments $c_{[i,i+1]}$ and $c_{[i+1,i+2]}$ are different;
  \item[(2a.3)] each $\Dred$-arc $c_{[i,i+1]}$ of $c$ lies on an elementary $\rbullet$-polygon $\PP_i$
    ($\PP_i$ has only one edge $\gamma_i$ on the boundary of $\Surf$,
    there is a unique marked point $p_i$ on $\gamma_i$),
    and $c$ has a natural \defines{grading}
    \[\tc: \{c(t_i) : i\} \to \ZZ\]
    given by
\[\tc_{i}=\tc(c(t_i)) := \begin{cases}
  \tc_{i-1} +1, & \text{if $p_i$ is on the left of $c_{[i,i+1]}(t_i)$}; \\
  \tc_{i-1} -1, & \text{if $p_i$ is on the right of $c_{[i,i+1]}(t_i)$}.
  \end{cases}\]
Cf. \cite{OPS2018}. \checks{We note that, in the complex associated with a projective resolution of a module, we will assume that $\min_{i} \tc_i = 0$ in this paper.}
    
  \item[(2a.4)] if $c(0)=c(1)$, then $\tc_0=\tc_{n(c)+1}$ lies in the inner $\innerSurf$ of $\Surf$.
  \end{itemize}
\begin{figure}[H]
  \centering
\begin{tikzpicture}
\foreach \x in {-4,-2,0,2,4}
\draw[shift={(\x,0)}][red][line width=1pt] (0,-1) -- (0,1);
\draw[violet][->] (-5,0)--( 5,0); \draw[violet] (5,0) node[right]{$c$}; \draw[blue] (-5,0) node{$\bullet$};
\draw[violet][->] (-5,0)--( 1,0);
\draw[white][dashed][line width=2pt] (-1.5,0)--(-0.5,0);
\foreach \x in {-4,-2,0,2,4}
\fill[violet] (\x, 0.0) circle(1pt);
\draw[violet] (-5, 0.0) node[above]{$c_{[0,1]}$};
\draw[violet] (-3, 0.0) node[above]{$c_{[1,2]}$};
\draw[violet] ( 1, 0.0) node[above]{$c_{[i,i+1]}$};
\draw[ black] ( 1, 0.5) node[above]{$\PP_i$};
\draw[violet] ( 3, 0.0) node[above]{$c_{[i+1,i+2]}$};
\draw (-5,0) node[below left]{$c(0)=c(t_0)$};
\draw[shift={(-4,0)}][cyan] ( 0.0,-0.2) arc(-90:-180:0.2) [->];
\draw[shift={(-2,0)}][cyan] ( 0.0,-0.2) arc(-90:-180:0.2) [->];
\draw[shift={( 0,0)}][cyan] ( 0.0,-0.2) arc(-90:-180:0.2) [->];
\draw[shift={( 2,0)}][cyan] ( 0.0,-0.2) arc(-90:-180:0.2) [->];
\draw[cyan] (-4.0, 0.0) node[below left]{$\tc_1$};
\draw[cyan] (-2.0, 0.0) node[below left]{$\tc_2$};
\draw[cyan] ( 0.0, 0.0) node[below left]{$\tc_i$};
\draw[cyan] ( 2.0, 0.0) node[below left]{$\tc_{i+1}$};
\draw[cyan] ( 1.5,-1.0) node[below]{\footnotesize{$\tc_{i+1}=\tc_i+1$}};
\fill[black!25] (0.5,2) -- (1.5,2) -- (1.5,2.2) -- (0.5,2.2);
\draw[line width=1pt]       (0.5,2.0) -- (1.5,2.0);
\draw[line width=1pt][ red] (0.5,2.0) -- (0.0,1.0) [dashed];
\draw[line width=1pt][ red] (1.5,2.0) -- (2.0,1.0) [dashed];
\fill[ red ] (0.0,1.0) circle(2pt); \fill[white] (0.0,1.0) circle(1.4pt);
\fill[ red ] (2.0,1.0) circle(2pt); \fill[white] (2.0,1.0) circle(1.4pt);
\fill[ red ] (0.5,2.0) circle(2pt); \fill[white] (0.5,2.0) circle(1.4pt);
\fill[ red ] (1.5,2.0) circle(2pt); \fill[white] (1.5,2.0) circle(1.4pt);
\fill[blue ] (1.0,2.0) circle(2pt); \draw[ blue] (1.0,2.0) node[above]{$p_i$};
\draw[black] (1.3,2.0) node[below]{$\gamma_i$};
\end{tikzpicture}
  \caption{Permissible curve}
  \label{fig:adm curve}
\end{figure}
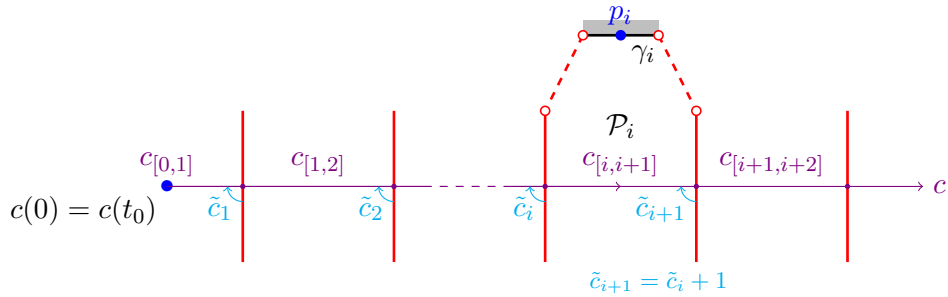
  
  \item[(2b)] We call that a $\gbullet$-curve $c: (0,1]\to \Surf$ is a \defines{left infinite admissible curve} if it is a sequence of $\Dred$-arc segments $\{c_{[i,i+1]}\}_{-\infty\=< i\=< N}$ ($N \in\NN$) such that:
  \begin{itemize}
    \item[(2b.1)] $c(t_i)=c_{[i,i+1]}(t_i)$ and $c(t_{N+1})=c_{[N,N+1]}(t_{N+1})$ where $(0<) \cdots <t_{-1}<t_0 <t_1 <\cdots <t_{N} < t_{N+1} = 1$;
    \item[(2b.2)] the conditions (2.2) and (2.3) holds;
    \item[(2b.3)] if the subcurve $c':(0,t_i]\to \Surf$ surrounds some boundary component $b$ of $\Surf$, 
      then it must be anticlockwise surrounds it, see \Pic \ref{fig:infty adm curve} (1).
\begin{figure}[H]
  \centering
\begin{tikzpicture}
\fill[black!25] (0,0) circle (0.3cm);
\draw[black][line width=1pt] (0,0) circle (0.3cm);
\draw[violet] (-2,1.6) -- (0,1.6) [<-];
\foreach \x in {1.4,1.2,...,0.6}
\draw[violet] (0,\x+0.2) arc(90:-180:\x+0.2) to[out=90,in=180] ( 0.0, \x);
\draw[violet] (0,0.4+0.2) arc(90:-180:0.4+0.2);
\draw (0,-2) node{$(2)$};
\end{tikzpicture}
\ \
\begin{tikzpicture}
\fill[black!25] (0,0) circle (0.3cm);
\draw[black][line width=1pt] (0,0) circle (0.3cm);
\draw[violet] (-2,1.6) -- (0,1.6);
\foreach \x in {1.4,1.2,...,0.6}
\draw[violet] (0,\x+0.2) arc(90:-180:\x+0.2) to[out=90,in=180] ( 0.0, \x);
\draw[violet] (0,0.4+0.2) arc(90:-180:0.4+0.2) [->];
\draw (0,-2) node{$(1)$};
\end{tikzpicture}
\caption{Left and right infinite admissible curve}
\label{fig:infty adm curve}
\end{figure}
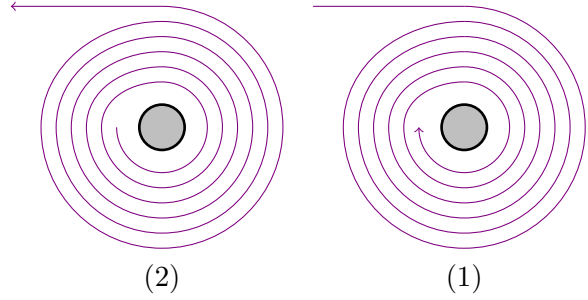
  \end{itemize}
  \item[(2c)] We call that a $\gbullet$-curve $c: [0,1)\to \Surf$ is a \defines{right infinite admissible curve} if it is a sequence of $\Dred$-arc segments $\{c_{[i,i+1]}\}_{-N\=< i\=< \infty}$ ($N \in\NN$) such that:
  \begin{itemize}
    \item[(2c.1)] $c(t_i)=c_{[i,i+1]}(t_i)$ and $c(t_{-N})=c_{[-N,-N+1]}(t_{-N})$ where $0=t_{-N}< \cdots <t_{-1}<t_0 <t_1 <\cdots (<1)$;
    \item[(2c.2)] the conditions (2.2) and (2.3) holds;
    \item[(2c.3)] if the subcurve $c':[t_i,1)\to \Surf$ surrounds some boundary component $b$ of $\Surf$,
      then it must be clockwise surrounds it, see \Pic \ref{fig:infty adm curve} (2).
  \end{itemize}
   
  \item[(2d)] An \defines{infinite admissible curve} is both left infinite and right infinite.
\end{enumerate}
\end{definition}

Finite admissible curves are used to describe string complex in the derived category $\mathds{K}^{-,b}(\proj(A)) \simeq \Dcat^b(A)$ of a gentle algebra $A$, which are bounded in complex category.

An admissible curve is actually a pair $(c,\tc)$.
In many cases, we \defines{use $\tc$ to represent the admissible curve $(c,\tc)$},
which can represent both its corresponding curve $c$ and the grading $\tc$.
A more general definition of graded curve (and its grading) can be referred to \cite{HKK2017,QZZ2022}.
Moreover, for simplicity, we always assume that $c$ and $c^{-1}:=c(1-t)$
with $t\in [0,1]$ are the same permissible/admissible curves,
and we use the following notations which are given by Qiu, Zhang and Zhou in \cite{QZZ2022}.
\begin{itemize}
  \item $\PC_{\m}(\SURF)$: the set of all permissible curves with endpoints lying in $\M\cup\E$;
  \item $\PC_{\oslash}(\SURF)$: the set of all permissible curves without endpoints (up to homotopy);
  \item $\AC_{\m}(\SURF):=\AC^{\m}_{\m}(\SURF)\cup\AC^{\oslash}_{\m}(\SURF)$, where
  \begin{itemize}
    \item $\AC^{\m}_{\m}(\SURF)$: the set of all admissible curves with endpoints lying in $\M$;
    \item $\AC^{\oslash}_{\m}(\SURF)$: the set of all admissible curves with only one endpoint and lies in $\M$;
  \end{itemize}
  \item $\AC^{\oslash}_{\oslash}(\SURF)$: the set of all admissible curves without endpoints (up to homotopy);
  \item $\AC(\SURF) := \AC_{\m}(\SURF) \cup \AC^{\oslash}_{\oslash}(\SURF)$.
\end{itemize}

Let $\SURF = (\Surf, \M, \Y, \Dblue, \Dred)$ be a marked ribbon surface of a gentle algebra $A=\kk\Q/\I$.
The following theorem shows that all indecomposable objects in $\modcat A$ and $\per A$ can be described by permissible curves and admissible curves, respectively.

\begin{theorem} \label{thm:OPS and BCS corresponding}
Let $\mathscr{J}$ be the set of all Jordan blocks with non-zero eigenvalue.
Then there are two bijections:
\begin{itemize}
  \item[\rm(1)] {\rm \cite[Theorems 3.8 and 3.9]{BCS2021}}
    $\MM: \PC_{\m}(\SURF) \cup (\PC_{\oslash}(\SURF)\times\mathscr{J}) \to \ind(\modcat(A))$;
  \item[\rm(2)]{\rm \cite[Theorem 2.13]{OPS2018}}
    $\X: \AC_{\m}(\SURF) \cup (\AC^{\oslash}_{\oslash}(\SURF)\cap\{\text{finite admissible}~\gbullet\text{-arcs}\} \times \mathscr{J}) \to \ind(\per(A))$.
\end{itemize}
\end{theorem}

\subsection{Semi-orthogonal decompositions}

Let $\calT$ be a triangulated category and $\shift{1}{}$ the shift defined on $\calT$.

\begin{definition}\rm
Suppose that $\calC$ and $\calD$ be two full triangulated subcategories of $\calT$.
We call that the pair $\langle\calC, \calD \rangle$ is a \defines{semi-orthogonal decomposition} of $\calT$
if the following conditions hold:
\begin{enumerate}[label=\textrm{(SOD\arabic*)}]
  \item $\Hom_{\calT}(\calC, \calD)=0$; \label{def-OD1}
  \item $\calT = \calC * \calD := \{ Z \in \calT :$ there is a distinguished triangle $C \to Z \to D \to \shift{1}{C}$
  with $C\in\calC$ and $D\in \calD$$\}$. \label{def-OD2}
\end{enumerate}
\end{definition}

Consider a gentle algebra $A$ and its marked ribbon surface $\SURF=(\Surf, \M,\Y,\Dblue,\Dred)$.
Let $\tc_1$ and $\tc_2$ be two admissible curve lying in $\AC_{\m}(\SURF)$ such that
$\emptyset \ne c_1\cap c_2 \in \bSurf$. Then for any $\gbullet$-marked point $p\in c_1\cap c_2$,
we have \[\Hom_{\Dcat^b(A)}(\X(c_1), \shift{t}{\X(c_2)}) \ne 0\]
for some $t\in\ZZ$ if $c_1$ is left to $c_2$ at the $\gbullet$-marked point $p$.
In this case, we have $t = \tc_2(0)-\tc_1(0)$ (without loss of generality, assume $c_1(0)=c_2(0)=p$),
and there exists a morphism $h: \X(c_1) \to \shift{t}{\X(c_2)}$ in this Hom-space corresponding to $p$.
Therefore, if the bounded derived category $\Dcat^b(A)$ of a gentle algebra $A$ has a semi-orthogonal decomposition
$\Dcat^b(A) = \langle \calC, \calD \rangle_{\Dcat^b(A)}$, then the following two facts hold:
\begin{enumerate}[label=\textrm{(D\arabic*)}]
  \item if $\X(\tc_1)\in \calC$, then $\X(\tc_2)\notin \calD$; \label{D1}
  \item if $\X(\tc_2)\in \calD$, then $\X(\tc_1)\notin \calC$. \label{D2}
\end{enumerate}
The mapping cone of $h$ induces a distinguished triangle
\[ \X(c_1) \to \shift{t}{\X(c_2)} \to Z \to \shift{1}{\X(c_1)} \]
in $\Dcat^b(A)$. Furthermore, by using \cite[Theorem 3.3]{OPS2018}, we have
\[ Z \cong \X(\tc), ~\text{for some}~ \tc\in\AC_{\m}(\SURF), \]
where $\tc$ is shown in \Pic \ref{fig:mapping cone}.
\begin{figure}[H]
  \centering
\begin{tikzpicture}
\draw[violet][line width=1pt] (-2,2)--(0,0)--(2,2);
\draw[violet][line width=1pt] (-2,2.1) to[out=-45,in=225] (2,2.1);
\fill[blue] (0,0) circle (3pt) (2,2) circle (3pt) (-2,2) circle (3pt);
\draw[black] (0,0) node[below]{$p$};
\draw (-1,1) node[left]{$\tc_1$} (1,1) node[right]{$\tc_2$} (0,1.3) node[above]{$\tc$};
\end{tikzpicture}
  \caption{The mapping cone of $\X(\tc_1)\to\X(\tc_2)$}
  \label{fig:mapping cone}
\end{figure}
It follows that the following fact:
\begin{enumerate}[label=\textrm{(D3)}]
  \item there is a $\gbullet$-FFAS $\Gamma=\{\tc_i: i\in I\}$ ($I$ is an index set) such that $\X(\Gamma)\subseteq \calC\cup\calD$. \label{D3}
\end{enumerate}
Kop\v{r}iva and \v{S}\v{t}ov\'{i}\v{c}ek show that each
semi-orthogonal decomposition of the bounded derived category $\Dcat^b(A)$
of a gentle algebra $A$ corresponds to a good cut (see Definition \ref{def:good cut} as follows)
of the marked ribbon surface $\SURF$ of $A$
by using the facts \ref{D1}, \ref{D2}, and \ref{D3}, see \cite[Theorems 3.16, 3.17]{KS2022}.

\begin{definition}[\!{\cite[Definition 3.12]{KS2022}}]
\label{def:good cut}\rm
Let $\SURF=(\Surf,\M,\Y,\Dblue,\Dred)$ be a marked ribbon surface and $\Omega$ be a set of some curves.
We call $\Omega$ a \defines{good cut} if it satisfies the following conditions.
\begin{enumerate}[label=\textrm{(GC\arabic*)}]
  \item Each curve $\omega$ in $\Omega$ is a \defines{dividing curve},
    i.e., $\omega$ is a curve from an endpoint lying in $\M$ to another endpoint lying in $\Y$,
    and, fixing the endpoints, $\omega$ not homotopy to a subset without marked point of a boundary component of $\bSurf$,
    see \Pic \ref{fig:div curve}; \label{GC1}
\begin{itemize}
  \item where the endpoint $r_{\omega}\in\Y$ of $\omega$
    split to two $\rbullet$-marked points $r_{\omega}^{\Left}$, say \defines{left added $\rbullet$-marked points},
    and $r_{\omega}^{\Right}$, say \defines{right added $\rbullet$-marked points}, by $\omega$ cutting $\Surf$,
  \item and the endpoint $b_{\omega}\in\M$ of $\omega$
    split to two $\gbullet$-marked points $b_{\omega}^{\Left}$, say \defines{left added $\gbullet$-marked points},
    and $b_{\omega}^{\Right}$, say \defines{right added $\gbullet$-marked points}, by $\omega$ cutting $\Surf$.
\end{itemize}
\begin{figure}[H]
  \centering
\begin{tikzpicture}
\draw[cyan] ( 0.00,-1.00) -- ( 0.00, 1.00);
\draw[cyan] ( 0.00,-1.00) -- ( 0.00, 0.00) [->];
\draw[cyan] ( 0.00, 0.00) node[left]{$\omega$};
\fill[line width=1pt][black!25]
  (-1.50, 1.00) -- ( 1.50, 1.00) -- ( 1.50, 1.20) -- (-1.50, 1.20);
\fill[line width=1pt][black!25]
  (-1.50,-1.00) -- ( 1.50,-1.00) -- ( 1.50,-1.20) -- (-1.50,-1.20);
\draw[line width=1pt][black]
  (-1.50, 1.00) -- ( 1.50, 1.00);
\draw[line width=1pt][black]
  (-1.50,-1.00) -- ( 1.50,-1.00);
\fill[ blue] ( 0.00,-1.00) circle(2.0pt);
\fill[ red ] ( 1.00,-1.00) circle(2.0pt);
\fill[white] ( 1.00,-1.00) circle(1.4pt);
\fill[ red ] (-1.00,-1.00) circle(2.0pt);
\fill[white] (-1.00,-1.00) circle(1.4pt);
\fill[ red ] ( 0.00, 1.00) circle(2.0pt);
\fill[white] ( 0.00, 1.00) circle(1.4pt);
\fill[ blue] (-1.00, 1.00) circle(2.0pt);
\fill[ blue] ( 1.00, 1.00) circle(2.0pt);
\draw ( 0.00, 1.00) node[above]{$r_{\omega}$};
\draw ( 0.00,-1.00) node[below]{$b_{\omega}$};
\end{tikzpicture}
\ \
\begin{tikzpicture}
\draw[white] (0,-1.45)--(0,1.45);
\draw[->][line width=0.7pt] (-0.5,0)--(0.5,0);
\end{tikzpicture}
\ \
\begin{tikzpicture}
\fill[black!25]
  (-1.30, 1.20) -- ( 3.30, 1.20) -- ( 3.30,-1.20) -- (-1.30,-1.20);
\fill[white]
  (-1.35, 1.00) -- (-0.50, 1.00) arc(90:-90:1) -- (-1.35,-1.00);
\draw[line width=1pt][black]
  (-1.30, 1.00) -- (-0.50, 1.00);
\draw[line width=1pt][black]
  (-0.50,-1.00) -- (-1.30,-1.00);
\draw[cyan]
  (-0.50,-1.00) arc(-90:90:1);
\draw[cyan]
  (-0.50,-1.00) arc(-90: 0:1) [->];
\fill[ red ] (-0.50, 1.00) circle(2.0pt);
\fill[white] (-0.50, 1.00) circle(1.4pt);
\fill[ blue] (-1.00, 1.00) circle(2.0pt);
\fill[ blue] (-0.50,-1.00) circle(2.0pt);
\fill[ red ] (-1.00,-1.00) circle(2.0pt);
\fill[white] (-1.00,-1.00) circle(1.4pt);
\draw (-0.50, 1.00) node[above]{$r_{\omega}^{\Left}$};
\draw (-0.50,-1.00) node[below]{$b_{\omega}^{\Left}$};
\fill[white][shift={(2,0)}]
  ( 1.35, 1.00) -- ( 0.50, 1.00) arc(90:270:1) -- ( 1.35,-1.00);
\draw[line width=1pt][black][shift={(2,0)}]
  ( 1.30, 1.00) -- ( 0.50, 1.00);
\draw[line width=1pt][black][shift={(2,0)}]
  ( 0.50,-1.00) -- ( 1.30,-1.00);
\draw[cyan][shift={(2,0)}]
  ( 0.50,-1.00) arc(-90:-270:1);
\draw[cyan][shift={(2,0)}][->]
  ( 0.50,-1.00) arc(-90:-180:1);
\fill[ red ][shift={(2,0)}] ( 0.50, 1.00) circle(2.0pt);
\fill[white][shift={(2,0)}] ( 0.50, 1.00) circle(1.4pt);
\fill[ blue][shift={(2,0)}] ( 1.00, 1.00) circle(2.0pt);
\fill[ blue][shift={(2,0)}] ( 0.50,-1.00) circle(2.0pt);
\fill[ red ][shift={(2,0)}] ( 1.00,-1.00) circle(2.0pt);
\fill[white][shift={(2,0)}] ( 1.00,-1.00) circle(1.4pt);
\draw[shift={(2,0)}] ( 0.50, 1.00) node[above]{$r_{\omega}^{\Right}$};
\draw[shift={(2,0)}] ( 0.50,-1.00) node[below]{$b_{\omega}^{\Right}$};
\end{tikzpicture}
  \caption{Dividing curve}
  \label{fig:div curve}
\end{figure}
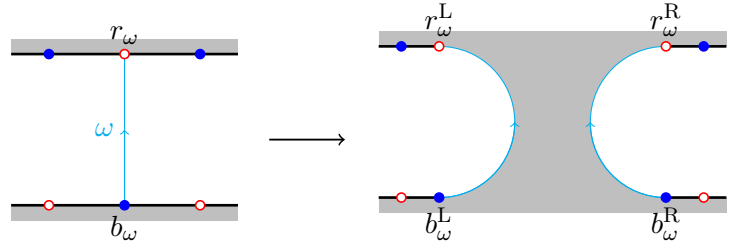

  \item Any two dividing curves neither intersect nor meet at endpoints. \label{GC2}
  \item Let $\SURF_{\Omega}$ be the cut surface obtain $\Omega$ cutting $\SURF$.
    No connected component of $\SURF_{\Omega}$ contains an left added marked point and an right added marked point. \label{GC3}
  \item No connected component of $\SURF_{\Omega}$ is trivial, i.e.,
    homeomorphic to an open disk with only two marked points,
    one $\gbullet$-marked point and one $\rbullet$-marked point, on its boundary. \label{GC4}
\end{enumerate}
\end{definition}

\begin{theorem}[{Kop\v{r}iva--\v{S}\v{t}ov\'{i}\v{c}ek, \cite[Theorems 3.15 and 3.16]{KS2022}}] \label{thm:KS2022}
Let $A$ be a gentle algebra.
\begin{enumerate}[label=\textrm{\rm(\arabic*)}]
  \item There exists a one-to-one correspondence between
    the semi-orthogonal decompositions of the perfect derived category $\per(A)$
    to a good cut of the marked ribbon surface $\SURF$ of $A$.
    \label{KSthm(1)}
  \item The correspondence in \ref{KSthm(1)} can be extent to a one-to-one correspondence between
    the semi-orthogonal decompositions of the bounded derived category $\Dcat^b(A)$
    to a good cut of $\SURF$.
    \label{KSthm(2)}
\end{enumerate}
\end{theorem}

Note that, due to our convention, the notation for semi-orthogonal decompositions used in this paper is opposite to that in \cite{KS2022}.

\subsection{Derived decompositions}

Let $\calA$ be an Abelian category.

\begin{definition} \rm
A \defines{derived decomposition} of an Abelian category $\calA$ is a pair $(\calC, \calD)$ of full Abelian subcategories,
say $[\calC, \calD]$, such that the following conditions hold:
\begin{enumerate}[label=\textrm{(DD\arabic*)}]
  \item the functor $\Dcat^b(\emb_1): \Dcat^b(\calC)\to\Dcat^b(\calA)$ induced by the canonical embedding $\emb_1: \calC \to \calA$ is fully faithful, and the functor $\Dcat^b(\emb_2): \Dcat^b(\calD)\to\Dcat^b(\calA)$ induced by the canonical embedding $\emb_2: \calD \to \calA$ is fully faithful; \label{def-DD1}
  \item $\langle \Ima(\Dcat^b(\emb_1)), \Ima(\Dcat^b(\emb_2)) \rangle$ is a semi-orthogonal decomposition of $\Dcat^b(\calA)$. \label{def-DD2}
\end{enumerate}
\end{definition}

For any gentle algebra $A$, Theorem \ref{thm:KS2022} provides the following one-to-one correspondence
\begin{align*}
 & \{ \text{good cuts of the marked ribbon surface~} \SURF \text{~of~} A\} \\
 \mathop{\longleftrightarrow}^{1-1}~ &
   \{ \text{semi-orthogonal decompositions of~} \Dcat^b(A) \} \\
 \mathop{\longleftrightarrow}^{1-1}~ &
   \{ \text{semi-orthogonal decompositions of~} \per(A) \},
\end{align*}
When $\calA$ is the module category $\modcat(A)$ of a gentle algebra $A$ with finite global dimension,
a pair $(\calC,\calD)$ of full Abelian subcategories of $\modcat(A)$
is a derived decomposition if and only if $\Dcat^b(\emb_1)$ and $\Dcat^b(\emb_2)$ are fully faithful,
and $\langle \Ima(\Dcat^b(\emb_1)), \Ima(\Dcat^b(\emb_2)) \rangle$ is a semi-orthogonal decomposition of $\Dcat^b(A)$.

\section{Rotations}
\label{sect:rota}

\subsection{The rotations of permissible curves}

Recall that a permissible curve $c$ is divided to some $\Dblue$-arc segments
\[c = c_{(0,1)}c_{(1,2)}\cdots c_{(m(c),m(c)+1)}, \]
and the endpoints $c(0)$ and $c(1)$ of $c$ lie in $\M\cup\E$.
Any $\Dblue$-arc segment $c_{(i,i+1)}$ is a curve in an elementary $\gbullet$-polygon,
we write this elementary $\gbullet$-polygon as $\PP_i$.
Obviously, if $c(0)$ (resp., $c(1)$) $\in\E$,
then $\PP_0$ (resp., $\PP_{m(c)}$) is a 2-gon which has two vertices lying in $\M$,
and there is a $\rbullet$-marked point lying in $\E$ on a boundary of it, see \Pic \ref{fig:endsegment}.

\begin{figure}[H]
  \centering
\begin{tikzpicture}[scale=1.3]
\fill[black!25](-1,1.2) arc(90:270:1.2);
\fill[ white]  (-1,1) arc(90:270:1);
\draw[ black][line width=1pt] (-1, 1) arc(90:270:1);
\draw[ blue ][line width=1pt] (-1,-1) -- (-1, 1);
\draw[ blue ][line width=1pt] ( 1,-1) -- ( 1, 1);
\draw[violet][line width=1pt] (-2, 0) -- ( 2, 0)[->];
\draw[ white][line width=2pt] (-.5,0) -- (.6, 0)[dotted];
\foreach \x in {1,-1}
\fill[blue ] (\x,1)circle(2pt) (\x,-1)circle(2pt);
\fill[ red ] (-2,0)circle(2pt);
\fill[white] (-2,0)circle(1.5pt);
\draw[black] (-1.5,-0.4) node{$\PP_0$};
\end{tikzpicture}
  \caption{The first segment of $c$}
  \label{fig:endsegment}
\end{figure}
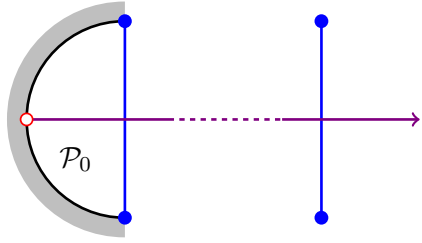

\begin{definition}\label{def:rotat} \rm
Let $A$ be a gentle algebra and $\SURF$ be its marked ribbon surface.
For a permissible curve $c = c_{(0,1)}c_{(1,2)}\cdots c_{m(c),m(c)+1}: [0,1]\to\Surf$,
we define that its \defines{rotation} is a $\gbullet$-curve, $c^{\rota}$, obtained by the following steps.
\begin{itemize}
  \item[Step 1] \textbf{Rotating the first $\Dblue$-arc segment}.

    \hspace{8pt}
    If $\PP_0$, as a elementary $\gbullet$-polygon, is not a 2-gon as shown in \Pic \ref{fig:arc segment I} Case $\gbullet$-C
    (of course, we allow it to be an $\infty$-elementary $\gbullet$-polygon),
    then we move the endpoint $c(0)$ along the edges lying in $\Egreen(\PP_0)$ of $\PP_0$ in a clockwise direction as far as possible, see \Pic \ref{fig:rotating curve} (a), (b);

    \hspace{8pt}
    If $\PP_0$ is a 2-gon as shown in \Pic \ref{fig:arc segment I} Case $\gbullet$-C,
    then we move the endpoint $c(0)$ to the next $\gbullet$-marked point along the positive direction
    (we suppose that the positive direction of the boundary $\partial S$ of a surface/polygon $S$ is the following walking direction: walking along $\partial S$, the inner of $S$ is on the left).

\begin{figure}[H]
\centering
\definecolor{ffqqqq}{rgb}{1,0,0}
\definecolor{bluearc}{rgb}{0,0,1}
\begin{tikzpicture}[scale=0.75] \small
\draw[bluearc][line width=1.2pt] (1.73,1) -- (0, 2) -- (-1.73, 1);
\draw[ black ][line width=1.2pt] (-1.73, 1) -- (-1.73,-1);
\draw[bluearc][line width=1.2pt] (-1.73,-1) -- (0,-2) -- (1.73,-1);
\draw[bluearc][line width=1.2pt] (1.73,-1) -- (1.73,1) [dotted];
\fill [bluearc] ( 1.73, 1) circle (2.8pt);
\fill [bluearc] ( 0.00, 2) circle (2.8pt);
\fill [bluearc] (-1.73, 1) circle (2.8pt);
\fill [bluearc] (-1.73,-1) circle (2.8pt);
\fill [bluearc] ( 0.00,-2) circle (2.8pt);
\fill [bluearc] ( 1.73,-1) circle (2.8pt);
\fill [bluearc] ( 1.73, 1) circle (2.8pt);
\draw [red][line width=0.55pt] (-1.73,0) -- (-0.89, 1.45);
\draw [red][line width=0.55pt] (-1.73,0) -- ( 0.89, 1.45);
\draw [red][line width=0.55pt] (-1.73,0) -- ( 1.73, 0.00) [dotted];
\draw [red][line width=0.55pt] (-1.73,0) -- ( 0.89,-1.45);
\draw [red][line width=0.55pt] (-1.73,0) -- (-0.89,-1.45);
\fill [red] (-1.73,0) circle (2.8pt); \fill [white] (-1.73,0) circle (2pt);
\draw[orange][line width=1.2pt] (0, 2.7) node[left]{$c$} to[out=0, in=165] (1,2.6) [dotted];
\draw[orange][line width=1.2pt] (1, 2.6) to[out=-15, in=180] (1.73,-1);
\draw[violet][line width=1.2pt] (0, 2.3) node[left]{$c^{\rota}$} to[out=0, in=165] (1,2.2) [dotted];
\draw[violet][line width=1.2pt] (1, 2.2)
  to[out=-15, in= 80] ( 1,0) to[out=-110,in=  0] (0,-1) -- (-1.73,-1);
\draw (0,-3) node{(a) elementary $\gbullet$-polygon};
\end{tikzpicture}
\ \
\begin{tikzpicture}[scale=0.75] \small
\filldraw[black!20] (0,0) circle (0.3cm);
\draw[line width=1.2pt] (0,0) circle (0.3cm);
\draw[bluearc][line width=1.2pt] (1.73,1) -- (0, 2) -- (-1.73, 1);
\draw[bluearc][line width=1.2pt, dotted] (-1.73, 1) -- (-1.73,-1);
\draw[bluearc][line width=1.2pt] (-1.73,-1) -- (0,-2) -- (1.73,-1) -- (1.73,1);
\fill [bluearc] ( 1.73, 1) circle (2.8pt);
\fill [bluearc] ( 0.00, 2) circle (2.8pt);
\fill [bluearc] (-1.73, 1) circle (2.8pt);
\fill [bluearc] (-1.73,-1) circle (2.8pt);
\fill [bluearc] ( 0.00,-2) circle (2.8pt);
\fill [bluearc] ( 1.73,-1) circle (2.8pt);
\fill [bluearc] ( 1.73, 1) circle (2.8pt);
\draw [red][line width=0.55pt] (0,0.3) -- (-0.89, 1.45);
\draw [red][line width=0.55pt] (0,0.3) -- ( 0.89, 1.45);
\draw [red][line width=0.55pt] (0,0.3) to[out= 40,in=140] ( 1.73, 0.00);
\draw [red][line width=0.55pt] (0,0.3) to[out= 20,in= 70] ( 0.89,-1.45);
\draw [red][line width=0.55pt] (0,0.3) to[out=160,in=110] (-0.89,-1.45);
\draw [red][line width=0.55pt] (0,0.3) to[out=140,in= 40] (-1.73, 0.00) [dotted];
\fill [red] (0,0.3) circle (2.8pt); \fill [white] (0,0.3) circle (2pt);
\draw (0,0) node{$b$};
\draw[orange][line width=1.2pt] (0, 2.7) node[left]{$c$} to[out=0, in=165] (1,2.6) [dotted];
\draw[orange][line width=1.2pt] (1, 2.6) to[out=-15, in=180] (1.73,-1);
\draw[violet][line width=1.2pt] (0, 2.3) node[left]{$c^{\rota}$} to[out=0, in=165] (1,2.2) [dotted];
\draw[violet][line width=1.2pt] (1, 2.2)
  to[out=-15, in= 90] ( 1,0) to[out=-90,in=  0] (0,-1)
  to[out=180, in=-90] (-1,0) to[out= 90,in=180] (0, 1)
  to[out=  0, in= 90] ( 0.8,0) to[out=-90,in=  0] (0,-0.8)
  to[out=180, in=-90] (-0.8,0) to[out= 90,in=180] (0, 0.8)
  to[out=  0, in= 90] ( 0.6,0) to[out=-90,in=  0] (0,-0.6)
  to[out=180, in=-90] (-0.6,0) to[out= 90,in=180] (0, 0.6);
\draw[violet][line width=1.2pt][dotted] (0, 0.6)
  to[out=  0, in= 90] ( 0.45,0) to[out=-90,in=  0] (0,-0.45)
  to[out=180, in=-90] (-0.45,0) to[out= 90,in=180] (0, 0.45);
\draw (0,-3) node{(b) $\infty$-elementary $\gbullet$-polygon};
\end{tikzpicture}
\ \
\begin{tikzpicture}[scale=0.75] \small
\draw[bluearc] (1.73,1) -- (0, 2) -- (-1.73, 1);
\draw[ black ][line width=1.2pt] (-1.73, 1) -- (-1.73,-1);
\draw[bluearc] (-1.73,-1) -- (0,-2) [dotted];
\draw[bluearc][line width=1.2pt] (0,-2) to[out=90, in=150] (1.73,-1);
\draw[ black ][line width=1.2pt] (0,-2) -- (1.73,-1);
\draw[bluearc] (1.73,-1) -- (1.73,1) [dotted];
\fill [bluearc] ( 1.73, 1) circle (2.8pt);
\fill [bluearc] ( 0.00, 2) circle (2.8pt);
\fill [bluearc] (-1.73, 1) circle (2.8pt);
\fill [bluearc] (-1.73,-1) circle (2.8pt);
\fill [bluearc] ( 0.00,-2) circle (2.8pt);
\fill [bluearc] ( 1.73,-1) circle (2.8pt);
\fill [bluearc] ( 1.73, 1) circle (2.8pt);
\draw [red][line width=0.55pt] (-1.73,0) -- (-0.89, 1.45);
\draw [red][line width=0.55pt] (-1.73,0) -- ( 0.89, 1.45);
\draw [red][line width=0.55pt] (-1.73,0) -- ( 1.73, 0.00) [dotted];
\draw [red][line width=0.55pt] (-1.73,0) -- ( 0.89,-1.45);
\draw [red][line width=0.55pt] (-1.73,0) -- (-0.89,-1.45) [dotted];
\fill [red] (-1.73, 0   ) circle (2.8pt); \fill [white] (-1.73, 0   ) circle (2pt);
\fill [red] ( 0.89,-1.45) circle (2.8pt); \fill [white] ( 0.89,-1.45) circle (2pt);
\draw[orange][line width=1.2pt] (0, 2.7) node[left]{$c$} to[out=0, in=165] (1,2.6) [dotted];
\draw[orange][line width=1.2pt] (1, 2.6) to[out=-15, in=90] ( 0.89,-1.45);
\draw[violet][line width=1.2pt] (0, 2.3) node[left]{$c^{\rota}$} to[out=0, in=165] (1,2.2) [dotted];
\draw[violet][line width=1.2pt] (1, 2.2) to[out=-15,in=120] ( 1.73,-1.00);
\draw (0,-3) node{(c) 2-polygon};
\end{tikzpicture}
\caption{Rotating (the case for $\PP_i$ to be \Pic \ref{fig:arc segment I} Case $\gbullet$-D can be seen as a special case of (b))}
\label{fig:rotating curve}
\end{figure}
  \item[Step 2] \textbf{Rotating the last $\Dblue$-arc segment}.
  We move the endpoint $c(1)$ along the edges of $\PP_{m(c)}$ by using the method similar to Step 1.
\end{itemize}
\end{definition}

The rotation $c^{\rota}$ of $c$ is an admissible curve, and the following result shows the relationship of $\X(c^{\rota})$ and $\MM(c)$.

\begin{theorem}
[{\!\!\cite[Theorem 3.8]{ZhangLiu2024}, \cite[Theorem 2.16]{Chang2025}, \cite[Theorem 3.3]{LZZpre2023}}]
\label{thm:Chang}
Let $A$ be a gentle algebra and $\SURF$ be its marked ribbon surface. Then there is an embedding
\[(-)^{\rota}: \PC_{\m}(\SURF) \to \AC_{\m}(\SURF), ~
c\mapsto \tc^{\rota}\]
such that $\H^{0}(\X(\tc^{\rota}))$ is isomorphic to $\MM(c)$ and $\H^{i}(\X(\tc^{\rota}))=0$ for any $i\neq 0$.
\end{theorem}

Next, we consider the inverse problem of Theorem \ref{thm:Chang},
i.e., for an admissible curve $\tvarsig$, find a curve $c$
such that $\varsigma\simeq c^{\rota}$ (if such $c$ exists)
and there are two grading $\tvarsig$ and $\tc^{\rota}$ satisfying
\begin{center}
  $\X(\tvarsig) \cong \X(\tc^{\rota})$.
\end{center}
To do this, we define the inverse rotation of $\tvarsig$ here.

\subsection{The inverse rotations of admissible curves}

Inverse rotation is a rotation applied to an admissible curve, aimed at restoring the indecomposable module corresponding to the projected complex.

\begin{definition}\label{def:inv rotat} \rm
Let $\varsigma:[0,1]\to \Surf$ {\rm(}with a grading $\tvarsig:\{\varsigma(t_k):1\=<k\=<n(\varsigma)\}\to\ZZ${\rm)} be an admissible curve.
Let $\varsigma_{\ast}$ be the first intersection in $\Surf\backslash\bSurf$ of $\tvarsig$ and $\Dblue$,
and $\varsigma^{\ast}$ be the last intersection in $\Surf\backslash\bSurf$ of $\tvarsig$ and $\Dblue$.
Assume that $\PP_{\ast}$ is the elementary $\gbullet$-polygon
such that the segment from $\varsigma(0)$ to $\varsigma_{\ast}$ of $\tvarsig$ lying in $\PP_{\ast}$,
and $\PP^{\ast}$ is the elementary $\gbullet$-polygon
such that the segment from $\varsigma^{\ast}$ to $\varsigma(1)$ of $\tvarsig$ lying in $\PP^{\ast}$.
Then we call $\PP_{\ast}$ is the \defines{starting $\gbullet$-polygon} of $\tvarsig$
and $\PP^{\ast}$ is the \defines{ending $\gbullet$-polygon} of $\tvarsig$.
Furthermore, we define that its \defines{inverse rotation} is a $\rbullet$-curve,
$c^{\antirota}$, obtained by the following steps.
\begin{itemize}
  \item[Step 1] \textbf{Rotating the first intersection}.
  All edges of $\PP_{\ast}$ are written as $a_{\gbullet}^{1}$, $a_{\gbullet}^{2}$, $\ldots$, $a_{\gbullet}^{u}$ in a clockwise direction.

    \hspace{8pt}
    If $\PP_{\ast}$ is neither a 2-gon nor an $\infty$-elementary $\gbullet$-polygon, see \Pic \ref{fig:antirota} (a), then we move the endpoint $\varsigma(0)$ along the edges lying in $\Egreen(\PP_*)$ of $\PP_*$ in a anticlockwise direction as far as possible,
    until $\varsigma(0)$ is moved to the first endpoint of $a_{\gbullet}^{t+1}$.
    Here, $a_{\gbullet}^{t+1}$ is the first $\Dblue$-arc crossed by $\tvarsig$.

    \hspace{8pt}
    If $\PP_{\ast}$ is a 2-gon as shown in \Pic \ref{fig:arc segment I} Case $\gbullet$-C,
    then we move the endpoint $\varsigma(0)$ to the previous $\rbullet$-marked point
    along the negative direction, see \Pic \ref{fig:antirota} (b).
\begin{figure}[H]
  \centering
\centering
\definecolor{ffqqqq}{rgb}{1,0,0}
\definecolor{bluearc}{rgb}{0,0,1}
\begin{tikzpicture}[scale=0.75] \small
\draw[bluearc][line width=1.2pt] ( 0.00, 2) -- (-1.73, 1) [dotted];
\draw[bluearc][line width=1.2pt] ( 1.73, 1) -- ( 0.00, 2);
\draw[ black ][line width=1.2pt] (-1.73, 1) -- (-1.73,-1);
\draw[bluearc][line width=1.2pt] (-1.73,-1) -- ( 0.00,-2);
\draw[bluearc][line width=1.2pt] ( 0.00,-2) -- ( 1.73,-1) [dotted];
\draw[bluearc][line width=1.2pt] ( 1.73,-1) -- ( 1.73, 1);
\draw(1,1.7) node{$a_{\gbullet}^t$};
\draw(2.2,0) node{$a_{\gbullet}^{t+1}$};
\fill [bluearc] ( 1.73, 1) circle (2.8pt);
\fill [bluearc] ( 0.00, 2) circle (2.8pt);
\fill [bluearc] (-1.73, 1) circle (2.8pt);
\fill [bluearc] (-1.73,-1) circle (2.8pt);
\fill [bluearc] ( 0.00,-2) circle (2.8pt);
\fill [bluearc] ( 1.73,-1) circle (2.8pt);
\fill [bluearc] ( 1.73, 1) circle (2.8pt);
\draw [red][line width=0.55pt] (-1.73,0) -- (-0.89, 1.45)[dotted];
\draw [red][line width=0.55pt] (-1.73,0) -- ( 0.89, 1.45);
\draw [red][line width=0.55pt] (-1.73,0) -- ( 1.73, 0.00);
\draw [red][line width=0.55pt] (-1.73,0) -- ( 0.89,-1.45)[dotted];
\draw [red][line width=0.55pt] (-1.73,0) -- (-0.89,-1.45);
\fill [red] (-1.73,0) circle (2.8pt); \fill [white] (-1.73,0) circle (2pt);
\draw[orange][line width=1.2pt] (0, 2.7) node[left]{$\tvarsig^{\antirota}$} to[out=0, in=165] (1,2.6) [dotted];
\draw[orange][line width=1.2pt] (1, 2.6) to[out=-15, in=180] (1.73,-1) [<-];
\draw[violet][line width=1.2pt] (0, 2.3) node[left]{$\tvarsig$} to[out=0, in=165] (1,2.2) [dotted];
\draw[violet][line width=1.2pt] (1, 2.2)
  to[out=-15, in= 80] ( 1,0) to[out=-110,in=  0] (0,-1) -- (-1.73,-1)[<-];
\draw[black] (-1.73*1.2,-1.00*1.2) -- (-0.00*1.2,-2.00*1.2) -- ( 1.73*1.2,-1.00*1.2) [->][dashed];
\draw (0,-3) node{(a)};
\end{tikzpicture}
\ \ \ \
\begin{tikzpicture}[scale=0.75] \small
\draw[bluearc][line width=1.2pt] (1.73,1) -- (0, 2) -- (-1.73, 1);
\draw[ black ][line width=1.2pt] (-1.73, 1) -- (-1.73,-1);
\draw[bluearc][line width=1.2pt] (-1.73,-1) -- (0,-2) [dotted];
\draw[bluearc][line width=1.2pt] (0,-2) to[out=90, in=150] (1.73,-1);
\draw[ black ][line width=1.2pt] (0,-2) -- (1.73,-1);
\draw[bluearc][line width=1.2pt] (1.73,-1) -- (1.73,1) [dotted];
\fill [bluearc] ( 1.73, 1) circle (2.8pt);
\fill [bluearc] ( 0.00, 2) circle (2.8pt);
\fill [bluearc] (-1.73, 1) circle (2.8pt);
\fill [bluearc] (-1.73,-1) circle (2.8pt);
\fill [bluearc] ( 0.00,-2) circle (2.8pt);
\fill [bluearc] ( 1.73,-1) circle (2.8pt);
\fill [bluearc] ( 1.73, 1) circle (2.8pt);
\draw [red][line width=0.55pt] (-1.73,0) -- (-0.89, 1.45);
\draw [red][line width=0.55pt] (-1.73,0) -- ( 0.89, 1.45);
\draw [red][line width=0.55pt] (-1.73,0) -- ( 1.73, 0.00) [dotted];
\draw [red][line width=0.55pt] (-1.73,0) -- ( 0.89,-1.45);
\draw [red][line width=0.55pt] (-1.73,0) -- (-0.89,-1.45) [dotted];
\fill [red] (-1.73, 0   ) circle (2.8pt); \fill [white] (-1.73, 0   ) circle (2pt);
\fill [red] ( 0.89,-1.45) circle (2.8pt); \fill [white] ( 0.89,-1.45) circle (2pt);
\draw[orange][line width=1.2pt] (0, 2.7) node[left]{$\tvarsig^{\antirota}$} to[out=0, in=165] (1,2.6) [dotted];
\draw[orange][line width=1.2pt] (1, 2.6) to[out=-15, in=90] ( 0.89,-1.45)[<-];
\draw[violet][line width=1.2pt] (0, 2.3) node[left]{$\tvarsig$} to[out=0, in=165] (0.8,2.2) [dotted];
\draw[violet][line width=1.2pt] (0.8, 2.2) to[out=-15,in=120] ( 1.73,-1.00)[<-];
\draw[black] ( 0.89*1.2,-1.45*1.2) -- ( 1.73*1.2,-1.00*1.2) [<-][dashed];
\draw (0,-3) node{(b)};
\end{tikzpicture}
  \caption{Inverse rotating}
  \label{fig:antirota}
\end{figure}
  \item[Step 2] \textbf{Rotating the last intersection}.
  We move the endpoint $\varsigma(1)$ along the edges of $\PP^{\ast}$
  by using the method similar to Step 1.
\end{itemize}
\end{definition}

Notice that we not need to consider the case of $\tvarsig \in
\AC ^{\oslash}_{\oslash}(\SURF) \cup \AC^{\oslash}_{\m}(\SURF)$
because each semi-orthogonal decomposition of $\Dcat^b(A)$
can be induced by a semi-orthogonal decomposition of $\per(A)$.
Moreover, it is clear that \[\MM((c^{\rota})^{\antirota}) \cong \MM(c)\]
holds for any $c\in\PC_{\m}(\SURF)$.

\begin{definition} \rm
We call that an admissible curve $\tvarsig$ satisfies \defines{$\Dblue$-arc property} if
it is such that all segments obtained by $\Dblue$ cutting $\tvarsig$,
except the segment from $\varsigma(0)$ to $\varsigma_{\ast}$ and the segment from $\varsigma^{\ast}$ to $\varsigma(1)$,
are $\Dblue$-arc segments.
\end{definition}

\begin{lemma} \label{lemm:antirota}
For any admissible curve $\tvarsig$ satisfying $\Dblue$-arc property, then $\varsigma^{\antirota}\in\PC_{\m}(\SURF)$,
i.e., exists a permissible curve $c\in \PC(\SURF)$ such that $\tvarsig \simeq \tc^{\rota}$.
\end{lemma}

\begin{proof}
We only proof the case for $\tvarsig$ to be finite. The cases for left infinite/right infinite/infinite are similar.
Assume $\varsigma = \varsigma_{[0,1]}\varsigma_{[1,2]}\cdots \varsigma_{[n(\varsigma),n(\varsigma)+1]}$
(each $\varsigma_{[i,i+1]}$ is a map $\varsigma_{[i,i+1]}:[t_i,t_{i+1}]\to\Surf$ by
the notations given in Definition \ref{def:perm and adm curve} (2)).
If $P_{\ast}$ and $P^{\ast}$ are not a 2-gons (we write all edges of $P_{\ast}$ are $a_{\gbullet}^{1}$, $a_{\gbullet}^{2}$, $\ldots$, $a_{\gbullet}^{u}$ in a clockwise direction, see \Pic \ref{fig:inver antirota}),
then we get that the segment from $\varsigma(0)$ to $\varsigma_{\ast}$ is
$\varsigma_{[0,1]}\varsigma_{[1,2]}\cdots \varsigma_{[u-(\theta+1),u-\theta]}$, i.e.,
\[ \varsigma = \varsigma_{[0,1]}\varsigma_{[1,2]}\cdots \varsigma_{[u-(\theta+1),u-\theta]}
\underbrace{\cdots \varsigma_{[n(\varsigma),n(\varsigma)+1]}}_{\sigma}.\]
Let $\sigma':[0,1]\to \PP_{\ast}$ ($\subseteq \Surf$) be the $\Dblue$-arc segment lying in $\PP_{\ast}$
such that $\sigma'(1)=\varsigma_{[u-(\theta+1),u-\theta]}(t_{u-\theta})$.
Then we obtain a curve $\sigma'\sigma$, and the segment of $\varsigma$ from $\varsigma^{\ast}$ to $\varsigma(1)$
is also a segment of the curve
\[\sigma'\sigma = \sigma'\varsigma_{[u-\theta,u-\theta+1]} \cdots \varsigma_{[n(\varsigma), n(\varsigma)+1]}. \]
\begin{figure}[H]
  \centering
\centering
\definecolor{ffqqqq}{rgb}{1,0,0}
\definecolor{bluearc}{rgb}{0,0,1}
\begin{tikzpicture} \small
\draw[bluearc][line width=1.2pt] ( 0.00, 2) -- (-1.73, 1) [dotted];
\draw[bluearc][line width=1.2pt] ( 1.73, 1) -- ( 0.00, 2);
\draw[ black ][line width=1.2pt] (-1.73, 1) -- (-1.73,-1);
\draw[bluearc][line width=1.2pt] (-1.73,-1) -- ( 0.00,-2);
\draw[bluearc][line width=1.2pt] ( 0.00,-2) -- ( 1.73,-1) [dotted];
\draw[bluearc][line width=1.2pt] ( 1.73,-1) -- ( 1.73, 1);
\draw( 1.0, 1.7) node{$a_{\gbullet}^{\theta}$};
\draw( 1.9, 0.0) node[right]{$a_{\gbullet}^{\theta+1}=a_{\gbullet}^{u-(u-(\theta+1))}$};
\draw(-1.0,-1.7) node{$a_{\gbullet}^u$};
\fill [bluearc] ( 1.73, 1) circle (2.8pt);
\fill [bluearc] ( 0.00, 2) circle (2.8pt);
\fill [bluearc] (-1.73, 1) circle (2.8pt);
\fill [bluearc] (-1.73,-1) circle (2.8pt);
\fill [bluearc] ( 0.00,-2) circle (2.8pt);
\fill [bluearc] ( 1.73,-1) circle (2.8pt);
\fill [bluearc] ( 1.73, 1) circle (2.8pt);
\draw [red][line width=0.55pt] (-1.73,0) -- (-0.89, 1.45)[dotted];
\draw [red][line width=0.55pt] (-1.73,0) -- ( 0.89, 1.45);
\draw [red][line width=0.55pt] (-1.73,0) -- ( 1.73, 0.00);
\draw [red][line width=0.55pt] (-1.73,0) -- ( 0.89,-1.45)[dotted];
\draw [red][line width=0.55pt] (-1.73,0) -- (-0.89,-1.45);
\fill [red] (-1.73,0) circle (2.8pt); \fill [white] (-1.73,0) circle (2pt);
\draw[orange][line width=1.2pt] (0, 2.7) node[left]{$\varsigma^{\antirota}$} to[out=0, in=165] (1,2.6) [dotted];
\draw[orange][line width=1.2pt] (1, 2.6) to[out=-15, in=180] (1.73,-1) [<-];
\draw[violet][line width=1.2pt] (0, 2.3) node[left]{$\tvarsig$} to[out=0, in=165] (1,2.2) [dotted];
\draw[violet][line width=1.2pt] (1, 2.2)
  to[out=-15, in= 80] ( 1,0) to[out=-110,in=  0] (0,-1) -- (-1.73,-1)[<-];
\draw[black] (-1.73*1.2,-1.00*1.2) -- (-0.00*1.2,-2.00*1.2) -- ( 1.73*1.2,-1.00*1.2) [->][dashed];
\end{tikzpicture}
\caption{Inverse rotating}
\label{fig:inver antirota}
\end{figure}

In a similar way, we can find a $\Dblue$-arc segment $\sigma''$ to replace the segment from $\varsigma^{\ast}$ to $\varsigma(1)$,
and obtain a new curve $\sigma'\sigma''$ which lies in $\PC(\SURF)$.
In this case, we have $c:=\sigma'\sigma'' \simeq \varsigma^{\antirota}$,
it follows that $\varsigma\simeq c^{\rota}$.

The case for at least one of $\PP_{\ast}$ and $\PP^{\ast}$ to be a 2-gon as shown in \Pic \ref{fig:arc segment I} Case $\gbullet$-C is similar.
\end{proof}

\begin{lemma} \label{lemm:rota}
For any $c\in \PC_{\m}(\SURF)$, we have $\tc^{\rota}\in \AC(\SURF)$,
and $\tc^{\rota}$ satisfies $\Dblue$-arc property.
\end{lemma}

\begin{proof}
First of all, by Theorem \ref{thm:Chang}, we have $\tvarsig := \tc^{\rota}\in\AC(\SURF)$.
Next, assume that all segments obtained by $\Dblue$ cutting $\varsigma$ are
$\varsigma_1$, $\varsigma_2$, $\ldots$, $\varsigma_{\ell}$,
where $\varsigma_1$ is the segment from $\varsigma(0)$ to $\varsigma_{\ast}$,
and $\varsigma_{\ell}$ is the segment from $\varsigma^{\ast}$ to $\varsigma(1)$.
If $\tvarsig$ does not satisfy $\Dblue$-arc property, then there exists an integer $i$, $2\=< i\=< \ell-1$,
such that $\varsigma_i$ is not a $\Dblue$-arc segment, i.e.,
$\varsigma_i$ is not any form shown in \Pic \ref{fig:arc segment I}.
In this case, we obtain that $c$ is not a permissible curve
since $\varsigma_2$, $\ldots$, $\varsigma_{\ell-1}$ can be seen as segments of $c$
obtained by $\Dblue$ cutting $c$, a contradiction.
\end{proof}

By Lemmas \ref{lemm:antirota} and \ref{lemm:rota}, we obtain the following result immediately.

\begin{proposition}\label{prop:rota}
The inverse rotation $\varsigma^{\antirota}$ of an admissible curve $\tvarsig$
is a permissible curve if and only if it satisfies $\Dblue$-arc property.
\end{proposition}

\begin{remark}\label{rmk:rota}\rm
Proposition \ref{prop:rota} shows that when $\tvarsig$ satisfies $\Dblue$-arc property,
$\varsigma^{\antirota}$ is a permissible curve, and, by Theorem \ref{thm:Chang},
there is an integer $n$ such that $\shift{n}{\X(\tvarsig)}$ ($\in\ind(\per(A))$) is a
projective resolution of $\MM(\varsigma^{\antirota})$.
\end{remark}

\section{Semi-orthogonal decompositions for gentle algebras} \label{sec:semi-orth of gentle}

In our paper, for a category $\calA$ and its full subcategory $\calC$,
we call $\calC$ is \defines{closed under extensions} if for any $X, Z\in \calC$,
every object $Y$ given by an extension $0\to X\to Y \to Z \to 0$ in $\calA$ is also an object in $\calC$.
For an Abelian category $\calA$ and a collection $\calS$ of some objects in $\calA$,
we denote by $\langle\calS\rangle_{\calA}$ the smallest full subcategory of $\calA$ containing $\calS$
such that it is closed under extensions.
%
For a triangulated category $\calT$ and a collection $\calS$
of objects in $\calT$, we denote by $\langle\calS\rangle_{\calT}$
the smallest full triangulated subcategory of $\calT$ containing $\calS$.

\subsection{Conditions for semi orthogonal decompositions}

In this subsection we provide a description of semi-orthogonal decompositions for gentle algebras by using FFAS.

\begin{proposition} \label{prop:OD}
Let $A$ be a gentle algebra and $\SURF=(\Surf,\M,\Y,\Dblue,\Dred)$ be its marked ribbon surface.
If there are two families of permissible curves $\calC= \{c_i : i\in I\}$
and $\calD = \{c_j' : j\in J\}$ such that the following conditions \ref{OD1} and \ref{OD2} hold, then
$\langle\langle\X(\calC^{\rota})\rangle_{\per(A)}, \langle\X(\calD^{\rota})\rangle_{\per(A)} \rangle$
is a semi-orthogonal decomposition of $\per(A)$.
\begin{enumerate}[label=\text{\rm(sod\arabic*)}]
  \item The union $\calC^{\rota} \cup \calD^{\rota}$ is a $\gbullet$-FFAS of $\SURF$. \label{OD1}
  \item For any $c_i\in\calC$ and $c_j'\in\calD$, one of the following conditions holds: \label{OD2}
    \begin{enumerate}[label=\text{\rm(sod2.\arabic*)}]
      \item $c_i^{\rota}\cap c_j'^{\rota} = \emptyset$, i.e., $c_i$ and $c_j'$ has no intersection anywhere on $\Surf$;
        \label{OD2.1}
      \item if $c_i^{\rota}\cap c_j'^{\rota} \ne \emptyset$, then for each intersection $p\in c_i^{\rota}\cap c_j'^{\rota}$,
        $p$ must be a $\gbullet$-marked point and $c_i$ is right to $c_j'$ at $p$.
        \label{OD2.2}
    \end{enumerate}
\end{enumerate}
\end{proposition}

\begin{proof}
If there are two indecomposable complexes $X\in\langle\X(\calC^{\rota})\rangle_{\per(A)}$
and $Y\in\langle\X(\calD^{\rota})\rangle_{\per(A)}$
such that $\Hom_{\per(A)}(X,Y)\ne 0$, then, consider $X\in\langle\X(\calC^{\rota})\rangle_{\per(A)}$, we have:
\begin{enumerate}[label=\text{\rm(\arabic*)}]
  \item there are admissible curves $\tvarsig_{X,1}$, $\tvarsig_{X,2}$, $\ldots$, $\tvarsig_{X,r} \in \calC^{\rota}$ such that
    $\varsigma_{X,i}(0)=\varsigma_{X,i+1}$ holds for all $1\=< i <r$;
  \item the admissible curve $\tvarsig_X$ corresponding to $X$ is a curve from $\varsigma_{X,1}(0)$ to $\varsigma_{X,r}(1)$;
  \item for the admissible curve $\tvarsig_Y$ corresponding to $Y$, $\tvarsig_X$ and $\tvarsig_Y$ has at least one intersection point $q$
    (in this case, we have two cases need be considered: (a) $q$ lies in the inner $\innerSurf$ of $\Surf$,
    (b) $\tvarsig_X\cap\tvarsig_Y\cap(\innerSurf)=\emptyset$,
    and $q$ is an endpoint of $\tvarsig_X$, see \Pic \ref{fig:X 260723} (A), (B) and (C)).
\end{enumerate}
\begin{figure}[htbp]
  \centering
\begin{tikzpicture}[scale=0.81]
\foreach \x in {0,120,180,240,300}
\draw
  [violet][line width=1.0pt][rotate= \x]
  [postaction={on each segment={mid arrow=violet}}]
    ( 2.00, 0.00) -- ( 1.00, 1.73);
\draw
  [white][line width=1.2pt][rotate=240]
  [postaction={on each segment={mid arrow=violet}}]
    ( 2.00, 0.00) -- ( 1.00, 1.73)[dotted];
\draw
  [violet][line width=1.0pt]
  [postaction={on each segment={mid arrow=violet}}]
  (-1.00, 1.73) to[out=-120, in= 180]
  ( 0.00,-1.25) to[out=   0, in= -60] ( 1.00, 1.73);
\foreach \x in {0,60,120,180,240,300}
\fill[blue][rotate=\x] ( 2.00, 0.00) circle(2pt);
\draw[violet][rotate=180-30] (2.25, 0.00) node{$\tvarsig_{X,1}$};
\draw[violet][rotate=240-30] (2.25, 0.00) node{$\tvarsig_{X,2}$};
\draw[violet][rotate=360-30] (2.25, 0.00) node{$\tvarsig_{X,r-1}$};
\draw[violet][rotate=420-30] (2.25, 0.00) node{$\tvarsig_{X,r}$};
\draw[violet][rotate=     0] (0.55,-1.05) node[above]{$\tvarsig_X$};
\draw[violet!50][line width=1.2pt] (0,-2) -- (0,2)[->] node[left]{$\tvarsig_Y$};
\draw[green][line width=1.25pt] (0,-1.25) circle(3pt) node[below right]{$q$};
\draw[black] (0,-2.5) node{(A): The intersection point $q$};
\draw[black] (0,-3.0) node{of $\tvarsig_X$ and $\tvarsig_Y$ lies in $\innerSurf$};
\end{tikzpicture}
\\
\begin{tikzpicture}[scale=0.81]
\foreach \x in {0,120,180,240,300}
\draw
  [violet][line width=1.0pt][rotate= \x]
  [postaction={on each segment={mid arrow=violet}}]
    ( 2.00, 0.00) -- ( 1.00, 1.73);
\draw
  [white][line width=1.2pt][rotate=240]
  [postaction={on each segment={mid arrow=violet}}]
    ( 2.00, 0.00) -- ( 1.00, 1.73)[dotted];
\draw
  [violet][line width=1.0pt]
  [postaction={on each segment={mid arrow=violet}}]
  (-1.00, 1.73) to[out=-120, in= 180]
  ( 0.00,-1.25) to[out=   0, in= -60] ( 1.00, 1.73);
\foreach \x in {0,60,120,180,240,300}
\fill[blue][rotate=\x] ( 2.00, 0.00) circle(2pt);
\draw[violet][rotate=180-30] (2.25, 0.00) node{$\tvarsig_{X,1}$};
\draw[violet][rotate=240-30] (2.25, 0.00) node{$\tvarsig_{X,2}$};
\draw[violet][rotate=360-30] (2.25, 0.00) node{$\tvarsig_{X,r-1}$};
\draw[violet][rotate=420-30] (2.25, 0.00) node{$\tvarsig_{X,r}$};
\draw[violet][rotate=     0] (0.55,-1.05) node[above]{$\tvarsig_X$};
\draw[violet!50][line width=1.2pt]
  (-1.00, 1.73) to[out=-125, in= 180]
  (-0.00,-1.45) --
  ( 0.20,-1.45) to[out=   0, in=  90]
  ( 0.85,-2.12)[->] node[right]{$\tvarsig_Y$};
\draw[green][line width=1.25pt] (-1.00, 1.73) circle(3pt) node[left]{$q$};
\draw[black] (0,-2.5) node{(B): The intersection point $q$};
\draw[black] (0,-3.0) node{of $\tvarsig_X$ and $\tvarsig_Y$ lies in $\bSurf$};
\end{tikzpicture}
\
\begin{tikzpicture}[scale=0.81]
\foreach \x in {0,120,180,240,300}
\draw
  [violet][line width=1.0pt][rotate= \x]
  [postaction={on each segment={mid arrow=violet}}]
    ( 2.00, 0.00) -- ( 1.00, 1.73);
\draw
  [white][line width=1.2pt][rotate=240]
  [postaction={on each segment={mid arrow=violet}}]
    ( 2.00, 0.00) -- ( 1.00, 1.73)[dotted];
\draw
  [violet][line width=1.0pt]
  [postaction={on each segment={mid arrow=violet}}]
  (-1.00, 1.73) to[out=-120, in= 180]
  ( 0.00,-1.25) to[out=   0, in= -60] ( 1.00, 1.73);
\foreach \x in {0,60,120,180,240,300}
\fill[blue][rotate=\x] ( 2.00, 0.00) circle(2pt);
\draw[violet][rotate=180-30] (2.25, 0.00) node{$\tvarsig_{X,1}$};
\draw[violet][rotate=240-30] (2.25, 0.00) node{$\tvarsig_{X,2}$};
\draw[violet][rotate=360-30] (2.25, 0.00) node{$\tvarsig_{X,r-1}$};
\draw[violet][rotate=420-30] (2.25, 0.00) node{$\tvarsig_{X,r}$};
\draw[violet][rotate=     0] (0.55,-1.05) node[above]{$\tvarsig_X$};
\draw[violet!50][line width=1.2pt]
  (-1.00, 1.73) -- (-2.00, 1.73)[->] node[left]{$\tvarsig_Y$};
\draw[green][line width=1.25pt] (-1.00, 1.73) circle(3pt) node[right]{$q$};
\draw[black] (0,-2.5) node{(C): The intersection point $q$};
\draw[black] (0,-3.0) node{of $\tvarsig_X$ and $\tvarsig_Y$ lies in $\bSurf$};
\end{tikzpicture}
  \caption{The complex $X$ lying in $\langle\X(\calC^{\rota})\rangle_{\per(A)}$
  and the complex $Y$ lying in $\langle\X(\calD^{\rota})\rangle_{\per(A)}$}
  \label{fig:X 260723}
\end{figure}
\noindent
In Case (a), there is an integer $n_{X,Y}$ such that $\Hom_{\per(A)}(X, \shift{n_{X,Y}}{Y})\ne 0$
since the intersection point $q$ provide a non-zero homomorphism
$\X(\tvarsig_X)\to\shift{n_{X,Y}}{\X(\tvarsig_Y)}$ in $\per(A)$.
Then $\tvarsig_Y$ intersect with some $\tvarsig_{X,\imath}$ ($1\=<\imath\=< r$).
In this situation, if $\varsigma_Y(0)=\varsigma_{X,\imath}(0)$ for some $1<\imath\=< r$,
then we obtain $\Hom_{\per(A)}(\X(\tvarsig_{X,\imath-1}), \shift{n_{\imath-1}}{Y})\ne 0$ for some $n_{\imath-1}\in\ZZ$;
otherwise, we have $\emptyset\ne\tvarsig_Y\cap\tvarsig_{X,\imath} \subseteq \innerSurf$ for some $1\=<\imath\=< r$,
it admits $\Hom_{\per(A)}(\X(\tvarsig_{X,\imath}), \shift{n_{\imath}'}{Y})\ne 0$ for some $n_{\imath}'\in\ZZ$.
Notice that each $\tvarsig_{X,\imath}$ lies in $\calC^{\rota}$,
then, by Proposition \ref{prop:rota}, the inverse rotation $\tvarsig_{X,\imath}^{\antirota}$ of $\tvarsig_{X,\imath}$ is a permissible curve in $\calC$.
Therefore, in Case (a), we obtain that
\begin{align}\label{eq:Hom260723}
 \text{there is~} c_i\in\calC \text{~such that~} \Hom_{\per(A)} (\X(\tc_i^{\rota}), \shift{n}{Y}) \ne 0
\end{align}

In Case (b), each intersection point of $\tvarsig_X$ and $\tvarsig_Y$ is a $\gbullet$-marked point.
Since $\Hom_{\per(A)}(X,Y)\ne 0$, there is a point $q\in\tvarsig_X\cap\tvarsig_Y$ such that
$\tvarsig_X$ is left to $\tvarsig_Y$ at $q$.
Without loss of generality, set $q=\varsigma_{X,1}(0)$, then we obtain two subcases as follows:
\begin{itemize}
  \item[(b.1)] $\varsigma_{X}$ is left to $\varsigma_{Y}$ at $q$,
    and $\varsigma_{Y}$ is left to $\varsigma_{X,1}$ at $q$ (see \Pic \ref{fig:X 260723} (B)),
  \item[(b.2)] $\varsigma_{X}$ is left to $\varsigma_{X,1}$ at $q$,
    and $\varsigma_{X,1}$ is left to $\varsigma_{Y}$ at $q$ (see \Pic \ref{fig:X 260723} (C)).
\end{itemize}
In Subcase (b.1), there exists a $\varsigma_{X,\imath}$ intersecting with $\varsigma_{Y}$,
and we can obtain \eqref{eq:Hom260723} by using a method similar to Case (a).
In Subcase (b.2), we have $\Hom_{\per(A)}(\X(\tvarsig_{X,1}),\shift{\ell}{Y}) \ne 0$ for some $\ell\in\ZZ$.
It follows that \eqref{eq:Hom260723} holds for some $c_i\in\calC$.

Next, we using the same way to consider $Y\in\langle\X(\calD^{\rota})\rangle_{\per(A)}$
and $\Hom_{\per(A)}(\X(\tc_i^{\rota}), \shift{\ell}{Y})\ne 0$, then we have:
\begin{enumerate}[label=\text{\rm(\arabic*$'$)}]
  \item there are admissible curves $\tvarsig_{Y,1}$, $\tvarsig_{Y,2}$, $\ldots$, $\tvarsig_{Y,s} \in \calD^{\rota}$ such that
    $\varsigma_{Y,j}(0)=\varsigma_{Y,j+1}$ holds for all $1\=< j<s$;
  \item the admissible curve $\tvarsig_Y$ corresponding to $Y$ is a curve from $\varsigma_{Y,1}(0)$ to $\varsigma_{X,s}(1)$;
  \item for the admissible curve $\tc_i^{\rota}$ corresponding to $\X(\tc_i^{\rota})$, $\tvarsig_Y$ and $\tc_i^{\rota}$ has at least one intersection point.
\end{enumerate}
Furthermore, using a method similar to that used to obtain \eqref{eq:Hom260723}, we can show that
\begin{align}\label{eq:Hom260723-2}
  \text{there is~} c_j'\in \calD \text{~such that~}
  \Hom_{\per(A)}(\X(\tc_i^{\rota}), \shift{m}{\X(\tc_j'^{\rota})}) \ne 0
\end{align}
holds for some $m\in\ZZ$. Thus,
\[\tc_i^{\rota}\cap\tc_j'^{\rota}\ne\emptyset.\]
By \ref{OD1}, $\tc_i^{\rota}\cap\tc_j'^{\rota} \subseteq \Dblue$, then \eqref{eq:Hom260723-2} admits that
there is an common endpoint $p$ such that $\tc_i^{\rota}$ is left to $\tc_j'^{\rota}$ at $p$.
However, it contradict with \ref{OD2.2}. Therefore, we have that
$\Hom_{\per(A)}(X,Y)=0$ holds for all $X\in\langle\X(\calC^{\rota})\rangle_{\per(A)}$
and $Y\in\langle\X(\calD^{\rota})\rangle_{\per(A)}$, i.e.,
\begin{align}\label{eq:Hom is zero}
\Hom_{\per(A)}(\langle\X(\calC^{\rota})\rangle_{\per(A)},\langle\X(\calD^{\rota})\rangle_{\per(A)})=0,
\end{align}
which shows that \ref{def-OD1} holds.

Finally, since $\calC^{\rota}\cup\calD^{\rota}$ is a $\rbullet$-FFAS,
we obtain that the triangulated subcategory of $\per(A)$
generated by $\X(\calC^{\rota}\cup\calD^{\rota}) =
\bigoplus_{\tvarsig\in \calC^{\rota}\cup\calD^{\rota}} \X(\tvarsig)$ is $\per(A)$.
Thus, for any indecomposable object $X$ in $\per(A)$, there is a family of distinguished triangle
\[ (\mathbf{T}_t:= ~ {_tX} \to X_t \to {^tX} \to \shift{1}{{_tX}})_{0\=<t\=< T} \]
such that:
${_0X}\cong \X({_0\tvarsig})$ for some ${_0\tvarsig} \in \calC^{\rota}\cup\calD^{\rota}$,
    and ${^0X}\cong \X({^0\tvarsig})$ for some ${^0\tvarsig}\in \calC^{\rota}\cup\calD^{\rota}$;
for each $\mathbf{T}_t$, $1\=< t < T$, ${_tX}$ and ${^tX}$ are one of the items in some $\mathbf{T}_{t'}$, here, $0\=< t'\=< t-1$;
and $X=X_t$.
By Proposition \ref{prop:rota}, ${_0}\tvarsig$ and ${^0}\tvarsig$ respectively have inverse rotations
${_0}\tvarsig^{\antirota}$ and ${^0}\tvarsig^{\antirota}$ lying in $\PC(\SURF)$
since ${_0}\tvarsig$ and ${^0}\tvarsig$ lies in $\calC^{\rota}\cup\calD^{\rota}$.
The following four cases show that $X\in\langle\X(\calC^{\rota})\rangle_{\per(A)} * \langle\X(\calD^{\rota})\rangle_{\per(A)}$:
\begin{enumerate}[label=\text{\rm(\alph*)}]
  \item If ${_0}\tvarsig^{\antirota}\in \calC$ and ${^0}\tvarsig^{\antirota}\in\calD$,
then we have $X=X_t \in \langle\X(\calC^{\rota})\rangle_{\per(A)} * \langle\X(\calD^{\rota})\rangle_{\per(A)}$.

  \item If ${_0}\tvarsig^{\antirota}, {^0}\tvarsig^{\antirota}\in\calC$,
then we have $X=X_t\in \langle\X(\calC^{\rota})\rangle_{\per(A)} \subseteq \langle\X(\calC^{\rota})\rangle_{\per(A)} * \langle\X(\calD^{\rota})\rangle_{\per(A)}$.

  \item If ${_0}\tvarsig^{\antirota}, {^0}\tvarsig^{\antirota}\in\calD$, then we have $X=X_t\in \langle\X(\calD^{\rota})\rangle_{\per(A)} \subseteq \langle\X(\calC^{\rota})\rangle_{\per(A)} * \langle\X(\calD^{\rota})\rangle_{\per(A)}$.
  \item If ${_0}\tvarsig^{\antirota}\in \calD$ and ${^0}\tvarsig^{\antirota}\in\calC$,
    then we have a distinguished triangle $\mathbf{T}:$ $\X({_0}\tvarsig) \to  X_1 \to \X({^0}\tvarsig)
    \mathop{\to}\limits^{\varphi} \shift{1}{\X({_0}\tvarsig)}$
    which implies a homomorphism $\varphi$ lying in
    $\Hom_{\per(A)}(\langle\calC^{\rota}\rangle_{\per(A)}, \langle\calD^{\rota}\rangle_{\per(A)})$.
    By \eqref{eq:Hom is zero}, we get $\varphi=0$, then $\mathbf{T}$ is split.
    Thus it is trivial that $X_1 \in \langle\X(\calC^{\rota})\rangle_{\per(A)} * \langle\X(\calD^{\rota})\rangle_{\per(A)}$.
\end{enumerate}
Then $\per(A)=\langle\X(\calC^{\rota})\rangle_{\per(A)} * \langle\X(\calD^{\rota})\rangle_{\per(A)}$,
we obtain \ref{def-OD2}.
\end{proof}

\subsection{Semi-orthogonal decompositions and full formal arc systems}

\begin{proposition} \label{prop:good cut}
Let $A$ be a gentle algebra and $\SURF=(\Surf,\M,\Y,\Dblue,\Dred)$ be its marked ribbon surface.
If $\per(A)$ has a semi-orthogonal decomposition $\langle\calF,\calG\rangle$,
then there is a $\gbullet$-FFAS $\Delta$ {\rm(}$\Delta$ and $\Dblue$ may be different{\rm)} such that the following statements hold.
\begin{enumerate}[label=\text{\rm(\arabic*)}]
  \item The good cut $\Omega$ corresponding to $\langle\calF,\calG\rangle$ satisfies $\Omega\cap\Delta\cap(\innerSurf)=\emptyset$.
    \label{prop item: good cut 1}
  \item The $\gbullet$-FFAS $\Delta$ has a decomposition $\Delta=D_1\cup D_2$ that is a disjoint union.
    \label{prop item: good cut 2}
  \item $\calF = \langle\X(D_1)\rangle_{\per(A)}$ and $\calG=\langle \X(D_2)\rangle_{\per(A)}$.
    \label{prop item: good cut 3}
\end{enumerate}
\end{proposition}

\begin{proof}
Good cut $\Omega$ divides the marked surface $(\Surf, \M, \Y)$ to some subsurface
$(\Surf_1, \M_1, \Y_1)$, $\ldots$, $(\Surf_t, \M_t, \Y_t)$.
These subsurfaces are marked surfaces, which have some boundary segments is a dividing curve in $\Omega$
(cf. the example shown in \Pic \ref{fig:goodcut-divid}, the subsurface $\Surf'$ is a marked surface).
\begin{figure}[H]
  \centering
\begin{tikzpicture}
\draw[black!25][line width=15pt][rotate= 30] (2,0) arc(0:30:2);
\draw[black!25][line width=15pt][rotate=210] (2,0) arc(0:90:2);
\fill[white] (2,0) arc(0:360:2);
\fill[red!15] ( 1.00,-1.73)--( 1.73, 1.00) arc(30:60:2)
  -- (-1.73,-1.00) arc(210:300:2);
\draw[red] ( 0.5,-0.5) node{$\Surf'$};
\draw[black][line width=1pt][rotate= 30] (2,0) arc(0:30:2);
\draw[black][line width=1pt][rotate=210] (2,0) arc(0:30:2);
\draw[black][line width=1pt][rotate=240] (2,0) arc(0:60:2)[dashed];
\foreach \x in {0,60,120,180,240,300}
\fill[blue][rotate=\x] (2,0) circle(2pt);
\foreach \x in {0,60,120,180,300}
\fill[red][rotate=\x+30] (2,0) circle(2pt);
\foreach \x in {0,60,120,180,300}
\fill[white][rotate=\x+30] (2,0) circle(1.5pt);
\draw[blue][line width=1pt] ( 1.00, 1.73) -- (-2.00, 0.00);
\draw[blue][line width=1pt] ( 1.00, 1.73) -- (-1.00,-1.73);
\draw[blue][line width=1pt] ( 1.00, 1.73) -- ( 1.00,-1.73);
\draw[blue][line width=1pt] (-1.00, 1.73) -- (-2.00, 0.00);
\draw[blue][line width=1pt] ( 1.00,-1.73) -- ( 2.00, 0.00);
\draw[cyan][postaction={on each segment={mid arrow=cyan}}]
  [line width=1pt]
  (-1.73, -1.00) -- ( 1.00, 1.73);
\draw[cyan] (-0.36, 0.36) node[left]{$\omega$};
\draw[cyan][postaction={on each segment={mid arrow=cyan}}]
  [line width=1pt]
  ( 1.73, 1.00) -- ( 1.00,-1.73);
\draw[cyan] ( 1.56, 0.36) node[right]{$\omega'$};
\end{tikzpicture}
  \caption{An example for marked subsurface obtained by}
  \vspace{-0.4cm}
  \begin{center}
    {a good cut $\Omega$ dividing marked surface $(\Surf,\M,\Y)$}
  \end{center}
  \label{fig:goodcut-divid}
\end{figure}

By Remark \ref{rmk:FFAS}, each marked surface $(\Surf_i,\M_i,\Y_i)$ has a $\gbullet$-FFAS $\Delta_i$,
and, by Definition \ref{def:perm and adm curve} \ref{def:adm curv},
each $\gbullet$-curve $\varsigma$ in $\Delta_i$
can be seen as an admissible curve with a natural grading $\tvarsig$.
Then:
\begin{itemize}
  \item we obtain a $\gbullet$-FFAS $\Delta:=\bigcup\limits_{i=1}^t \Delta_i$,
  \item for each $\SURF_i:=(\Surf_i,\M_i,\Y_i,\Delta_i,\Delta_i^{\bot})$;
    by Definition \ref{def:good cut} \ref{GC1},
    $\M_i$ contains two types of marked points: one type comes from the original $\M$, and the other type consists either entirely of left added marked points or entirely of right added marked points;
  \item Furthermore, $\Omega\cap\Delta\cap(\innerSurf)=\emptyset$ (thus, \ref{prop item: good cut 1} holds).
\end{itemize}

Next, we use $\SURF_i^{\Left}$ to represent the marked ribbon surface $\SURF_i$ with left added marked points,
and use $\SURF_i^{\Right}$ to represent the marked ribbon surface $\SURF_i$ with right added marked points.
Let
\begin{align*}
 & D_1 := \{\tvarsig\in\Delta: \tvarsig\subset \SURF_i^{\Left}  \text{~for some~} i\}, \\
 & D_2 := \{\tvarsig\in\Delta: \tvarsig\subset \SURF_i^{\Right} \text{~for some~} i\}.
\end{align*}
Then $\Delta = D_1\cup D_2$, i.e., \ref{prop item: good cut 2} holds,
and by Theorem \ref{thm:KS2022}, we obtain a semi-orthogonal decomposition
\begin{align}\label{eq in prop:good cut}
  \per(A) = \langle \langle\X(D_1)\rangle_{\per(A)}, \langle\X(D_2)\rangle_{\per(A)}\rangle
\end{align}
of $\per(A)$ which is induced by $\Omega$. Thus, using Theorem \ref{thm:KS2022} again,
we have $\calF=\langle\X(D_1)\rangle_{\per(A)}$ and $\calG=\langle\X(D_2)\rangle_{\per(A)}$,
and so \ref{prop item: good cut 3} holds.
\end{proof}

Now the following result is the first main result of our paper.

\begin{theorem} \label{thm:main 260724}
Let $A$ be a gentle algebra and $\SURF$ be its marked ribbon surface.
Then the following statements are equivalent:
\begin{enumerate}[label={\rm(\arabic*)}]
  \item $\per(A)$ has a semi-orthogonal decomposition
    {\rm(}equivalently, $\Dcat^b(A)$ has a semi-orthogonal decomposition{\rm)}; \label{thm:main 260724 1}
  \item the marked ribbon surface $\SURF$ has a $\gbullet$-FFAS $\Delta$, and $\Delta$ has
    a disjoint union decomposition satisfying \ref{OD1} and \ref{OD2}; \label{thm:main 260724 2}
  \item the marked ribbon surface $\SURF$ has a good cut. \label{thm:main 260724 3}
\end{enumerate}
\end{theorem}

\begin{proof}
\checks{The trivial case, i.e., the case for good cut to be an empty cut, is clear, we ignore its proof. Now we prove the non-trivial case.}
Proposition \ref{prop:OD} admits that \ref{thm:main 260724 2} yields \ref{thm:main 260724 1}.
Theorem \ref{thm:KS2022} and Proposition \ref{prop:good cut} admit that \ref{thm:main 260724 1} yields \ref{thm:main 260724 2}.
The equivalence of \ref{thm:main 260724 1} and \ref{thm:main 260724 3} holds by Theorem \ref{thm:KS2022}.
\end{proof}

\begin{example} \label{examp:gentle260802-od} \rm
Consider the gentle algebra $A$ given in Example \ref{examp:gentle260802}.
The marked surface, write as $\SURF$, of it is shown in Example \ref{examp:gentle260802-surf}.
Now, we consider another $\gbullet$-FFAS $\Delta=\{c_1,c_2,c_3,c_4,c_5\}$ of $\SURF$,
see \Pic \ref{fig:gentle260802-FFAS}.
\begin{figure}[H]
  \centering
\begin{tikzpicture}
\draw[blue!25]
  (-1.73,-1.00) -- (-1.73, 1.00) --
  ( 0.00, 2.00) -- ( 0.00,-2.00) --
  ( 1.73, 1.00) -- ( 1.73,-1.00);
\draw[red][line width=1pt] (-1.00,-1.73) -- (-2.00, 0.00);
\draw[red][line width=1pt] (-1.00,-1.73) -- (-1.00, 1.73);
\draw[red][line width=1pt] (-1.00,-1.73) -- ( 1.00, 1.73);
\draw[red][line width=1pt] ( 1.00, 1.73) -- ( 1.00,-1.73);
\draw[red][line width=1pt] ( 1.00,-1.73) -- ( 2.00, 0.00);
\draw[line width=1pt] (0,0) circle(2cm);
\foreach \x in {0,60,120,180,240,300}
\fill[blue][rotate= \x] (0,2) circle(1mm);
\foreach \x in {0,60,120,180,240,300}
\fill[white][rotate= \x] (2,0) circle(1mm);
\foreach \x in {0,60,120,180,240,300}
\draw[red][line width=0.7pt][rotate= \x] (2,0) circle(1mm);
\draw[blue][line width=0.65pt]
  (-1.73, 1.00) -- ( 0.00,-2.00) -- (-1.73,-1.00);
\draw[blue][line width=0.65pt]
  (-1.73, 1.00) -- ( 1.73, 1.00) -- ( 0.00, 2.00);
\draw[blue][line width=0.65pt]
  (-1.73, 1.00) -- ( 1.73,-1.00);
\draw[blue] (-0.30, 0.15) node[above]{$c_3$};
\draw[blue] (-0.61,-1.77) node[above]{$c_1$};
\draw[blue] (-0.28,-1.45) node[above]{$c_2$};
\draw[blue] (-0.30, 1.00) node[above]{$c_4$};
\draw[blue] ( 0.50, 1.75) node[below]{$c_5$};
\draw[blue] (-1.73, 1.00) node[ left]{$p_1$};
\draw[blue] ( 0.00,-2.00) node[below]{$p_2$};
\draw[blue] ( 1.73, 1.00) node[right]{$p_3$};
\end{tikzpicture}
  \caption{A $\gbullet$-FFAS $\Delta$ of the marked ribbon surface given in Example \ref{examp:gentle260802-surf}}
  \label{fig:gentle260802-FFAS}
\end{figure}
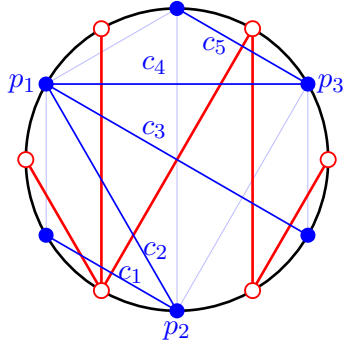
Let $\calC=\{c_2,c_3,c_5\}$ and $\calD=\{c_1,c_4\}$.
Then $c_2$ is right to $c_1$ at the marked point $p_1$,
$c_3$ and $c_2$ are right to $c_4$ at the marked point $p_2$,
and $c_5$ is right to $c_4$ at the marked point $p_3$.
Then, by Theorem \ref{thm:main 260724}, $\per(A)$ has a semi-orthogonal decomposition.
By Proposition \ref{prop:good cut}, this semi-orthogonal decomposition is
\[
\langle
  \langle
    \add(\X(\tc_2)\oplus\X(\tc_3)\oplus\X(\tc_5))
  \rangle_{\per(A)},
  \langle
    \add(\X(\tc_1)\oplus\X(\tc_4))
  \rangle_{\per(A)}
\rangle,
\]
where:
\begin{itemize}
  \item $\X(\tc_1)$ is the complex $P(3)\to P(2) \to P(1)$ which is quasi-isomorphic to $S(1)[0]$;
  \item $\X(\tc_2)$ is the complex $P(3) \to P(2)$ which is quasi-isomorphic to $S(2)[0]$;
  \item $\X(\tc_3)$ is the complex $P(3)\oplus P(5) \To{\left(\bsm a_3 & a_4 \\ a_2 & 0 \esm\right)} P(4)\oplus P(2)$ which is quasi-isomorphic to $({^2}{_3}{^4})[0]$;
  \item $\X(\tc_4)$ is the complex $P(3) \To{\left(\bsm a_2\\ a_3 \esm\right)} P(2)\oplus P(4)$ which is quasi-isomorphic to $({^2}{_3}{^4}{_5})[0]$;
  \item $\X(\tc_5)$ is the complex $P(3) \to P(4)$ which is quasi-isomorphic to $(^4_5)[0]$.
\end{itemize}
\end{example}

\begin{corollary} \label{coro:SOD-components-gentle}
Let $A$ be a gentle algebra and suppose that $\per(A)$ has a non-trivial semi-orthogonal decomposition $\langle \calX,\calY\rangle$.
\begin{enumerate}[label={\rm(\arabic*)}]
  \item Then there exist gentle algebras $B_{\calX}$ and $B_{\calY}$, which are not necessarily connected,
and triangle equivalences $\calX\simeq \per(B_{\calX})$ and $\calY\simeq \per(B_{\calY})$.
  \item If, in addition, $A$ has a finite global dimension, then $B_{\calX}$ and $B_{\calY}$ have finite global dimension.
Consequently, $\calX\simeq \Dcat^b(B_{\calX})$ and $\calY\simeq \Dcat^b(B_{\calY})$.
\end{enumerate}
\end{corollary}

\begin{proof}
Let $\SURF$ be the marked ribbon surface of $A$.
By Theorem \ref{thm:KS2022}, the semi-orthogonal decomposition $\langle\calX,\calY\rangle$ corresponds to a good cut $\Omega$ of
$\SURF$.
Write the cut surface as a disjoint union of its connected components:
\[\SURF_{\Omega} = \bigsqcup_{i\in I}\SURF_i.\]
By condition \ref{GC3} in Definition \ref{def:good cut}, no connected component of $\SURF_{\Omega}$ contains both a left added marked point and a right added marked point. Hence there is a disjoint union
decomposition $I=I_{\calX}\sqcup I_{\calY}$, where the components indexed by $I_{\calX}$ and $I_{\calY}$ correspond, respectively, to the two terms $\calX$ and $\calY$ of the given semi-orthogonal decomposition.

For every $i\in I$, choose a $\gbullet$-FFAS $\Delta_i$ of $\SURF_i$ and let $B_i:=A(\SURF_i)$ be the gentle algebra associated with the resulting marked ribbon surface by Construction \ref{construction}.
Let \[ B_{\calX}:=\prod_{i\in I_{\calX}}B_i \quad\text{and}\quad B_{\calY}:=\prod_{i\in I_{\calY}}B_i.\]
A finite product of gentle algebras is again a gentle algebra if
disconnected bound quivers are allowed. Moreover,
\[
   \per(B_{\calX})
      \simeq\bigoplus_{i\in I_{\calX}}\per(B_i),
   \qquad
   \per(B_{\calY})
      \simeq\bigoplus_{i\in I_{\calY}}\per(B_i).
\]

We prove the assertion for $\calX$; the argument for $\calY$ is identical.
For each $i\in I_{\calX}$, the geometric model for $\per(B_i)$ identifies its indecomposable objects with graded curves in $\SURF_i$.
Regarding such a curve as a curve in $\SURF$ after gluing the cut surface back along $\Omega$ gives a triangle functor
\[ \Phi_i:\per(B_i)\longrightarrow\per(A). \]
This functor is fully faithful.
Indeed, the basis morphisms between objects represented by graded curves
are determined by their intersection points and by the local order of the curves at their common marked endpoints.
For two curves lying in the same component $\SURF_i$, these data are unchanged when $\SURF_i$ is glued back into $\SURF$.
The geometric description of compositions is local as well, and the smoothing of intersections, which describes mapping cones, is preserved under the gluing. Therefore, the correspondence on graded curves induces a fully faithful triangle functor $\Phi_i$.
If $i,j\in I_{\calX}$ and $i\ne j$, then the images of $\Phi_i$ and $\Phi_j$ are mutually orthogonal, i.e.,
\begin{align*}
  & \Hom_{\per(A)}(\Ima(\Phi_i), \Ima(\Phi_j)[n])=0 \text{~for all~} n\in\ZZ,\\
\text{and} \,
  & \Hom_{\per(A)}(\Ima(\Phi_j), \Ima(\Phi_i)[n])=0 \text{~for all~} n\in\ZZ.
\end{align*}
Indeed, the two sides of each cutting curve belong to components of different types: one contains left added marked points and the other contains right added marked points.
Hence two components indexed by $I_{\calX}$ are not glued directly to each other.
Moreover, by condition \ref{GC2}, different cutting curves have no common endpoints.
Consequently, curves contained in two distinct components indexed by $I_{\calX}$ acquire neither intersection points nor common marked endpoints after gluing.
The geometric description of morphisms therefore gives the two vanishing statements above.
It follows that the functors $\Phi_i$, $i\in I_{\calX}$, induce a fully faithful triangle functor
\[ \Phi_{\calX}: \per(B_{\calX})\longrightarrow\per(A). \]

Let $D_{\calX}:=\bigcup_{i\in I_{\calX}}\Delta_i$.
The essential image of $\Phi_{\calX}$ is the full triangulated subcategory generated by the objects represented by the curves in $D_{\calX}$, i.e., $\Ima(\Phi_{\calX}) = \langle \X(D_{\calX})\rangle_{\per(A)}$.
On the other hand, by Proposition \ref{prop:good cut}, the semi-orthogonal component associated with the components indexed by
$I_{\calX}$ is precisely $\calX= \langle \X(D_{\calX})\rangle_{\per(A)}$.
Therefore, $\Phi_{\calX}$ induces a triangle equivalence $\per(B_{\calX})\simeq\calX$.
Similarly, one obtains $\per(B_{\calY})\simeq\calY$.

Finally, suppose that $A$ has finite global dimension.
Then $\per(A)$ is homologically smooth.
Homological smoothness is preserved under the triangle equivalences constructed above.
Since $\calX$ and $\calY$ are admissible subcategories of $\per(A)$,
they are also homologically smooth.
The above equivalences imply that $B_{\calX}$ and $B_{\calY}$ are homologically smooth.
For a finite dimensional gentle algebra, homological smoothness is equivalent to finite global dimension.
Hence $\gldim B_{\calX}<\infty$ and $\gldim B_{\calY}<\infty$.
Consequently, $\per(B_{\calX})\simeq\Dcat^b(B_{\calX})$ and $\per(B_{\calY})\simeq\Dcat^b(B_{\calY})$, which completes the proof.
\end{proof}

\section{Derived composition for gentle algebras}
\label{sect:derivedcompos}

This subsection aims to show under what conditions the module category of a gentle algebra admits a derived decomposition. For this purpose, we consider the canonical embedding functor from the module category to the perfect derived category. Lemma \ref{lemm:FGamma} explains how this functor can be interpreted via the geometric model, and Proposition \ref{prop:FGamma} establishes its fully faithfulness.

\subsection{Filtrations and rotation functors}

Let $\calA$ be an Abelian category and $X$ be an object in $\calA$.
For a full subcategory $\langle \calS:=\add X \rangle_{\calA}$ of an Abelian category $\calA$,
any object $M\in\langle \calS \rangle_{\calA}$ has a filtration
\[ 0=M_0\subseteq M_1\subseteq\cdots\subseteq M_r=M \]
with $M_i/M_{i-1}\in\mathcal S$, $1\leq i\leq r)$, i.e,
let $\operatorname{Filt}(\calS)$ be the minimal full subcategory of $\calA$
containing all objects which have $\calS$-filtrations (equivalently,
$\operatorname{Filt}(\calS) :=\{M\in\calA:$ there exists $0=M_0\subseteq M_1 \subseteq \cdots
\subseteq M_r=M$ such that $M_i/M_{i-1}\in\calS$ holds for all $1\=< i \=< r$$\}$),
then we have a well-known fact as follows.

\begin{lemma} \label{lemm:filtration}
$\operatorname{Filt}(\calS) = \langle \calS \rangle_{\calA}.$
\end{lemma}

\begin{proof}
Indeed, on the one hand, we have that $\operatorname{Filt}(\calS)$ is closed under extension
since any short exact sequence
$0 \to X \To{\iota} Y \To{\pi} Z \to 0$ with $X$, $Z\in \operatorname{Filt}(\calS)$ induces a filtration
\[ 0 =X_0 \subseteq X_1 \subseteq \cdots \subseteq X_r = X = Y_0
  \subseteq Y_1 \subseteq \cdots \subseteq Y_s=Y \]
of $Y$, where $r,s\in\NN$, and $Y_i = \pi^{-1}(Z_i)$ ($1\=< i\=< s$) is given by the filtration
$0=Z_0\subseteq Z_1 \subseteq \cdots \subseteq Z_s=Z$ of $Z$
(note that we have $Y_i/Y_{i-1} \cong Z_i/Z_{i-1} \in \calS$ by isomorphism theorem).
Thus, $\langle \calS \rangle_{\calA} \subseteq \operatorname{Filt}(\calS)$.

On the other hand, each $M\in\operatorname{Filt}(\calS)$ has an $\calS$-filtration
\[0=M_0\subseteq M_1\subseteq\cdots\subseteq M_r=M,
\quad M_i/M_{i-1}\in\calS.\]
When $r=1$, we have $M=M_1/M_0\in\calS
\subseteq\langle\calS\rangle_{\calA}$.
If $r>1$, we assume $M_{r-1}\in\langle\calS\rangle_{\calA}$.
Since $M/M_{r-1}\in\calS\subseteq\langle\calS\rangle_{\calA}$,
the following short exact sequence
\[0\to  M_{r-1}\to M \to M/M_{r-1}\to 0 \]
yields that $M \in \langle \calS\rangle_{\calA}$ by induction, and then we have $\operatorname{Filt}(\calS) \subseteq \langle \calS \rangle_{\calA}$.
\end{proof}

\begin{lemma}\label{lemm:emb curve}
Let $A$ be a gentle algebra and $\SURF$ be its marked ribbon surface.
For a permissible curve $c\in\PC{\SURF}$ and an admissible curve $\tvarsig\in\AC(\SURF)$,
if $\tc^{\rota} \simeq \tvarsig$ holds for some grading $\tc^{\rota}$ of $c^{\rota}$,
then $\MM(c) \mapsto \shift{n}{\X(\tvarsig)}$ admits a functor
\[ F_c : \add(\MM(c)) \to \langle\X(\tvarsig)\rangle_{\Dcat^b(A)} = \langle\X(\tc^{\rota})\rangle_{\Dcat^b(A)},\]
where $\shift{n}{\X(\tvarsig)}$ is quasi-isomorphic to $\MM(c)$ for some $n\in\ZZ$,
and $\langle\X(\tvarsig)\rangle_{\Dcat^b(A)}$ is the minimal full triangulated subcategory of $\Dcat^b(A)$ containing $\X(\tvarsig)$.
Furthermore, if the global dimension of $A$ is finite, then $F_c$ is of the form
\[ F_c : \add(\MM(c)) \to \langle\X(\tc^{\rota})\rangle_{\per(A)}.\]
\end{lemma}

\begin{proof}
Since $\MM(c)$ is indecomposable, we obtain that $\add(\MM(c)) = \{\MM(c)^{\oplus\upsilon} : \upsilon\in\NN\}$,
and each morphism in $\add(\MM(c))$ is of the form
$(f_{ij})_{\upsilon_2\times \upsilon_1}: \MM(c)^{\oplus \upsilon_1} \to \MM(c)^{\oplus \upsilon_2}$,
where every $f_{ij}$ are endomorphism lying in $\End_{A}(\MM(c))$
($A$ is the gentle algebra of the marked ribbon surface of $\SURF$).
By Theorem \ref{thm:Chang} and Proposition \ref{prop:rota} (or by Remark \ref{rmk:rota}),
for any $\imath\in\{1,2\}$, there exists an integer $n_{\imath} \in\ZZ$ such that
the projective resolution of $\MM(c)^{\oplus \upsilon_{\imath}}$
is $\shift{n_{\imath}}{\X(\tc^{\rota})}^{\oplus \upsilon_{\imath}}$,
then the homomorphism $(f_{ij})_{\upsilon_2\times \upsilon_1}$ of modules induces
a homomorphism of complexes from $\shift{n_1}{\X(\tc^{\rota})}^{\oplus \upsilon_1}$
to $\shift{n_2}{\X(\tc^{\rota})}^{\oplus \upsilon_2}$.
It admits a functor, i.e., $F_c$, from $\add(\MM(c))$ to $\langle\X(\tc^{\rota})\rangle_{\Dcat^b(A)}$ clearly.
In particular, if the global dimension of $A$ is finite, then we have $\per(A)\simeq K^b(\proj(A)) \simeq \Dcat^b(A)$.
Then we have $\langle\X(\tc^{\rota})\rangle_{\Dcat^b(A)}=\langle\X(\tc^{\rota})\rangle_{\per(A)}$ as required.
\end{proof}

Let $\Gamma$ be a set of some permissible curves. We call it is a \defines{generalized dissection} of $\SURF$
if arbitrary two permissible curves $c_1, c_2 \in\Gamma$ has no intersection lying in $\innerSurf$.
In this case, for any $c_1, c_2\in\Gamma$, we have
\[ \dim \big(\Hom_A(\MM(c_1),\tau\MM(c_2)) \oplus \Hom_A(\MM(c_2),\tau\MM(c_1))\big) = 0,  \]
and then, for the direct sum $T=\bigoplus_{c\in\Gamma}\MM(c)$, we have
\[ \Hom_A(T,\tau T)=0, \]
i.e., $T$ is $\tau$-rigid. See \cite{HZZ2023}. In particular, if $\Gamma$ is a maximal generalized dissection,
then $T$ is a support $\tau$-tilting module.
An \defines{admissible dissection} $\Sigma$ of $\Surf$ is a set of some admissible curves such that
any two admissible curves $\tvarsig_1, \tvarsig_2\in\Sigma$ has no intersection lying in $\innerSurf$.
For any generalized dissection $\Gamma = \{c_i\in\PC(\SURF): i\in I\}$ ($I$ is a finite index set),
we have that $\Gamma^{\rota}:=\{\tc_i^{\rota}\in\AC(\SURF):c_i\in \Gamma\}$ is an admissible dissection.
Here, each $\tc^{\rota}$ corresponds to the complex $\X(\tc^{\rota})$ which is quasi-isomorphic to $\MM(c)$.
In particular, if $A$ is hereditary, and $\Gamma$ is maximal, then, by \cite[Proposition 3.16 and Theorem 4.11]{AMV2016},
the direct sum $\X(\Gamma^{\rota}) := \bigoplus_{i\in I} \X(\tc_i^{\rota})$ is a 2-term silting complex.

\begin{lemma} \label{lemm:FGamma}
Let $A$ be a gentle algebra and $\SURF$ be its marked ribbon surface.
Let $\Gamma$ be a collection of some permissible curves. Then there is a functor as follows
\[ F_{\Gamma}: \langle \add(\MM(\Gamma)) \rangle_{\modcat(A)} \to \langle \X(\Gamma^{\rota}) \rangle_{\Dcat^b(A)}\]
\checks{sending each string module $\MM(\gamma)$ corresponding to the permissible curve $\gamma$ to the projective complex $\X(\gamma^{\rota})[n_{\gamma}]$, where $n_{\gamma}$ is an integer such that $\X(\gamma^{\rota})[n_{\gamma}]$ corresponds to the projective resolution of $\MM(\gamma)$.}
\end{lemma}

\begin{proof}
For any $c\in\Gamma$, we have $\MM(c)\in \langle \add(\MM(\Gamma)) \rangle_{\modcat(A)}$,
then Lemma \ref{lemm:emb curve} shows that the projective complex
given by the admissible curve $\tc^{\rota}$ and some integer $n_c\in\ZZ$
is quasi-isomorphic to $\MM(c)$, and provides a correspondence $F_c: \MM(c)\mapsto \shift{n_c}{\X(\tc^{\rota})}$.

Take $c_1\ne c_2\in\Gamma$. For $i\in\{1,2\}$, suppose that the two elementary $\gbullet$-polygons, one containing the segment $(c_i)_{(0,1)}$ and the other containing the segment $(c_i)_{(m(c_i),m(c_i)+1)}$, are not $\infty$-elementary $\gbullet$-polygons.
\checks{Then, in this case, we have $\pdim \MM(c) < \infty$, see \cite{LGH2024}.}
If at least one of them is an $\infty$-elementary polygon, then $\pdim \MM(c) =\infty$, 
\checks{we consider these cases in Remark \ref{rmk:FGamma}.}
Let $N$ be an $A$-module which is an extension of $\MM(c_1)$ and $\MM(c_2)$,
then we obtain two cases as follows:
\begin{itemize}
  \item[1$^{\mathrm{st}}$] this extension is trivial, i.e., $N\cong \MM(c_1)\oplus\MM(c_2)$;
  \item[2$^{\mathrm{nd}}$] this extension is non-trivial, in this case, we have $c_1\cap c_2\cap\innerSurf\ne \emptyset$
    and $N$ corresponds to either a permissible curve $c_N$ or two permissible curve $c_{N,1}$ and $c_{N,2}$.
\end{itemize}

In the 1$^{\mathrm{st}}$ case, $F_{c_1}$ and $F_{c_2}$ induces a correspondence $N\mapsto F_{c_1}(\MM(c_1))\oplus F_{c_2}(\MM(c_2))$.

Now we consider the 2$^{\mathrm{nd}}$ case, and assume that $N$ corresponds to a permissible curve $c_N$
(we can consider the case for $N$ corresponding to two permissible curves $c_{N,1}$ and $c_{N,2}$ in a similar way).
Without loss of generality, we assume that $\MM(c_1)$ is isomorphic to a submodule of $N$,
then the positional relationship of $c_1$, $c_2$ and $c_3$ is shown in \Pic \ref{fig:extof perm curv}
(the number of edges of all starting and ending $\gbullet$-polygons of $c_1$, $c_2$ and $c_3$
are greater than or equal to $3$ in \Pic \ref{fig:extof perm curv}.
This proof only proves this situation, and the proof is similar for cases where 2-gons).

\begin{figure}[H]
  \centering
\begin{tikzpicture}
\fill[black!25] ( 7.00,-1.00) -- ( 6.00,-2.00) -- ( 6.25,-2.00) -- ( 7.00,-1.25);
\fill[black!25] (-5.00, 1.00) -- (-7.00, 0.50) -- (-7.00, 0.70) -- (-5.82, 1.00);
\fill[black!25] (-1.00,-2.00) -- (-0.00,-2.00) -- (-0.00,-2.20) -- (-1.00,-2.20);
\draw[ blue][line width=0.72pt] ( 0.00, 1.00) -- (-3.00,-2.00);
\draw[ blue][line width=0.72pt] ( 0.00, 1.00) -- ( 3.00,-2.00);
\draw[ blue][line width=0.72pt] ( 0.00, 1.00) -- (-4.00,-2.00);
\draw[ blue][line width=0.72pt] ( 0.00, 1.00) -- ( 4.00,-2.00);
\draw[ blue][line width=0.72pt] ( 0.00,-2.00) arc(140:35:0.65);
\draw[ blue][line width=0.72pt] ( 1.00,-2.00) arc(140:35:0.65)[dotted];
\draw[ blue][line width=0.72pt] ( 2.00,-2.00) arc(140:35:0.65);
\draw[ blue][line width=0.72pt] (-2.00,-2.00) arc(140:35:0.65);
\draw[ blue][line width=0.72pt] (-3.00,-2.00) arc(140:35:0.65)[dotted];
\draw[ blue][line width=0.72pt] (-2.00, 1.00) -- (-6.00,-2.00);
\draw[ blue][line width=0.72pt] ( 2.00, 1.00) -- ( 6.00,-2.00);
\draw[ blue][line width=0.72pt] (-2.00, 1.00) -- (-5.00, 1.00);
\draw[black][line width=0.72pt] (-5.00, 1.00) -- (-7.00, 0.50)[dotted];
\draw[ blue][line width=0.72pt] (-7.00, 0.50) -- (-7.00,-1.00);
\draw[ blue][line width=0.72pt] ( 2.00, 1.00) -- ( 5.00, 1.00);
\draw[ blue][line width=0.72pt] ( 5.00, 1.00) -- ( 7.00, 0.50)[dotted];
\draw[ blue][line width=0.72pt] ( 7.00, 0.50) -- ( 7.00,-1.00);
\draw[black][line width=0.72pt] (-1.00,-2.00) -- (-0.00,-2.00);
\draw[ blue][line width=0.72pt] (-7.00,-1.00) -- (-6.00,-2.00)[dotted];
\draw[black][line width=0.72pt] ( 7.00,-1.00) -- ( 6.00,-2.00)[dotted];
\fill[blue]
  ( 0.00, 1.00) circle(2pt) ( 0.00,-2.00) circle(2pt)
  (-1.00,-2.00) circle(2pt) ( 1.00,-2.00) circle(2pt)
  (-2.00,-2.00) circle(2pt) ( 2.00,-2.00) circle(2pt)
  (-3.00,-2.00) circle(2pt) ( 3.00,-2.00) circle(2pt)
  (-4.00,-2.00) circle(2pt) ( 4.00,-2.00) circle(2pt)
  (-7.00,-1.00) circle(2pt) ( 7.00,-1.00) circle(2pt)
  (-7.00, 0.50) circle(2pt) ( 7.00, 0.50) circle(2pt)
  (-5.00, 1.00) circle(2pt) ( 5.00, 1.00) circle(2pt)
  (-2.00, 1.00) circle(2pt) ( 2.00, 1.00) circle(2pt);
\fill[red] (-0.50,-2.00) circle(2pt); \fill[white] (-0.50,-2.00) circle(1.55pt);
\fill[red] (-6.00, 0.75) circle(2pt); \fill[white] (-6.00, 0.75) circle(1.55pt);
\fill[red] ( 6.50,-1.50) circle(2pt); \fill[white] ( 6.50,-1.50) circle(1.55pt);
\draw[orange][line width=1pt] (-2.00,-2.00) to[out=  50,in= 180] ( 3.00,-0.34)[->];
\draw[orange][line width=1pt] ( 3.00,-0.34) to[out=   0,in= 270] ( 5.00, 1.00);
\draw[orange][line width=1pt] ( 2.00,-2.00) to[out= 130,in=   0] (-3.00,-0.34)[->];
\draw[orange][line width=1pt] (-3.00,-0.34) to[out= 180,in= -90] (-5.00, 1.00);
\draw[orange][line width=1pt] (-5.00, 1.00) to[out= -70,in= 180] ( 0.00, 0.20) to[out=0,in=-110] ( 5.00, 1.00);
\draw[orange] ( 3.00,-0.34) node[below]{$c_1$};
\draw[orange] (-3.00,-0.34) node[below]{$c_2$};
\draw[orange] ( 0.00, 0.20) node[below]{$c_N$};
\end{tikzpicture}
  \caption{The extension of $\MM(c_1)$ and $\MM(c_2)$}
  \label{fig:extof perm curv}
\end{figure}

By Theorem \ref{thm:Chang} and Proposition \ref{prop:rota} (or by Remark \ref{rmk:rota}),
the positional relationship of $c_1^{\rota}$, $c_2^{\rota}$ and $c_3^{\rota}$ is shown in \Pic \ref{fig:extof admi curv},
and we have three isomorphisms in $\per(A)$
\begin{center}
  $\shift{n_1}{\X(\tc_1^{\rota})}\cong \MM(c_1)$, $\shift{n_2}{\X(\tc_2^{\rota})}\cong \MM(c_2)$,
  and $\shift{n_N}{\X(\tc_N^{\rota})}\cong \MM(c_N)=N$
\end{center}
for some $n_1$, $n_2$, $n_N\in\ZZ$. In this case, we define $F_{c_N}(N) = F_{c_N}(\MM(c_N)) \cong \shift{n_N}{\X(\tc_N^{\rota})}$.
\begin{figure}[H]
  \centering
\begin{tikzpicture}
\fill[black!25] ( 7.00,-1.00) -- ( 6.00,-2.00) -- ( 6.25,-2.00) -- ( 7.00,-1.25);
\fill[black!25] (-5.00, 1.00) -- (-7.00, 0.50) -- (-7.00, 0.70) -- (-5.82, 1.00);
\fill[black!25] (-1.00,-2.00) -- (-0.00,-2.00) -- (-0.00,-2.20) -- (-1.00,-2.20);
\draw[ blue] ( 0.00, 1.00) -- (-3.00,-2.00);
\draw[ blue] ( 0.00, 1.00) -- ( 3.00,-2.00);
\draw[ blue] ( 0.00, 1.00) -- (-4.00,-2.00);
\draw[ blue] ( 0.00, 1.00) -- ( 4.00,-2.00);
\draw[ blue] ( 0.00,-2.00) arc(140:35:0.65);
\draw[ blue] ( 1.00,-2.00) arc(140:35:0.65)[dotted];
\draw[ blue] ( 2.00,-2.00) arc(140:35:0.65);
\draw[ blue] (-2.00,-2.00) arc(140:35:0.65);
\draw[ blue] (-3.00,-2.00) arc(140:35:0.65)[dotted];
\draw[ blue] (-2.00, 1.00) -- (-6.00,-2.00);
\draw[ blue] ( 2.00, 1.00) -- ( 6.00,-2.00);
\draw[ blue] (-2.00, 1.00) -- (-5.00, 1.00);
\draw[ blue] (-7.00, 0.50) -- (-7.00,-1.00);
\draw[ blue] ( 2.00, 1.00) -- ( 5.00, 1.00);
\draw[ blue] ( 5.00, 1.00) -- ( 7.00, 0.50)[dotted];
\draw[ blue] ( 7.00, 0.50) -- ( 7.00,-1.00);
\draw[ blue] (-7.00,-1.00) -- (-6.00,-2.00)[dotted];
\draw[black][line width=0.72pt] (-1.00,-2.00) -- (-0.00,-2.00);
\draw[black][line width=0.72pt] (-5.00, 1.00) -- (-7.00, 0.50)[dotted];
\draw[black][line width=0.72pt] ( 7.00,-1.00) -- ( 6.00,-2.00)[dotted];
\draw[red][line width=0.72pt] (-0.50,-2.00) to[out=  90,in=   0] (-1.96,-0.80);
\draw[red][line width=0.72pt] (-0.50,-2.00) to[out= 110,in=  90] (-2.50,-1.80)[dotted];
\draw[red][line width=0.72pt] (-0.50,-2.00) to[out= 130,in=  55] (-1.50,-1.80);
\draw[red][line width=0.72pt] (-0.50,-2.00) to[out=  80,in= 180] ( 1.96,-0.80);
\draw[red][line width=0.72pt] (-0.50,-2.00) to[out=  65,in= 130] ( 2.50,-1.80);
\draw[red][line width=0.72pt] (-0.50,-2.00) to[out=  50,in= 130] ( 1.50,-1.80)[dotted];
\draw[red][line width=0.72pt] (-0.50,-2.00) to[out=  30,in= 160] ( 0.50,-1.80);
\draw[red][line width=0.72pt] (-6.00, 0.75) -- (-7.00,-0.20);
\draw[red][line width=0.72pt] (-6.00, 0.75) -- (-6.50,-1.50)[dotted];
\draw[red][line width=0.72pt] (-6.00, 0.75) -- (-4.00,-0.50);
\draw[red][line width=0.72pt] (-6.00, 0.75) -- (-3.50, 1.00);
\draw[red][line width=0.72pt] ( 6.50,-1.50) -- ( 7.00,-0.20);
\draw[red][line width=0.72pt] ( 6.50,-1.50) -- ( 6.00, 0.80)[dotted];
\draw[red][line width=0.72pt] ( 6.50,-1.50) -- ( 4.00,-0.50);
\draw[red][line width=0.72pt] ( 6.50,-1.50) -- ( 3.50, 1.00);
\fill[blue]
  ( 0.00, 1.00) circle(2pt) ( 0.00,-2.00) circle(2pt)
  (-1.00,-2.00) circle(2pt) ( 1.00,-2.00) circle(2pt)
  (-2.00,-2.00) circle(2pt) ( 2.00,-2.00) circle(2pt)
  (-3.00,-2.00) circle(2pt) ( 3.00,-2.00) circle(2pt)
  (-4.00,-2.00) circle(2pt) ( 4.00,-2.00) circle(2pt)
  (-7.00,-1.00) circle(2pt) ( 7.00,-1.00) circle(2pt)
  (-7.00, 0.50) circle(2pt) ( 7.00, 0.50) circle(2pt)
  (-5.00, 1.00) circle(2pt) ( 5.00, 1.00) circle(2pt)
  (-2.00, 1.00) circle(2pt) ( 2.00, 1.00) circle(2pt);
\fill[red] (-0.50,-2.00) circle(2pt); \fill[white] (-0.50,-2.00) circle(1.55pt);
\fill[red] (-6.00, 0.75) circle(2pt); \fill[white] (-6.00, 0.75) circle(1.55pt);
\fill[red] ( 6.50,-1.50) circle(2pt); \fill[white] ( 6.50,-1.50) circle(1.55pt);
\draw[violet][line width=1pt] ( 0.00,-2.00) to[out=  85,in= 180] ( 7.00,-1.00);
\draw[violet][line width=1pt] ( 0.00,-2.00) to[out=  90,in=   0] (-7.00, 0.50);
\draw[violet][line width=1pt] (-7.00, 0.50) to[out=   5,in= 170] ( 7.00,-1.00);
\draw[violet] ( 3.80,-0.53) node[below]{$c_1^{\rota}$};
\draw[violet] (-4.00, 0.61) node[below]{$c_2^{\rota}$};
\draw[violet] (-0.00, 0.38) node[below]{$c_N^{\rota}$};
\end{tikzpicture}
  \caption{The extension of $\X(\tc_1^{\rota})$ and $\MM(\tc_2^{\rota})$}
  \label{fig:extof admi curv}
\end{figure}

Next, we show that
\begin{center}
  $\shift{0}{M} ~\widetilde{\in}~ \langle\X(\Gamma^{\rota})\rangle_{\per(A)}$
\footnote{``$X~\widetilde{\in}~\calC$'' represent $X$ isomorphic to an object in the category $\calC$}
$\subseteq \langle\X(\Gamma^{\rota})\rangle_{\Dcat^b(A)}$
for every $M\in \langle \add(\MM(\Gamma))\rangle_{\modcat(A)}$.
\end{center}
The case for $M$ with a infinite projective dimension is shown in Remark \ref{rmk:FGamma}.
To do this, by Lemma \ref{lemm:filtration},
let $0 = M_0\subseteq M_1\subseteq\cdots\subseteq M_r = M$ be a finite filtration such that
$M_i/M_{i-1}\in\add(\MM(\Gamma))$ for every $1\=< i\=< r$. We argue by induction on $r$.
If $r=1$, then $M\in\add(\MM(\Gamma))$, in this case, we have
$\shift{0}{M}~\widetilde{\in}~ \langle\X(\Gamma^{\rota})\rangle_{\per(A)}$.
Suppose that $r>1$. The short exact sequence
$0 \to M_{r-1}
  \to M
  \to M/M_{r-1}
  \to 0$
induces a distinguished triangle
$\shift{0}{M_{r-1}}
  \to \shift{0}{M}
  \to \shift{0}{M/M_{r-1}}
  \to \shift{1}{M_{r-1}}$.
By the induction hypothesis,
$\shift{0}{M_{r-1}} ~\widetilde{\in}~ \langle\X(\Gamma^{\rota})\rangle_{\per(A)}$, while
$\shift{0}{M/M_{r-1}} ~\widetilde{\in}~ \langle\X(\Gamma^{\rota})\rangle_{\per(A)}$ by the case $r=1$.
Since $\langle\X(\Gamma^{\rota})\rangle_{\per(A)}$ is triangulated, we obtain that
$\shift{0}{M} ~\widetilde{\in}~ \langle\X(\Gamma^{\rota})\rangle_{\per(A)}$.

Finally, for any indecomposable module $M$ in $\langle \add(\MM(\Gamma)) \rangle_{\modcat(A)}$,
we have a functor $F_{\Gamma}$ induced by the following conditions.
\begin{itemize}
  \item[(1)] $F_{\Gamma}$ sends each module $M$ to the complex $(\pmb{P}_{\bullet}, d_{\bullet})\in\per(A)$
    where $(\pmb{P}_{\bullet}, d_{\bullet})$ is quasi-isomorphic to $\shift{0}{M}$;
  \item[(2)] For each $h\in \Hom_{\modcat(A)}(M_1,M_2)$,
    $F_{\Gamma}(h)$ is the homomorphism $F_{\Gamma}(h): \X(\tc_1^{\rota}) \to \X(\tc_2^{\rota})$ of complexes
    induced by the commutative diagram
\[
\xymatrix{
(\pmb{P}_{\bullet}, d_{\bullet}) =
 & 0 \ar[r]
 & P_r \ar[r]^{d_r}\ar[d]_{h_r}
 & \cdots \ar[r]
 & P_1\ar[r]^{d_1}\ar[d]_{h_1}
 & P_0\ar[r]^{d_0} \ar[d]_{h_0}
 & 0 \\
(\pmb{Q}_{\bullet}, f_{\bullet}) =
 & 0 \ar[r]
 & Q_r \ar[r]^{f_r}
 & \cdots \ar[r]
 & Q_1\ar[r]^{f_1}
 & Q_0\ar[r]^{f_0}
 & 0, \\
}
\]
    where $(\pmb{P}_{\bullet}, d_{\bullet})$ is the projective resolution of $M_1$,
    $(\pmb{Q}_{\bullet}, f_{\bullet})$ is the the projective resolution of $M_2$,
    and $(h_r)_{r\>=0}$, a chain mapping (up to chain homotopy) between two projective resolutions, is induced by $h$.
    Then we have $F_{\Gamma}(h) = (h_r)_{r\>=0}: (\pmb{P}_{\bullet}, d_{\bullet}) \to (\pmb{Q}_{\bullet}, f_{\bullet})$.
    \qedhere
\end{itemize}
\end{proof}

\begin{remark}\label{rmk:FGamma} \rm
In the proof of Lemma \ref{lemm:FGamma}, we assume that the first and last elementary $\gbullet$-polygons crossing by $c_i$ are not $\infty$-elementary polygons.
It follows that the projective dimension, $\pdim\MM(c_i)$, of $\MM(c_i)$ is finite.
We can use the proof methods for the other cases, i.e.,
when at least one of $\pdim\MM(c_1)$ and $\pdim\MM(c_2)$ is infinite.
\begin{enumerate}[label={\rm(\arabic*)}]
  \item If $\pdim\MM(c_1)=\infty$ and $\pdim\MM(c_2)<\infty$,
  then one of the first elementary $\gbullet$-polygon $\PP_{1*}$
  and the last elementary $\gbullet$-polygon $\PP_{1}^*$
  is an $\infty$-elementary polygon. Then three subcases will be discovered. \label{FGamma:1}
  \begin{enumerate}[label={\rm(\alph*)}]
    \item The elementary $\gbullet$-polygon $\PP_{1*}$ is an $\infty$-elementary polygon, whereas $\PP_{1}^*$ is not.
      Then the positional relationship of $c_1$, $c_2$ and $c_3$ is shown in \Pic \ref{fig:extof perm curv-infty 1.1},
      and the positional relationship of $c_1^{\rota}$, $c_2^{\rota}$ and $c_3^{\rota}$ is shown in \Pic \ref{fig:extof admi curv-infty 1.1}.
      \label{FGamma:1.1}
    \item The elementary $\gbullet$-polygon $\PP_1^*$ is an $\infty$-elementary polygon, whereas $\PP_{1*}$ is not.
      We ignore the figures of this case.
    \item The elementary $\gbullet$-polygons $\PP_{1*}$ and $\PP_1^*$ both are $\infty$-elementary polygons.
      We ignore the figures of this case.
  \end{enumerate}

  \item If $\pdim\MM(c_1)<\infty$ and $\pdim\MM(c_2)=\infty$,
    then one of the first elementary $\gbullet$-polygon $\PP_{2*}$
    and the last elementary $\gbullet$-polygon $\PP_{2}^*$
    is an $\infty$-elementary polygon.
    This is similar to \ref{FGamma:1}.
    We ignore the figures of this case.

  \item If $\pdim\MM(c_1)=\infty$ and $\pdim\MM(c_2)=\infty$,
    then one of the first elementary $\gbullet$-polygon $\PP_{1*}$
    and the last elementary $\gbullet$-polygon $\PP_{1}^*$
    is an $\infty$-elementary polygon;
    and one of the first elementary $\gbullet$-polygon $\PP_{2*}$
    and the last elementary $\gbullet$-polygon $\PP_{2}^*$
    is an $\infty$-elementary polygon.
    We ignore the figures of this case.
\end{enumerate}
\end{remark}

\begin{figure}[H]
  \centering
\begin{tikzpicture}
\fill[black!25] ( 7.00,-1.00) -- ( 6.00,-2.00) -- ( 6.25,-2.00) -- ( 7.00,-1.25);
\fill[black!25] (-5.00, 1.00) -- (-7.00, 0.50) -- (-7.00, 0.70) -- (-5.82, 1.00);
\fill[black!25] ( 0.00,-1.20) circle(0.23cm);
\draw[black][line width=1pt] ( 0.00,-1.20) circle(0.23cm);
\draw ( 0.55,-1.20) node{$\PP_{1,*}$};
\draw ( 6.25,-0.20) node{$\PP_{1}^*$};
\draw[ blue][line width=0.72pt] ( 0.00, 1.00) -- (-3.00,-2.00);
\draw[ blue][line width=0.72pt] ( 0.00, 1.00) -- ( 3.00,-2.00);
\draw[ blue][line width=0.72pt] ( 0.00, 1.00) -- (-4.00,-2.00);
\draw[ blue][line width=0.72pt] ( 0.00, 1.00) -- ( 4.00,-2.00);
\draw[ blue][line width=0.72pt] (-1.00,-2.00) arc(140:35:0.65);
\draw[ blue][line width=0.72pt] ( 0.00,-2.00) arc(140:35:0.65);
\draw[ blue][line width=0.72pt] ( 1.00,-2.00) arc(140:35:0.65)[dotted];
\draw[ blue][line width=0.72pt] ( 2.00,-2.00) arc(140:35:0.65);
\draw[ blue][line width=0.72pt] (-2.00,-2.00) arc(140:35:0.65);
\draw[ blue][line width=0.72pt] (-3.00,-2.00) arc(140:35:0.65)[dotted];
\draw[ blue][line width=0.72pt] (-2.00, 1.00) -- (-6.00,-2.00);
\draw[ blue][line width=0.72pt] ( 2.00, 1.00) -- ( 6.00,-2.00);
\draw[ blue][line width=0.72pt] (-2.00, 1.00) -- (-5.00, 1.00);
\draw[black][line width=0.72pt] (-5.00, 1.00) -- (-7.00, 0.50)[dotted];
\draw[ blue][line width=0.72pt] (-7.00, 0.50) -- (-7.00,-1.00);
\draw[ blue][line width=0.72pt] ( 2.00, 1.00) -- ( 5.00, 1.00);
\draw[ blue][line width=0.72pt] ( 5.00, 1.00) -- ( 7.00, 0.50)[dotted];
\draw[ blue][line width=0.72pt] ( 7.00, 0.50) -- ( 7.00,-1.00);
\draw[ blue][line width=0.72pt] (-7.00,-1.00) -- (-6.00,-2.00)[dotted];
\draw[black][line width=0.72pt] ( 7.00,-1.00) -- ( 6.00,-2.00)[dotted];
\fill[blue]
  ( 0.00, 1.00) circle(2pt) ( 0.00,-2.00) circle(2pt)
  (-1.00,-2.00) circle(2pt) ( 1.00,-2.00) circle(2pt)
  (-2.00,-2.00) circle(2pt) ( 2.00,-2.00) circle(2pt)
  (-3.00,-2.00) circle(2pt) ( 3.00,-2.00) circle(2pt)
  (-4.00,-2.00) circle(2pt) ( 4.00,-2.00) circle(2pt)
  (-7.00,-1.00) circle(2pt) ( 7.00,-1.00) circle(2pt)
  (-7.00, 0.50) circle(2pt) ( 7.00, 0.50) circle(2pt)
  (-5.00, 1.00) circle(2pt) ( 5.00, 1.00) circle(2pt)
  (-2.00, 1.00) circle(2pt) ( 2.00, 1.00) circle(2pt);
\fill[red] (-6.00, 0.75) circle(2pt); \fill[white] (-6.00, 0.75) circle(1.55pt);
\fill[red] ( 6.50,-1.50) circle(2pt); \fill[white] ( 6.50,-1.50) circle(1.55pt);
\draw[orange][line width=1pt] (-2.00,-2.00) to[out=  50,in= 180] ( 3.00,-0.34)[->];
\draw[orange][line width=1pt] ( 3.00,-0.34) to[out=   0,in= 270] ( 5.00, 1.00);
\draw[orange][line width=1pt] ( 2.00,-2.00) to[out= 130,in=   0] (-3.00,-0.34)[->];
\draw[orange][line width=1pt] (-3.00,-0.34) to[out= 180,in= -90] (-5.00, 1.00);
\draw[orange][line width=1pt] (-5.00, 1.00) to[out= -70,in= 180] ( 0.00, 0.20) to[out=0,in=-110] ( 5.00, 1.00);
\draw[orange] ( 3.00,-0.34) node[below]{$c_1$};
\draw[orange] (-3.00,-0.34) node[below]{$c_2$};
\draw[orange] ( 0.00, 0.20) node[below]{$c_N$};
\end{tikzpicture}
  \caption{The extension of $\MM(c_1)$ and $\MM(c_2)$: case \ref{FGamma:1.1}}
  \label{fig:extof perm curv-infty 1.1}
\end{figure}

\begin{figure}[H]
  \centering
\begin{tikzpicture}
\fill[black!25] ( 7.00,-1.00) -- ( 6.00,-2.00) -- ( 6.25,-2.00) -- ( 7.00,-1.25);
\fill[black!25] (-5.00, 1.00) -- (-7.00, 0.50) -- (-7.00, 0.70) -- (-5.82, 1.00);
\fill[black!25] ( 0.00,-1.00) circle(0.23cm);
\draw[black][line width=1pt] ( 0.00,-1.00) circle(0.23cm);
\draw[ blue] ( 0.00, 1.00) -- (-3.00,-2.00);
\draw[ blue] ( 0.00, 1.00) -- ( 3.00,-2.00);
\draw[ blue] ( 0.00, 1.00) -- (-4.00,-2.00);
\draw[ blue] ( 0.00, 1.00) -- ( 4.00,-2.00);
\draw[ blue] (-1.00,-2.00) arc(140:35:0.65);
\draw[ blue] ( 0.00,-2.00) arc(140:35:0.65);
\draw[ blue] ( 1.00,-2.00) arc(140:35:0.65)[dotted];
\draw[ blue] ( 2.00,-2.00) arc(140:35:0.65);
\draw[ blue] (-2.00,-2.00) arc(140:35:0.65);
\draw[ blue] (-3.00,-2.00) arc(140:35:0.65)[dotted];
\draw[ blue] (-2.00, 1.00) -- (-6.00,-2.00);
\draw[ blue] ( 2.00, 1.00) -- ( 6.00,-2.00);
\draw[ blue] (-2.00, 1.00) -- (-5.00, 1.00);
\draw[ blue] (-7.00, 0.50) -- (-7.00,-1.00);
\draw[ blue] ( 2.00, 1.00) -- ( 5.00, 1.00);
\draw[ blue] ( 5.00, 1.00) -- ( 7.00, 0.50)[dotted];
\draw[ blue] ( 7.00, 0.50) -- ( 7.00,-1.00);
\draw[ blue] (-7.00,-1.00) -- (-6.00,-2.00)[dotted];
\draw[black][line width=0.72pt] (-5.00, 1.00) -- (-7.00, 0.50)[dotted];
\draw[black][line width=0.72pt] ( 7.00,-1.00) -- ( 6.00,-2.00)[dotted];
\draw[red][line width=0.72pt] (-6.00, 0.75) -- (-7.00,-0.20);
\draw[red][line width=0.72pt] (-6.00, 0.75) -- (-6.50,-1.50)[dotted];
\draw[red][line width=0.72pt] (-6.00, 0.75) -- (-4.00,-0.50);
\draw[red][line width=0.72pt] (-6.00, 0.75) -- (-3.50, 1.00);
\draw[red][line width=0.72pt] ( 6.50,-1.50) -- ( 7.00,-0.20);
\draw[red][line width=0.72pt] ( 6.50,-1.50) -- ( 6.00, 0.80)[dotted];
\draw[red][line width=0.72pt] ( 6.50,-1.50) -- ( 4.00,-0.50);
\draw[red][line width=0.72pt] ( 6.50,-1.50) -- ( 3.50, 1.00);
\fill[blue]
  ( 0.00, 1.00) circle(2pt) ( 0.00,-2.00) circle(2pt)
  (-1.00,-2.00) circle(2pt) ( 1.00,-2.00) circle(2pt)
  (-2.00,-2.00) circle(2pt) ( 2.00,-2.00) circle(2pt)
  (-3.00,-2.00) circle(2pt) ( 3.00,-2.00) circle(2pt)
  (-4.00,-2.00) circle(2pt) ( 4.00,-2.00) circle(2pt)
  (-7.00,-1.00) circle(2pt) ( 7.00,-1.00) circle(2pt)
  (-7.00, 0.50) circle(2pt) ( 7.00, 0.50) circle(2pt)
  (-5.00, 1.00) circle(2pt) ( 5.00, 1.00) circle(2pt)
  (-2.00, 1.00) circle(2pt) ( 2.00, 1.00) circle(2pt);
\fill[red] (-6.00, 0.75) circle(2pt); \fill[white] (-6.00, 0.75) circle(1.55pt);
\fill[red] ( 6.50,-1.50) circle(2pt); \fill[white] ( 6.50,-1.50) circle(1.55pt);
\draw[violet][line width=1pt] ( 7.00,-1.00)
  to[out= 180,in=   0] ( 0.00,-1.80) arc( -90:-180: 0.80)
  to[out=  90,in= 180] ( 0.00,-0.40) arc(  90:-180: 0.60)
  to[out=  90,in= 180] ( 0.00,-0.55) arc(  90:-180: 0.45)
  to[out=  90,in= 180] ( 0.00,-0.65);
\draw[violet][line width=1pt] ( 0.00,-0.65) arc(  90:-180: 0.35) [dashed];
\draw[violet!50][line width=1pt] (-7.00, 0.50)
  to[out=   0,in= 180] (-0.00,-0.30) arc(  90:-180: 0.70)
  to[out=  90,in= 180] (-0.00,-0.50) arc(  90:-180: 0.50)
  to[out=  90,in= 180] (-0.00,-0.60);
\draw[violet!50][line width=1pt] (-0.00,-0.60) arc(  90:-180: 0.40) [dashed];
\draw[violet][line width=1pt] (-7.00, 0.50) to[out=   5,in= 170] ( 7.00,-1.00);
\draw[violet] ( 4.10,-1.23) node[below]{$c_1^{\rota}$};
\draw[violet] (-4.00, 0.31) node[below]{$c_2^{\rota}$};
\draw[violet] (-0.00, 0.38) node[below]{$c_N^{\rota}$};
\end{tikzpicture}
  \caption{The extension of $\X(\tc_1^{\rota})$ and $\MM(\tc_2^{\rota})$: case \ref{FGamma:1.1}}
  \label{fig:extof admi curv-infty 1.1}
\end{figure}

\begin{proposition} \label{prop:FGamma}
Keep the notation from Lemma \ref{lemm:FGamma}. \checks{If the global dimension $\gldim A$ of $A$ is finite,} then the functor $F_{\Gamma}$ is fully faithful.
\end{proposition}

\begin{proof}
By the construction of $F_{\Gamma}$ in Lemma \ref{lemm:FGamma},
for any indecomposable modules $M$, $N\in$ $\langle\add(\MM(\Gamma))\rangle_{\modcat(A)}$,
let $c_M$ be the permissible curve corresponding to $M$
and $c_N$ be the permissible curve corresponding to $N$,
then there is a quasi-isomorphism
$q_M:F_{\Gamma}(M)\xrightarrow{\sim} \shift{0}{M}$ such that $\shift{0}{M}\cong \X(\tc_M^{\rota})$
for some grading $\tc_M^{\rota}$ of $c_M^{\rota}$,
and there is a quasi-isomorphism
$q_N:F_{\Gamma}(N)\xrightarrow{\sim} \shift{0}{N}$ such that $\shift{0}{N}\cong \X(\tc_N^{\rota})$
for some grading $\tc_N^{\rota}$ of $c_N^{\rota}$.
For every homomorphism $f:M\to N$ in $\langle\add(\MM(\Gamma))\rangle_{\modcat(A)}$, we have
\begin{align}\label{eq:prop:FGamma 1}
 q_N\compos F_{\Gamma}(f) = \shift{0}{f}\compos q_M
\end{align}
Since $\langle\add(\MM(\Gamma))\rangle_{\modcat(A)}$ is a full subcategory of $\modcat(A)$, we have
\[\Hom_{\langle\add(\MM(\Gamma))\rangle_{\modcat(A)}}(M,N)=\Hom_A(M,N).\]
The canonical stalk-complex embedding $\modcat(A)\to\Dcat^b(A)$,
$M\longmapsto \shift{0}{M}$ is fully faithful. Hence there is a natural isomorphism
$\Hom_A(M,N) \xrightarrow{\cong} \Hom_{\Dcat^b(A)}(\shift{0}{M},\shift{0}{N})$.
If $A$ has finite global dimension, we have $\Dcat^b(A)\simeq\per(A)$.
Furthermore, $\langle\X(\Gamma^{\rota})\rangle_{\per(A)}$ is a full subcategory of $\per(A)$.
Therefore, using the isomorphisms $q_M$ and $q_N$, we obtain natural isomorphisms
\begin{align}
\Hom_{\langle\add(\MM(\Gamma))\rangle_{\modcat(A)}}(M,N)
&=
\Hom_A(M,N) \nonumber \\
&\xrightarrow{\cong}
\Hom_{\Dcat^b(A)}(\shift{0}{M},\shift{0}{N}) \nonumber \\
&\xrightarrow{\cong}
\Hom_{\per(A)}(\X(\tc_M^{\rota}), \X(\tc_N^{\rota})) \nonumber \\
&\xrightarrow{\cong}
\Hom_{\per(A)}(F_{\Gamma}(M),F_{\Gamma}(N)) \label{eq:prop:FGamma 2}\\
&=
\Hom_{\langle\X(\Gamma^{\rota})\rangle_{\per(A)}}
(F_{\Gamma}(M),F_{\Gamma}(N)), \nonumber
\end{align}
By \eqref{eq:prop:FGamma 1}, the above composite is precisely the map induced by $F_{\Gamma}$:
\[ \Hom_{\langle\add(\MM(\Gamma))\rangle_{\modcat(A)}}(M,N)
\to \Hom_{\langle\X(\Gamma^{\rota})\rangle_{\per(A)}} (F_{\Gamma}(M),F_{\Gamma}(N)). \]
It is therefore bijective. Hence $F_{\Gamma}$ is fully faithful.
\end{proof}

\subsection{Derived decompositions and full formal arc system}

The second main result of this paper is as follows. We will use it to study the derived decompositions for gentle algebras.

\begin{theorem}\label{thm:main 260725}
Let $A$ be a gentle algebra with finite global dimension and $\SURF$ be its marked ribbon surface.
If $\SURF$ has a $\gbullet$-FFAS $\Delta$ such that:
\begin{enumerate}[label={\rm(\arabic*)}]
  \item $\Delta$ has a disjoint union decomposition $\Delta=D_1\cup D_2$ satisfying \ref{OD1} and \ref{OD2};
    \label{thm:main 260725 1}
  \item $\Delta^{\antirota} \subseteq \PC(\SURF)$, i.e., any curve in $\Delta$, as an admissible curve whose grading is natural, has an inverse rotation lying in $\PC(\SURF)$,
    \label{thm:main 260725 2}
\end{enumerate}
then $\per(A)$ has a semi-orthogonal decomposition
\[
\langle
  \langle F_{\Delta^{\antirota}}(\add(\MM(D_1^{\antirota}))) \rangle_{\per(A)},
  \langle F_{\Delta^{\antirota}}(\add(\MM(D_2^{\antirota}))) \rangle_{\per(A)}
\rangle,
\]
where $F_{\Delta^{\antirota}}$ is the fully faithful functor
\begin{align*}
  F_{\Delta^{\antirota}}: \langle \add(\MM(\Delta^{\antirota}))\rangle_{\modcat(A)}
& \to \langle\X(\Delta)\rangle_{\per(A)} \\
\MM(\tvarsig^{\antirota})
& \mapsto \X((\tvarsig^{\antirota})^{\rota}) \cong \shift{n_{\tvarsig}}{\X(\tvarsig)}
\end{align*}
given in Proposition \ref{prop:FGamma},
$n_{\tvarsig}\in\ZZ$ is an integer such that $\MM(\tvarsig^{\antirota})[0] \cong \X(\tvarsig)[n_{\tvarsig}]$ holds in $\per(A)$.
\end{theorem}

\begin{proof}
By Proposition \ref{prop:OD} (or Theorem \ref{thm:main 260724}) and the condition \ref{thm:main 260725 1},
$\SURF$ has a good cut $\Omega$, which corresponds to the semi-orthogonal decomposition
$\langle \langle \X(D_1) \rangle_{\per(A)}, \langle \X(D_2) \rangle_{\per(A)} \rangle$ of $\per(A)$.
By \ref{thm:main 260725 2}, we obtain $D_1^{\antirota}\subseteq \PC(\SURF)$ and $D_2^{\antirota}\subseteq \PC(\SURF)$,
and then we get two full extension-closed subcategories $\langle\add(\MM(D_1^{\antirota}))\rangle_{\modcat(A)}$
and $\langle\add(\MM(D_2^{\antirota}))\rangle_{\modcat(A)}$ of $\modcat(A)$.
By Proposition \ref{prop:FGamma}, we get two fully faithful functors
\begin{align*}
 & F_{D_1^{\antirota}}: \langle\add(\MM(D_1^{\antirota}))\rangle_{\modcat(A)}
   \to \langle\X(D_1)\rangle_{\per(A)}, \\
 & F_{D_2^{\antirota}}: \langle\add(\MM(D_2^{\antirota}))\rangle_{\modcat(A)}
   \to \langle\X(D_2)\rangle_{\per(A)}.
\end{align*}
Since $F_{D_1^{\antirota}}$ sends each indecomposable module corresponding to permissible curve
$\tvarsig^{\antirota}\in D_1^{\antirota}$ to the indecomposable object
(induced by the projective resolution of $\M(\tvarsig^{\antirota})$)
corresponding to $(\tvarsig^{\antirota})^{\rota}\simeq \tvarsig\in D_1$,
we have
\begin{align*}
 \langle \Ima(F_{D_1^{\antirota}}) \rangle_{\per(A)} = \langle\X(D_1)\rangle_{\per(A)}.
\end{align*}
Similarly, we have
\begin{align*}
 \langle \Ima(F_{D_2^{\antirota}}) \rangle_{\per(A)} = \langle\X(D_2)\rangle_{\per(A)}.
\end{align*}
Notice that, by Proposition \ref{prop:good cut} (see the formula \eqref{eq in prop:good cut} in its proof),
we have a semi-orthogonal decomposition
$\langle\langle\X(D_1)\rangle_{\per(A)}, \langle\X(D_2)\rangle_{\per(A)}\rangle$
of $\per(A)$, and, for each $i\in\{1,2\}$,
$\langle \Ima(F_{D_i^{\antirota}})\rangle_{\per(A)} = $
$\langle F_{D_i^{\antirota}}(\langle\add(\MM(D_i^{\antirota}))\rangle_{\modcat(A)} ) \rangle_{\per(A)} = $
$\langle F_{D_i^{\antirota}}(\add(\MM(D_i^{\antirota})) ) \rangle_{\per(A)}$,
hence
\begin{align*}
  & \langle\langle\X(D_1)\rangle_{\per(A)}, \langle\X(D_2)\rangle_{\per(A)}\rangle \\
=~& \langle\langle\Ima(F_{D_1^{\antirota}}), \Ima(F_{D_2^{\antirota}})\rangle \\
=~& \langle
      \langle F_{D_1^{\antirota}}(\add(\MM(D_1^{\antirota}))) \rangle_{\per(A)},
      \langle F_{D_2^{\antirota}}(\add(\MM(D_2^{\antirota}))) \rangle_{\per(A)}
    \rangle\\
=~& \langle
      \langle F_{\Delta^{\antirota}}(\add(\MM(D_1^{\antirota}))) \rangle_{\per(A)},
      \langle F_{\Delta^{\antirota}}(\add(\MM(D_2^{\antirota}))) \rangle_{\per(A)}
    \rangle
\end{align*}
as required.
\end{proof}

\begin{example} \label{examp:gentle260802-dd} \rm
Keep the notation from Example \ref{examp:gentle260802-od}.
The inverse rotation $\Delta^{\antirota}$ is shown in \Pic \ref{fig:gentle260802-antirota}.
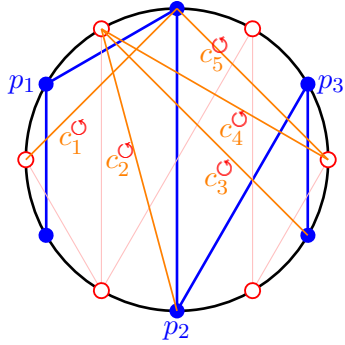
\begin{figure}[htbp]
  \centering
\begin{tikzpicture}
\draw[blue][line width=1pt]
  (-1.73,-1.00) -- (-1.73, 1.00) --
  ( 0.00, 2.00) -- ( 0.00,-2.00) --
  ( 1.73, 1.00) -- ( 1.73,-1.00);
\draw[red!25] (-1.00,-1.73) -- (-2.00, 0.00);
\draw[red!25] (-1.00,-1.73) -- (-1.00, 1.73);
\draw[red!25] (-1.00,-1.73) -- ( 1.00, 1.73);
\draw[red!25] ( 1.00, 1.73) -- ( 1.00,-1.73);
\draw[red!25] ( 1.00,-1.73) -- ( 2.00, 0.00);
\draw[line width=1pt] (0,0) circle(2cm);
\foreach \x in {0,60,120,180,240,300}
\fill[blue][rotate= \x] (0,2) circle(1mm);
\foreach \x in {0,60,120,180,240,300}
\fill[white][rotate= \x] (2,0) circle(1mm);
\foreach \x in {0,60,120,180,240,300}
\draw[red][line width=0.7pt][rotate= \x] (2,0) circle(1mm);
\draw[orange][line width=0.65pt] ( 0.00, 2.00) -- (-2.00, 0.00);
\draw[orange][line width=0.65pt] (-1.00, 1.73) -- ( 0.00,-2.00);
\draw[orange][line width=0.65pt] (-1.00, 1.73) -- ( 1.73,-1.00);
\draw[orange][line width=0.65pt] (-1.00, 1.73) -- ( 2.00, 0.00);
\draw[orange][line width=0.65pt] ( 0.00, 2.00) -- ( 2.00, 0.00);
\draw[orange] ( 0.56, 0.12) node[below]{$c_3^{\antirota}$};
\draw[orange] (-1.70, 0.30) node[right]{$c_1^{\antirota}$};
\draw[orange] (-0.41,-0.00) node[ left]{$c_2^{\antirota}$};
\draw[orange] ( 0.76, 0.77) node[below]{$c_4^{\antirota}$};
\draw[orange] ( 0.50, 1.75) node[below]{$c_5^{\antirota}$};
\draw[blue] (-1.73, 1.00) node[ left]{$p_1$};
\draw[blue] ( 0.00,-2.00) node[below]{$p_2$};
\draw[blue] ( 1.73, 1.00) node[right]{$p_3$};
\end{tikzpicture}
  \caption{The inverse rotation of $\Delta$ given in Example \ref{examp:gentle260802-od}}
  \label{fig:gentle260802-antirota}
\end{figure}
Consider the set $\calC^{\antirota}:=\{c_i^{\antirota}: i\in\{2,3,5\}\}$
and the set $\calD^{\antirota}:=\{c_j^{\antirota}: j\in\{1,4\}\}$.
Then $(\calC^{\antirota}, \calD^{\antirota})$ satisfy the conditions \ref{OD1} and \ref{OD2},
it admits a semi-orthogonal decomposition of $\per(A)$, see Example \ref{examp:gentle260802-od}.
Here, $F_{\Delta^{\antirota}}$ sends each $\MM(c_i^{\antirota})$ to $\X(\tc_i)$.
\end{example}

\begin{remark}\rm
For each $i\in\{1,2\}$, let $\emb_i:\langle\add(\MM(D_i^{\antirota}))\rangle_{\modcat(A)}\to\modcat(A)$
be the canonical exact embedding. Denote by
\[
s_i:\langle\add(\MM(D_i^{\antirota}))\rangle_{\modcat(A)}
\to \Dcat^b(\langle\add(\MM(D_i^{\antirota}))\rangle_{\modcat(A)}),
\qquad s_A:\modcat(A)\to\Dcat^b(A)
\]
the canonical stalk-complex functors. Then we have $\Dcat^b(\emb_i)\compos s_i = s_A\compos\emb_i$.
Let $F_{D_i^{\antirota}}$ be the functor given in Proposition \ref{prop:FGamma},
and let $\jmath_i: \langle\X(D_i)\rangle_{\Dcat^b(A)}\to \Dcat^b(A)$
be the canonical embedding. By the construction of $F_{D_i^{\antirota}}$,
for every $M\in \langle\add(\MM(D_i^{\antirota}))\rangle_{\modcat(A)} $ there is a natural quasi-isomorphism
\[ q_M:F_{D_i^{\antirota}}(M) \xrightarrow{\sim} \shift{0}{M}. \]
Consequently, there is a natural isomorphism of functors
\[ \jmath_i\compos F_i \cong s_A\compos\emb_i = \Dcat^b(\emb_i)\compos s_i. \]
Notice that this natural isomorphism does not by itself imply
that $\Dcat^b(\emb_i)$ is fully faithful. The latter requires
the additional comparison of extension groups appearing in
Corollary \ref{coro:main 260725}.
\end{remark}

\begin{corollary} \label{coro:main 260725}
Keep the notations from Theorem \ref{thm:main 260725}.
For each $i\in\{1,2\}$, let $\calC_i:=\langle\add(\MM(D_i^{\antirota}))\rangle_{\modcat(A)}$
and let $\emb_i:\calC_i\to\modcat(A)$ be an exact canonical embedding.
If:
\begin{enumerate}[label={\rm(\arabic*)}]
  \item $\calC_1$ and $\calC_2$ are Abelian subcategories;
  \item for all $X,Y\in\calC_i$ and all $n\>= 0$, the canonical map
    $\Ext^n_{\calC_i}(X,Y) \to \Ext^n_A(X,Y)$
  is an isomorphism,
\end{enumerate}
then $\modcat(A)$ has a derived composition
\[
[\calC_1,\calC_2]=[\langle \add(\MM(D_1^{\antirota})) \rangle_{\modcat(A)},
\langle \add(\MM(D_2^{\antirota})) \rangle_{\modcat(A)}].
\]
\end{corollary}


\begin{proof}
For each $i\in\{1,2\}$, the assumption on extension groups
implies that the exact embedding $\emb_i$ induces a fully faithful triangulated functor
\begin{align}\label{eq:coro:main 260725 1}
  \Dcat^b(\emb_i):\Dcat^b(\calC_i)\to \Dcat^b(A)
\end{align}
Next, we prove that
\begin{align}\label{eq:coro:main 260725 2}
\Ima(\Dcat^b(\emb_i))= \langle\X(D_i)\rangle_{\Dcat^b(A)}.
\end{align}
By Lemma \ref{lemm:FGamma}, for every $M\in\calC_i$,
the stalk complex $M[0]$ belongs to $\langle\X(D_i)\rangle_{\Dcat^b(A)}$.
Every bounded complex over $\calC_i$ can be obtained from its terms,
regarded as stalk complexes, by taking finitely many shifts and cones.
Therefore, we get
\begin{align}\label{eq:coro:main 260725 3}
  \Ima(\Dcat^b(\emb_i)) \subseteq \langle\X(D_i)\rangle_{\Dcat^b(A)}.
\end{align}
Conversely, for every $\tvarsig\in D_i$, the module $\MM(\tvarsig^{\antirota})$ belongs to $\calC_i$,
and, by Theorem \ref{thm:Chang}, we have that
$\Dcat^b(\emb_i) (\MM(\tvarsig^{\antirota})[0]) \cong \shift{n_{\tvarsig}}{\X(\tvarsig)}$
for some $n_{\tvarsig}\in\ZZ$. Thus, the essential image of $\Dcat^b(\emb_i)$ contains, up to shifts,
all the generators $\X(\tvarsig)$ with $\tvarsig\in D_i$.
Since the essential image of a fully faithful triangulated functor is a full triangulated subcategory,
we have
\begin{align}\label{eq:coro:main 260725 4}
  \langle\X(D_i)\rangle_{\Dcat^b(A)} \subseteq \Ima(\Dcat^b(\emb_i)).
\end{align}
Equations \eqref{eq:coro:main 260725 3} and \eqref{eq:coro:main 260725 4} prove \eqref{eq:coro:main 260725 2}.
Note that the global dimension of $A$ is finite, then we have $\Dcat^b(A)=\per(A)$.
By Theorem \ref{thm:main 260725}, we get that
$\langle \langle\X(D_1)\rangle_{\per(A)}, \langle\X(D_2)\rangle_{\per(A)} \rangle$
is a semi-orthogonal decomposition of $\per(A)$.
By \eqref{eq:coro:main 260725 2}, we obtain a semi-orthogonal decomposition
$\langle \Ima(\Dcat^b(\emb_1)), \Ima(\Dcat^b(\emb_2)) \rangle$ of $\Dcat^b(A)$.
Thus, by the definition of a derived decomposition,
$[\calC_1,\calC_2]$ is a derived decomposition of $\modcat(A)$.
\end{proof}

\paragraph{Competing Interests}
The authors declare that they have no conflicts of interest as defined by the journal, nor any other interests that could be perceived as influencing the results presented in this paper.

\paragraph{Authors' Contributions}
The order of authors is alphabetical, and all authors contributed equally to the conception, methodology, derivation, and writing of this paper.

\paragraph{Ethical Approval}
This article does not require ethical approval.

\paragraph{Acknowledgements}
We would like to express our gratitude to the Guizhou Provincial Key Laboratory of Applied Mathematics and Computing Power \& Algorithms for providing us with an academic discussion laboratory.





\def\cprime{$'$}

\end{document}